\documentclass[11pt]{amsart}
\usepackage{url, 
  amssymb,setspace, mathrsfs,fontenc, comment}
 \usepackage[alphabetic]{amsrefs}
\usepackage{fullpage} 
\usepackage{color}
\usepackage{tikz-cd}
\usepackage[all]{xy}
\usepackage{amsmath}

\usepackage{url, 
  amssymb,setspace, mathrsfs,fontenc}
\usepackage{amsrefs}
\usepackage{graphicx}
\usepackage[mathcal]{euscript}
\usepackage{verbatim}

\usepackage{hyperref}
\hypersetup{backref=true}

\newtheorem{thm}{Theorem}
\newtheorem{lem}[thm]{Lemma}
\newtheorem{prop}[thm]{Proposition}
\newtheorem{conj}[thm]{Conjecture}
\newtheorem{cor}[thm]{Corollary}
\newtheorem{defe}[thm]{Definition}

\theoremstyle{remark}
\newtheorem{rem}[thm]{Remark}
\newtheorem{exam}[thm]{Example}

\newtheorem{ntn}[thm]{Convention}

\DefineSimpleKey{bib}{myurl}
\newcommand\myurl[1]{\url{#1}}
\BibSpec{webpage}{
  +{}{\PrintAuthors} {author}
  +{,}{ \textit} {title}
  +{}{ \parenthesize} {date}
  +{,}{ \myurl} {myurl}
}
\usepackage{arydshln}

\newcommand{\nc}{\newcommand}

\nc{\ssec}{\subsection}

\nc{\on}{\operatorname}

\nc {\sE} {\mathscr{E}}
\nc {\sF} {\mathscr{F}}
\nc {\sL} {\mathscr{L}}
\nc {\sD}{\mathscr{D}}
\nc{\sA}{\mathscr{A}}

\nc {\cC} {\mathcal{C}}
\nc {\cG} {\mathcal{G}}
\nc {\cV}{\mathcal{V}}
\nc {\cK}{{k(\!(s)\!)}}

\nc {\cE} {\mathcal{E}}
\nc {\cO}{\mathcal{O}}
\nc {\cF}{\mathcal{F}}
\nc {\cZ}{\mathcal{Z}}
\nc {\cD}{\mathcal{D}}
\nc{\cDt}{\mathcal{D}^\times}
\nc {\cH}{\mathcal{H}}
\nc {\bZ}{\mathbb{Z}}
\nc {\bQ}{\mathbb{Q}}
\nc {\bQl}{\overline{\mathbb{Q}}_\ell}
\nc {\bQlt}{\bQl^\times} 
\nc {\FG}{\mathrm{FG}}
\nc {\dR}{\mathrm{dR}}

\nc {\uG} {\underline{G}}
\nc {\cB}{\mathcal{B}}
\nc {\cU}{\mathcal{U}}
\nc{\rat}{\mathrm{rat}}
\nc{\Hyp}{\mathscr{Hyp} }
      
\nc {\fF}{\mathfrak{F}}
\nc {\fB}{\mathfrak{B}}
\nc {\fx}{\mathfrak{x}}
\nc {\fy}{\mathfrak{y}}
\nc {\fb}{\mathfrak{b}}
\nc {\fk}{\mathfrak{k}}
\nc {\fI}{\mathfrak{i}}
\nc {\fJ}{\mathfrak{j}}
\nc {\fg} {\mathfrak{g}}
\nc {\fu} {\mathfrak{u}}
\nc {\fl} {\mathfrak{l}}
\nc {\fn} {\mathfrak{n}}
\nc {\cP} {\mathcal{P}}
\nc {\ft}{\mathfrak{t}}
\nc {\fz} {\mathfrak{z}}
\nc {\fc}{\mathfrak{c}}
\nc {\fh}{\mathfrak{h}}
\nc {\fp}{\mathfrak{p}}
\nc{\bone}{\mathbf{1}}
\nc{\tg} {\mathtt{g}}
\nc {\hfg} {\widehat{\fg}}
\nc {\hG} {\check{G}}

\nc {\bGm} {\mathbb{G}_m}
\nc {\bGa} {\mathbb{G}_a}
\nc {\bL}{\mathbf{L}}
\nc {\bK}{\mathbf{K}}
\nc {\bJ}{\mathbf{J}}
\nc {\bI}{\mathbf{I}}
\nc{\bC}{\mathbb{C}}
\nc{\bV}{\mathbb{V}}
\nc{\bP}{\mathbb{P}}
\nc{\bA}{\mathbb{A}}

\nc {\Q}{\mathrm{Q}}
\nc {\diag}{\mathrm{diag}}
\nc {\ev}{\mathrm{ev}}
\nc {\Res}{\mathrm{Res}}
\nc {\Fl}{\mathcal{F}\ell}
\nc {\Ad}{\mathrm{Ad}}
\nc {\pr}{\mathrm{pr}}
\nc{\Sl}{\mathfrak{sl}}
\nc {\gl}{\mathfrak{gl}}
\nc{\ra}{\rightarrow}
\nc {\quo}{\mathopen{ /\!/}}
\nc{\GL}{\mathrm{GL}}
\nc{\SL}{\mathrm{SL}}
\nc{\PGL}{\mathrm{PGL}}
\nc {\Bun}{\mathrm{Bun}}

\nc         {\rar}[1]       {\stackrel{#1}{\longrightarrow}}

\nc {\fa}{\mathfrak{a}}
\nc{\Hitch}{\mathrm{Hitch}}

\nc{\RS}{\mathrm{RS}}

\nc {\tp}{\mathfrak{p}}
\nc {\cA}{\mathcal{A}}
\nc {\cN}{\mathcal{N}}
\nc {\cW}{\mathcal{W}}

\nc{\opp}{\mathrm{opp}}
\nc {\Ind}{\mathrm{Ind}}
\nc {\sAn}{\mathrm{can}}
\nc {\Vac}{\mathrm{Vac}}
\nc {\Op}{\mathrm{Op}}
\nc {\Lg}{\check{\fg}}
\nc {\LV}{\check{V}}
\nc {\Lh}{\check{h}}
\nc {\LG}{\check{G}}
\nc {\Spec}{\mathrm{Spec}}
\nc {\End}{\mathrm{End}}

\nc {\rX}{\mathring{X}}

\nc {\sW}{\mathscr{W}}
\nc {\sH}{\mathscr{H}}
\nc {\sV}{\mathscr{V}}
\nc {\geom}{\mathrm{geom}}
\nc {\Irr}{\mathrm{Irr}}
\nc {\fm}{\mathfrak{m}}
\nc {\aff}{\mathrm{aff}}
\nc {\Aut}{\mathrm{Aut}}
\nc {\cJ}{\mathcal{J}}
\nc {\fs}{\mathfrak{s}}
\nc {\Stab}{\mathrm{Stab}}
\nc {\tw}{{\widetilde{w}}}
\nc {\gen}{\mathrm{gen}}
\nc {\genn}{\mathrm{genn}}
\nc {\sss}{\mathrm{ss}}
\nc {\spp}{\mathrm{sp}}
\nc {\Hom}{\mathrm{Hom}}
\nc {\bm}{\mathbf{m}}
\nc {\HG}{\mathcal{HG}}
\nc {\Gal}{\mathrm{Gal}}

\nc{\tP}{\mathtt{P}}
\nc{\tL}{\mathtt{L}}
\nc{\tU}{\mathtt{U}}

\nc{\tW}{\widetilde{W}}
\nc{\Hk}{\on{Hk}}
\nc {\cL}{\mathcal{L}}
\nc {\talpha}{\widetilde{\alpha}}
\nc {\tQ}{{\widetilde{Q}}}
\nc {\ochi}{\overline{\chi}}
\nc {\tdelta}{\widetilde{\Delta}}
\nc {\wt}{\mathrm{wt}}
\nc {\fQ}{\mathfrak{Q}}

\nc {\Rep}{\mathrm{Rep}}
\nc {\Hecke}{\mathrm{Hecke}}
\nc {\Gr}{\mathrm{Gr}}
\nc {\GR}{\mathrm{GR}}
\nc {\IC}{\mathrm{IC}}
\nc {\Std}{\mathrm{Std}} 
\nc {\Db}{\mathrm{D}^{\mathrm{b}}}
\nc {\tr}{\mathrm{tr}}
\nc {\ur}{\mathrm{ur}}

\begin{document} 
\renewcommand{\thepart}{\Roman{part}}

\renewcommand{\partname}{\hspace*{20mm} Part}
\address{School of Mathematics and Physics, The University of Queensland} 
\email{masoud@uq.edu.au}

\address{Department of Mathematics, California Institute of Technology} 
\email{lyi@caltech.edu}

\date{\today}

\begin{abstract} Generalised hypergeometric sheaves are rigid local systems on the punctured projective line with remarkable properties. Their study originated in the seminal work of Riemann on the Euler--Gauss hypergeometric function and has blossomed into an active field with connections to many areas of mathematics. In this paper, we construct the  Hecke eigensheaves whose eigenvalues are the irreducible hypergeometric local systems, thus confirming a central conjecture of the geometric Langlands program for hypergeometrics. The key new concept is the notion of hypergeometric automorphic data. We prove that this automorphic data is generically rigid (in the sense of Zhiwei Yun) and identify the resulting Hecke eigenvalue with hypergeometric sheaves. The definition of  hypergeometric automorphic data in the tame case involves the mirabolic subgroup, while in the wild case, semistable (but not necessarily stable) vectors coming from principal gradings intervene. \end{abstract} 

\title{Geometric Langlands for hypergeometric sheaves} 
\author{Masoud Kamgarpour and Lingfei Yi} 
\date{\today} 
\maketitle

\tableofcontents

\section{Introduction} 
 
\subsection{Overview} 
Let $k$ be a field, $X$ a smooth projective curve over $k$, and $G$ a reductive group over $k(X)$ with Langlands dual group $\LG$. The goal of the geometric Langlands program is to establish a duality between moduli of $\LG$-local systems on $X$ and moduli of $G$-bundles on $X$ \cite{LaumonLanglands, BD, Frenkel, GaitsgoryRecent, BenZviNadler}.  
The following is a core conjecture of the field: 

\begin{conj} \label{c:main}  Let $S\subset X$ be a finite set. 
For every irreducible 
$\LG$-local system $E$ on $X-S$, there exists a (non-zero) perverse sheaf $\sA=\sA_E$  on the moduli of $G$-bundles on $X$, equipped with appropriate level structure on $S$, such that $\sA$ is a Hecke eigensheaf with eigenvalue $E$. 
\end{conj}

By a local system, we mean either a lisse $\ell$-adic sheaf in  characteristic $p\neq \ell$ or a flat connection in characteristic zero, in which case perverse sheaves should be replaced with holonomic $\cD$-modules. For recollections on the notion of ramified Hecke eigensheaves, see \S \ref{s:eigensheaves}. 
Our main result is the following: 
   
\begin{thm}\label{t:mainTh} 
Conjecture \ref{c:main} is true for all irreducible (generalised) hypergeometric local systems. 
\end{thm} 

We refer the reader to \S \ref{s:main} for a precise formulation.
We prove this theorem by first constructing a moduli stack of bundles with appropriate level structures on $\bP^1$, and then proving that each connected component of this stack supports a unique irreducible  perverse sheaf satisfying a suitable equivariance condition. Thus, we obtain a ``rigid automorphic data'' in the sense of \cite{YunCDM}. The main theorem of \emph{op. cit.} then implies that this perverse sheaf is a Hecke eigensheaf. Finally, we use an explicit realisation of the relevant part of the Hecke stack to show that the resulting eigenvalue is isomorphic to a hypergeometric local system. 

In the rest of the introduction, we briefly review the notion of hypergeometric local systems and discuss previous progress on Conjecture \ref{c:main}. We then give an outline of the proof of Theorem \ref{t:mainTh} and indicate potential applications.

\subsection{Hypergeometric local systems} \label{ss:hypergeometrics} 
Hypergeometric \emph{functions} have a long and celebrated history, going back to the works of Wallis, Euler, Gauss, Kummer, and Riemann. They have found applications to numerous areas such as the theory of special functions \cite{AndrewsAskeyRoy}, Lie theory \cite{Koornwinder}, number theory \cite{ZagierDifferentialEquation, Wadim}, conformal field theory \cite{Varchenko}, integrable systems \cite{Opdam}, and mirror symmetry \cite{Horja, ZagierZinger, CortiGolyshev, LT}.

\subsubsection{Local systems} 
The geometry underpinning hypergeometric functions emerged from Riemann's study of the local system of solutions of the Euler--Gauss hypergeometric differential equation. Using the remarkable properties of this local system, Riemann was able to give a conceptual explanation for hypergeometric identities of Gauss and Kummer. Riemann's investigation was a stunning success largely because the hypergeometric local system is \emph{rigid}, i.e., determined up to isomorphism as soon as one knows the local monodromies at the singular points.

\subsubsection{Hypergeometric sheaves} 
In the modern era, the subject of hypergeometric (and more generally rigid) local systems was rejuvenated in works of Katz \cite{KatzKloosterman, KatzBook, KatzRigid}. In particular, Katz defined the $\ell$-adic realisation of hypergeometric local systems as lisse sheaves over finite fields. For brevity, we call these \emph{hypergeometric sheaves}. Their Frobenius trace functions, known as \emph{finite hypergeometric functions}, were discovered independently by J. Greene \cite{Greene}, and have found applications in describing motives of elliptic curves and other varieties \cite{Ono, BCM, DembeleWadim}.

\subsubsection{Tame vs. wild} 
(Generalised) hypergeometric local systems come in two variants: 
\begin{itemize} 
\item[(i)] The \emph{tame hypergeometric local systems} are local systems on $\bP^1-\{0,1,\infty\}$, with tame singularities at $\{0,1,\infty\}$, and pseudo-reflection monodromy at $1$.
\item[(ii)]  The \emph{wild hypergeometric local systems} are local systems on $\bP^1-\{0,\infty\}$, with a tame singularity at $0$ and a wild singularity at $\infty$. 
\end{itemize}

While the above terminology is standard in positive characteristics, in characteristic zero one usually uses ``regular singular'' and ``irregular'' instead of tame and wild, respectively.  We note that hypergeometric local systems in characteristic zero arise from differential equations governing the generalised hypergeometric functions ${}_p F_q$, cf. \cite{Slater}.

\subsubsection{Important examples of hypergeometric local systems} \mbox{}
\begin{itemize} 
\item[(i)] The celebrated Riemann local system governing the Euler--Gauss hypergeometric function ${}_2 F_1$ is a  regular singular  hypergeometric local system of rank $2$.
\item[(ii)] The Bessel local system, defined using the  Bessel differential equation, is an irregular hypergeometric local system of rank $n$ and smallest possible slope, namely, $\frac{1}{n}$; see \cite{XuZhu} for a recent treatment. 
\item[(iii)] Deligne's Kloosterman sheaf is a wild $\ell$-adic hypergeometric local system of rank $n$ and smallest possible Swan break $\frac{1}{n}$ \cite{Deligne, KatzKloosterman}. This is the $\ell$-adic analogue of the Bessel local system.
\item[(iv)] Quantum cohomology of $(2n-1)$-dimensional quadrics gives rise to a rank $2n$ irregular hypergeometric local system of slope $\frac{1}{2n-1}$, cf. \cite[\S 3.3]{GorbounovSmirnov}, \cite{PRW}. For instance, for the $3$-dimensional quadric, one obtains the connection 
\[
d+
\begin{pmatrix}
0 & 1 & 0 & 0 \\
0 & 0 & 1 & 0\\
t & 0 & 0 & 1\\
0 & t & 0 & 0\\
\end{pmatrix}\frac{dt}{t}. 
\]
\end{itemize}  
The last example, which was explained to us by Thomas Lam, illustrates that hypergeometric local systems not of Riemann or Bessel type are also important in applications. This was one of our main motivations for studying geometric Langlands duality for general hypergeometric local systems.

\subsection{Geometric Langlands}  
Geometric Langlands program has its origins in Deligne's elegant proof of unramified geometric class field theory; cf. \cite{Guignard} where Deligne's approach is generalised to the ramified case. 
In a groundbreaking paper \cite{Drinfeld83}, Drinfeld generalised  Deligne's approach to construct unramified automorphic forms for $\GL_2$. Subsequently, Laumon  gave a conjectural generalisation of Drinfeld's construction to $\GL_n$ and formulated the unramified version of Conjecture \ref{c:main}, thus shifting focus to automorphic \emph{sheaves}.  Conjecture \ref{c:main} for rank one local systems amounts to geometric class field they (cf. \cite{Serre, Guignard}). 
For unramified local systems of arbitrary rank, this conjecture was proved by Frenkel, Gaitsgory, and Vilonen over a finite field \cite{FGV} and by Gaitsgory over an arbitrary field \cite{Gaitsgory}. In contrast, as detailed below, there has been sporadic progress on Conjecture \ref{c:main} for \emph{ramified} local systems of rank bigger than one.

\subsubsection{Ramified eigensheaves in positive characteristic} 

\begin{itemize} 
\item[(i)] If $E$ is a rank $2$ \emph{tame unipotent} local system (i.e., a tamely ramified local system with unipotent monodromy), the corresponding Hecke eigensheaf was essentially constructed by Drinfeld  \cite{Drinfeld87}. 

\item[(ii)] If $E$ is a tame unipotent local system of rank $3$, the corresponding Hecke eigensheaf was constructed by Heinloth \cite{Heinloth}, who also gave a conjectural construction in higher ranks. 

\item[(iii)] 
If $E$ is a tame unipotent  local system on $\bP^1$ minus two points, then Conjecture \ref{c:main} follows from \cite{AB, Bezrukavnikov}, cf. \cite[Theorem 4.16]{BenZviNadler}. This approach works for general split reductive groups. 

\item[(iv)] If $E$ is a tame unipotent local system of rank $2$ on $\bP^1$ minus four points, then the corresponding Hecke eigensheaf is discussed in \cite{DeBos}. 
 
\item[(v)] If $E$ is a (wild) Kloosterman sheaf of arbitrary rank, the corresponding Hecke eigensheaf was constructed by Heinloth--Ngo--Yun \cite{HNY}. This approach works for all quasi-split reductive groups and was further generalised by \cite{YunEpipelagic} to epipelagic local systems. 
 
 \item[(vi)] In \cite{YunGalois}, Yun constructs Hecke eigensheaves on $\bP^1-\{0,1,\infty\}$ whose Hecke eigenvalues are tamely ramified at $0$, $1$, and $\infty$, and used these eigenvalues to give a positive answer to a question of Serre regarding existence of motives with exceptional monodromy. 
 
\item[(vii)] If $E$ is a (tame or wild) hypergeometric local system of rank $2$, the corresponding Hecke eigensheaf was constructed by Yun \cite{YunCDM}. 
\end{itemize} 

As far as we know, this summarises what is known about Conjecture \ref{c:main} in positive characteristic. In particular, we see that aside from the Kloosterman case, Conjecture \ref{c:main} for (tame and wild) hypergeometric sheaves of rank bigger than $2$ was not previously treated in the literature.

\subsubsection{Progress in characteristic zero} 
The above results also hold in characteristic zero; however, there has been more progress in this setting because $\cD$-modules are better understood than $\ell$-adic sheaves: 
\begin{itemize} 
\item[(i)] In \cite{BD}, Beilinson and Drinfeld constructed unramified Hecke eigensheaves for simple groups as quantum fibres of the Hitchin map. This approach relates Conjecture \ref{c:main} to representations of affine Kac--Moody algebras at the critical level, cf. \cite{FrenkelGaitsgory, FrenkelRamified, FFL}.  More recently, Zhu has shown that, under favourable conditions, the Beilinson--Drinfeld machinery can be adapted to the ramified setting \cite{Zhu}. In practice, there are two major obstacles for such adaptation:
\begin{enumerate} 
\item[(a)] It is is not known if the quantum fibres of ramified Hitchin maps are non-zero, the essential difficulty being our inability to determine the dimension of the ramified global nilpotent cone. (This is, however, resolved in the tame unipotent case, cf. \cite{Faltings, BKV}.)  
\item[(b)] Their machinery works with \emph{opers}, which is a local system equipped with additional data. While it is known that every local system has an oper structure \cite{ArinkinOper}, it is difficult to control the resulting singularities (this is the problem of ``apparent singularities", cf. the introduction of \cite{KatzRigid})
\end{enumerate} 

\item[(ii)] Beilinson and Drinfeld have a different conjectural construction of unramified Hecke eigensheaves using chiral Hecke algebra. As far as we understand, this is still conjectural because one does not know how to show that the resulting Hecke eigen $\cD$-modules are non-zero, cf. \cite[\S 20.5]{FrenkelBenZvi}. 

\item[(iii)]
Beilinson and Drinfeld also proposed a categorical version of Conjecture \ref{c:main} as, roughly speaking, an equivalence between the category of $\hG$ local systems on $X$ and the category of $\cD$-modules on $\Bun_G$. For progress on the categorical correspondence, see \cite{Arinkin, ArinkinFedorov} for $G=\GL_2$ on $\bP^1$ minus $4$ points, \cite{ArinkinGaitsgory, GaitsgoryRecent} for arbitrary $G$ and $X$ in the unramified setting, and \cite{NadlerYun} for $G=\GL_2$ on $\bP^1$ minus $3$ points, in the tame unipotent setting. 

\item[(iv)] Donagi and Pontev \cite{DP19} have constructed Hecke eigensheaves for tame unipotent rank $2$ local systems on $\bP^1$ minus five points using Langlands duality for the Hitchin system and non-abelian Hodge theory \cite{DP09, DPHitchin}. 

\item[(v)] Kapustin and Witten interpreted the unramified geometric Langlands duality in the context of four dimensional super Yang--Mills theories, relating Langlands duality to $S$-duality \cite{WittenKapustin}. This was generalised to the regular singular setting in \cite{WittenGukov} and the irregular setting in \cite{Witten}. 
\end{itemize} 

As far as we understand, Conjecture of \ref{c:main} for general hypergeometric connections cannot be proved using the aforementioned results.

\subsection{Outline of our approach}  
Henceforth, we let $G=\GL_n$ and $X=\bP^1$ over a finite field $k$. (We emphasise that no restrictions on the characteristic of $k$ is necessary.) The proof of Theorem \ref{t:mainTh} proceeds as follows:
\begin{enumerate} 
\item[(a)] To every rank $n$ hypergeometric local system $\sH$, we associate a group scheme $\cG$ over $X$ that is generically isomorphic to $G$. Following Thomas Lam, we think of $\cG$ as the group scheme ``controlling''  $\sH$. 
Let $\Bun_\cG$ denote the moduli stack of $\cG$-bundles on $X$. One can think of $\Bun_\cG$ as a moduli of rank $n$ vector bundles on $X$, equipped with certain level structures at the ramification points. The technical part of the paper concerns giving a combinatorial description of  $\Bun_\cG$. This leads to an explicit description of the \emph{generic} locus, i.e. the locus where bundles have minimal automorphisms. 

\item[(c)] The stack $\Bun_\cG$ will carry an action of a certain group scheme $H$, 
which is equipped with a one-dimensional character sheaf $\cC$. Thus, we may speak of $(H, \cC)$-equivariant perverse sheaves on $\Bun_\cG$. 
The correct definition of $\cG$ and $(H,\cC)$ is one of the main contributions of this paper and is explained further below. The most subtle part is how to define $\cG$ and $(H,\cC)$ at $x=1$ in the tame setting and at $x=\infty$ in the wild setting.
Our choices are motivated by the local Langlands correspondence and the local--global compatibility; see \S \ref{s:main} for details. In particular, in the wild setting, the construction depends on  the choice of a  functional coming from a principal grading of $\gl_n$. This further solidifies the link between the Vinberg theory and the Langlands program \cite{GrossReeder, HNY, RY, YunEpipelagic}, though as far as we understand, it is the first time that semistable functionals which are \emph{not} stable play a central role.\footnote{Konstantin Jakob informs us that he and Yun have also encountered this phenomena \cite{JY}.}   

\item[(d)] 
The main ``rigidity result'' of the paper states that (up to local systems coming from the base field $k$)  there exists a unique $(H, \cC)$-equivariant irreducible perverse sheaf $\sA$ on each connected component of $\Bun_\cG$.  This is proved by showing that the only orbits that are ``relevant", i.e. support $(H,\cC)$-equivariant sheaves, are the generic orbits (i.e., the orbits with minimal automorphisms). The proof of this fact in the tame setting is modelled on \cite[Theorem 2.4.2]{YunGalois}. The proof in the wild setting is significantly more complicated than \cite[Lemma 2.3]{HNY} (which treats the Kloosterman case), because the level structure at $\infty$ in our case is more intricate and the relevant orbit is not the trivial bundle.\footnote{This is because the bundle in $\Bun_{\cG}^0$ whose automorphism has minimal dimension is not trivial; see \S \ref{ss:genericWild}.}  One of the ingredients of the proof (which may be new and of independent interest) is a representation-theoretic characterisation of barycentres of facets in the standard apartment; see the appendix for details. 

\item[(e)] A key result of \cite{YunCDM} implies that $\sA$ is a ``weak Hecke eigensheaf''. Roughly speaking, this means that $\sA$ is a Hecke eigensheaf on each connected component of $\Bun_\cG$. 
We use an argument of \cite{HNY} to show that connected components are isomorphic; thus, we actually have a genuine Hecke eigensheaf.

\item[(f)] Finally, following the example of \cite[\S 3]{HNY}, we compute the relevant part of the Hecke correspondence explicitly to obtain an expression for the Hecke eigenvalue $\sE$. By comparing the Frobenius trace functions of $\sE$ and $\sH$, we conclude that $\sE$ is geometrically (i.e. over $\overline{k}$) isomorphic to $\sH$, thus establishing Theorem \ref{t:mainTh}. 

\end{enumerate} 

\subsubsection{Rigid automorphic data} 
In the main body of the text, we use the language of rigid automorphic data. The notion of rigidity for automorphic representations was developed by Z. Yun, building on earlier works of Heinloth, Ngo, and Yun \cite{HNY, YunGalois, YunCDM, YunEpipelagic}. One of Yun's main results is that whenever one has a rigid automorphic data satisfying appropriate properties, one obtains a Hecke eigensheaf. Furthermore, the resulting Hecke eigenvalue is expected to be rigid. Conversely, 
it is expected that whenever one has a rigid local system $E$, there exists a rigid automorphic data giving rise to a Hecke eigensheaf $\sA$ with eigenvalue $E$. We confirm this expectation for irreducible hypergeometric local systems of arbitrary rank.

\subsection{The group schemes $\cG$ and $H$} 
We now give a description of the group schemes $\cG$ and $H$, referring the reader to \S \ref{s:main} for more details.

\subsubsection{Tame setting}\label{sss:tameIntro}
Let $\sH$ be a tame rank $n$ hypergeometric local system on $\bP^1-\{0,1,\infty\}$ with pseudo-reflection monodromy at $x=1$. We define the group scheme $\cG$ controlling $\sH$ by  
\begin{equation}\label{eq:tameGroup} 
\cG(\cO_x) =
\begin{cases}
I^\opp(1) & x=0; \\ 
\mathtt{Q} & x=1; \\ 
I(1) & x=\infty; \\
G(\cO_x) & \textrm{otherwise}. 
\end{cases} 
\end{equation} 

Here $I$ is the Iwahori and $I(1)$ is its first Moy--Prasad subgroup (equivalently, the pro-unipotent radical). 
The group $\mathtt{Q}$ is defined to be the preimage of $Q$ under reduction map $G(\cO_1)\rightarrow G(k)$, where  $Q\subset G$ is the mirabolic subgroup, i.e., 
\[
Q:=
\begin{pmatrix}
\mathrm{GL_{n-1}} & *\\
0                 & 1
\end{pmatrix}.
\] 
The group $H$ is defined by 
\[
H:= I^\opp/I^\opp(1) \times I/I(1) \simeq T\times T\simeq \bGm^n\times \bGm^n 
\]
If $n=2$, the group schemes $\cG$ and $H$ coincide with the ones constructed in  \cite[\S 2.8.5]{YunCDM}.

\subsubsection{Wild setting}\label{sss:wildIntro}
Let $\sH$ be a rank $n$ hypergeometric local system on $\bP^1-\{0,\infty\}$ with a wild singularity at $\infty$ of highest break $\frac{1}{d}$ where $1\leq d\leq n$. Let $P$ be the parahoric group associated to $\check{\rho}/d$, where $\check{\rho}$ is the half-sum of positive coroots. Let $P(j)$ be the $j$th Moy--Prasad subgroup of $P$. We define the group scheme $\cG$ controlling $\sH$ by
\begin{equation} \label{eq:wildGroup} 
\cG(\cO_x) \simeq
\begin{cases}
I^\opp(1) & x=0; \\ 
P(2) & x=\infty; \\
G(\cO_x) & \textrm{otherwise}. 
\end{cases} 
\end{equation} 
The construction of the group $H$ depends on the choice of an appropriate semistable linear functional $\phi$ on $P(1)/P(2)$; see \eqref{eq:phisp} for the definition of this functional. Recall that the reductive quotient $P/P(1)$ acts on the vector space $P(1)/P(2)$ and therefore also on the dual space. Let $B_\phi$ be a Borel subgroup of the stabiliser of $\phi$ in $P/P(1)$.  Then 
\[
H:=(I^\opp/I^\opp(1)) \times (B_\phi \ltimes P(1)/P(2)).
\]
See \S\ref{ss:wild data} for more details. Note that if $d=n$ (so that $\sH$ is the Kloosterman sheaf), $P$ is the Iwahori, and the functional $\phi$ can be any ``generic functional'' in the sense of \cite[\S 1.3]{HNY}. In this case, the stabiliser of $\phi$ is the centre; thus, $H=I^\opp/I^\opp(1)\times I(1)/I(2)$, in agreement with \emph{op. cit.}. We note that this is the \emph{only} case when $\phi$ is stable. (Indeed, $n$ is the only elliptic number of $\gl_n$).

\subsection{Hecke realisations of hypergeometrics} Our main results can be interpreted as giving a new realisation of hypergeometric local systems, namely, a realisation via Hecke correspondences. We explain what this means in the tame and wild setting. 

\subsubsection{Tame setting}\label{sss:HeckeTame} The ``relevant part'' of the Hecke correspondence (see \eqref{eq:relevantHecke}) is given by restriction to $\bP^1-\{0,1,\infty\}$ of the diagram 
\begin{equation}\label{eq:HeckeTame}
\begin{tikzcd}
&  (\bGm-\{1\})^n \arrow{dl}[swap]{\pi_1} \arrow{dr}{\pi_2}  & \\
\bGm^{2n} &       &  \bP^1-\{0,\infty\},
\end{tikzcd} 
\end{equation}
where 
\[
\pi_2(y_1,...,y_n)=\prod_{i=1}^n(y_i-1)/y_i,
\]
and
\[
\pi_1(y_1,...,y_n)=((1-y_1)^{-1}, ..., (1-y_n)^{-1} , -y_1^{-1}, y_2^{-1}, ..., y_{n-1}^{-1}, -y_n^{-1} ).
\]
The Hecke eigenvalue is then given by 
\begin{equation}\label{eq:tameEigenvalue} 
\sE= \pi_{2,!} \pi_1^* \cL[n-1], 
\end{equation} 
where $\cL$ is a character sheaf defined using the hypergeometric data. As we shall see, $\sE$ is actually a tame hypergeometric local system.

\subsubsection{Wild setting} 
In the wild case, the relevant part of the Hecke correspondence (see \eqref{eq:relevantHecke} and Proposition \ref{p:wildHecke}) is given by the following diagram
\[
\begin{tikzcd}
&  \bGm^d\times(\bGm-\{1\})^{n-d} \arrow{dl} \arrow{dr}  & \\
\bGm^n\times\bGm^{n-d} \times \bGa&       &  \bP^1-\{0,\infty\}.
\end{tikzcd}
\]
The description of the maps here are more complicated and involve the combinatorics of the normalised Kac coordinate of ${\check{\rho}}/{d}$; see \S \ref{s:wild eigenvalue} for details. If $n=d$ and the map $\bGm^n\ra \bGm^n$ is the identity, we recover the Deligne correspondence defining the Kloosterman sheaf \cite{Deligne}.

\subsection{Potential applications} 

\subsubsection{Mirror symmetry} Our motivation for studying Langlands duality for hypergeometric local systems comes from mirror symmetry. 
Building on \cite{PRW}, Lam and Templier \cite{LT} noted that Bessel local systems are closely related to quantum connections of minuscule flag varieties and used geometric Langlands duality to settle mirror symmetry for these varieties. Subsequently, Thomas Lam noted that the quantum connection of smooth odd quadrics are closely related to certain hypergeometric local systems not of Bessel type. He asked what Langlands duality for these connections look like. The present paper is an attempt to answer his question. We hope that our solution has applications to mirror symmetry for smooth quadrics.

\subsubsection{Hypergeometric automorphic representations} \label{sss:hypergeomtricRep}
While our focus has been the \emph{geometric} Langlands program, this work also has implications for the \emph{classical} Langlands correspondence. Namely, recall that the global Langlands correspondence for functions fields, proved by L. Lafforgue \cite{Lafforgue}, establishes a bijection between irreducible rank $n$ local systems on $X-S$ and  irreducible cuspidal automorphic representations of $\GL_n$. We call a cuspidal automorphic representation \emph{hypergeometric} if it is mapped to a hypergeometric local system under the Langlands correspondence.
Lafforgue's bijection is not explicit; thus, a priori, it is not known what hypergeometric automorphic representations look like. Our work make these representation more explicit. 
Indeed, taking the Frobenius trace function of hypergeometric Hecke eigensheaves (which are explicitly constructed in this paper), one obtain the Hecke eigenform of hypergeometric automorphic representations. In \cite{Gross}, Gross gave a characterisation of the Kloosterman automorphic representation as the unique automorphic representation which is unramified on $\bP^1-\{0,\infty\}$, Steinberg at $0$, and simple supercuspidal at $\infty$. (A simple proof of Gross's result is given \cite[Lemma 2.1]{HNY}.) The analogous characterisation of hypergeometric automorphic representations is discussed in \S \ref{s:main}.

\subsubsection{Hypergeometric motives} 
Using the notion of convolution, Katz constructed hypergeometric motives over the ring of integers $R$ of appropriate number fields \cite[\S 8.17]{KatzBook}. These motives are lisse $\ell$-adic sheaves appearing in the cohomology of certain schemes over $R$. Specialising to a finite field (resp. complex numbers), one obtains the tame $\ell$-adic (resp. complex) hypergeometric local systems.\footnote{Wild hypergeometrics should also be motivic in the sense of \cite{FJ}.} It is tempting to conjecture that the expression \eqref{eq:tameEigenvalue} gives an alternative realisation of hypergeometric motives. We are currently unable to prove this conjecture because the key notion used in \cite{KatzBook}, namely Deligne's theorem on semi-continuity of the Swan conductor, applies only to schemes of relative dimension one (the map $p$ has relative dimension $n-1$). In another direction, Yun has used Langlands duality for Kloosterman sheaves to prove conjectures of Evans regarding moments of Kloosterman sums \cite{YunEvans}. We hope that our work has implications for understanding moments of general hypergeometric sums.

\subsubsection{Hypergeometric sheaves for reductive groups} 
One of the main achievements of \cite{HNY, YunGalois, YunCDM, YunEpipelagic} is that they constructed rigid automorphic data for general (quasi-split) reductive groups and used this to \emph{define} local systems with remarkable properties, e.g. having exceptional monodromy. Preliminary computations show that there is also a reductive analogue of general hypergeometric local systems. This is the subject of our work in progress.

\subsubsection{Beyond the $\ell$-adic setting} 
We state and prove our main results for $\ell$-adic hypergeometric local systems in positive characteristic. However, after appropriate modifications, our results also apply to hypergeometric local systems in characteristic zero. The $p$-adic companion (cf.\cite{Abe, Kedlaya}) of the hypergeometric sheaves are known as \emph{hypergeometric (overconvergent) F-isocrystals} and have been studied by, e.g., Miyatani \cite{Miyatani}. Thanks to the recent work of Xu and Zhu on the $p$-adic geometric Langlands program \cite{XuZhu}, our results can also be adapted to construct Hecke eigensheaves for hypergeometric $F$-isocrystals.

\subsection{Structure of the text} 
\subsubsection{}
Sections \ref{s:notation}--\ref{s:eigensheaves} are preliminaries and we recommend the reader skips them on the first go. In \S \ref{s:notation}, we set the notation (which is mostly standard). In \S \ref{s:hypergeometric}, we review Katz's theory of $\ell$-adic hypergeometric local systems. In \S \ref{s:integralModels}, we discuss the notion of an integral model, which is a convenient tool for discussing ramifications in the geometric Langlands program. In \S \ref{s:eigensheaves}, we explain what one means by a ramified Hecke eigensheaf. In \S \ref{s:rigid}, we give a simplified account of Yun's notion of rigid automorphic data.  

\subsubsection{} 
Sections \ref{s:main}--\ref{s:wild eigenvalue} contain the main results and their proofs. 
The main definitions and theorems appear in \S \ref{s:main}. Rigidity of hypergeometric automorphic data is proved in \S \ref{s:tameRigidity} (resp. \S \ref{s:wildRigidity}) for the tame (resp. wild) case. Determination of Hecke eigenvalue is achieved in \S \ref{s:tame eigenvalue} (resp. \S \ref{s:wild eigenvalue}) for the tame (resp. wild) case. 
Finally, we gather some results about inner principal gradings in the appendix.

\subsection{Acknowledgement} 
This paper fulfils a part of the vision of Zhiwei Yun for the role of rigidity in the geometric Langlands program. Our intellectual debt to the work of Yun and collaborators is obvious \cite{HNY, YunGalois, YunCDM, YunEpipelagic}. As noted above, this project was initiated in response to a question by Thomas Lam. We would like to thank him for raising this penetrating question and for many subsequent illuminating discussions. We would also like to thank Dima Arinkin, David Ben-Zvi, Jochen Heinloth, Konstantin Jakob, Paul Levy, Will Sawin, Ole Warnaar,  Daxin Xu, Zhiwei Yun, Xinwen Zhu, and Wadim Zudilin for helpful discussions. MK would like to especially thank Dan Sage for collaborations on rigid connections (\cite{KamgarpourSage0, KamgarpourSage}) and for teaching him about parahorics and Moy--Prasad subgroups. MK has been supported by two Australian Research Council Discovery Projects. LY has been supported by a CalTech Graduate Student Fellowship.


\section{Notation} \label{s:notation} 
In this section, we define the notation used in the rest of the article. Our notation is mostly standard, so the reader is encouraged to skip this section, referring to it when necessary.

\subsection{Geometric data} Let  $k$ be a finite field and  $\overline{k}$ an algebraic closure of $k$. Let  $\ell$ be a prime different from the characteristic of $k$. We frequently abuse notation and use the same letter to refer to a (ind-)scheme and its $k$-points. If $Y$ is an object over $k$, we sometimes write $\overline{Y}$ for $Y\otimes_k \overline{k} $. If $Y$ is an Artin stack over $k$, we let $D^b(Y)$ denote the derived category of $\bQl$-adic sheaves on $Y$ and $\mathrm{Perv}(Y)$ the subcategory of perverse sheaves.

\subsubsection{} 
Let $X=\bP^1$ be the projective line over $k$, $t$  a local coordinate at $0$, and $s=t^{-1}$ a local coordinate at $\infty$. Let $F=k(t)=k(s)$ denote the function field of $X$. For each $x\in X$,  let $\cO_x$ be the completed local ring at $x$ , $F_x$ its field of fractions, and $t_x$ a uniformiser. 
Let $\bA:=\bA_F=\otimes'_{x\in X} F_x$ denote the ring of adeles and $\cO_\bA:=\otimes_{x\in X} \cO_x$ its ring of integers. The set of closed points of $X$ is usually denoted by $|X|$; however, we abuse notation and denote this set by $X$ as well.

\subsection{Group-theoretic data} 
We note that many of our constructions work for general reductive groups; however, for simplicity, we assume throughout that  $G=\GL_n$. Thus, the Langlands dual  $\hG$ is also $\GL_n$. We prefer to use the notation $G$ and $\hG$ because  this reminds us which side of Langlands duality we are referring to.  Let $Z(G)$ and $Z(\hG)$ denote the centre of $G$ and $\hG$, respectively. Denote the Lie algebra of $G$ by $\fg=\mathfrak{gl}_n$.

\subsubsection{} 
Let $T$ and $B$ be the subgroups of diagonal and upper triangular matrices, respectively. Let $W\simeq S_n$ denote the Weyl group.
Let $X_*(T):=\Hom(\bGm, T)$ and $X^*(T):=\Hom(T,\bGm)$ denote the set of cocharacters and characters, respectively. Let $\cA:=X_*(T)\otimes \mathbb{R}$ denote the standard apartment and $\cA_\bQ:=X_*(T)\otimes \bQ$ the rational apartment. We  write $t=\diag(t_1,...,t_n)$ for elements of $T\simeq (k^\times)^n$ and define subgroups $T_j\subseteq T$ by 
\begin{equation}\label{eq:Tj}
T_j:=\{t=\mathrm{diag}(1,...,1, t_j, 1,...,1)\mid t_j \in k^{\times}\}. 
\end{equation}

\subsubsection{} 
Let $\Phi=\Phi(G)$ denote the root system of $G$. Then $\Phi=\{\alpha_{ij}\}$ where $i$ and $j$ run over all pairs of non-equal integers in $\{1,2,...,n\}$ and $\alpha_{ij}$ is the root whose root subspace is generated by $E_{ij}$. If $\alpha=\alpha_{ij}$, we sometimes write $E_\alpha$ for $E_{ij}$.
Let $E_{ij}^*$ denote the functional on $\fg$ defined by $E_{ij}^*(A):=a_{ij}$ for $A=(a_{ij})\in\fg$.

\subsubsection{}\label{sss:standardRoots}
Let $\check{\rho}$ denote the half sum of positive coroots. Thus, 
\[
\check{\rho} = \frac{1}{2}\diag(n-1, n-3, ..., -(n-3), -(n-1)).
\]
Let $\alpha_i:=\alpha_{i,i+1}$.
Then $\Delta=\Delta(G)=\{\alpha_1,...,\alpha_{n-1}\}$ is the set of standard simple roots for $G$. 
We write $\mathrm{ht}(\alpha)=\langle\alpha,\check{\rho}\rangle$ for the height of a root $\alpha$. The highest root is then $\theta:=\sum_{i=1}^{n-1}\alpha_i$. For each $\alpha \in \Phi$, we let $U_\alpha\subset G$ denote the corresponding one parameter unipotent subgroup. 
For a subgroup $H\subseteq G$, we write $\Phi(H)$ for the subset of roots $\alpha\in\Phi(G)$ satisfying $U_\alpha \subseteq H$.

\subsubsection{}
At various places in \S \ref{s:tameRigidity} and \S\ref{s:wildRigidity}, we use exponentials of nilpotent matrices. In fact, all exponentials that appear are of the form $\exp(X)=I+X$. Thus, no restriction on the characteristic of the base field $k$ is required.

\subsection{Loop group} 
For each $x\in X$, 
let $G(F_x)$  (res. $G(\cO_x)$) be the loop group (resp. the positive loop group). Let $\fg(F_x)$ and $\fg(\cO_x)$ denote the corresponding loop algebras. Let $\Phi^\aff$ denote the set of affine roots. Each affine root $\tilde{\alpha}$ can be written as a sum $\alpha+m$ where $\alpha \in \Phi$ and $m\in \bZ$. The corresponding one-parameter subgroup is 
$U_{\tilde{\alpha}} = U_\alpha(t_x^m)$. Let $\Delta^\aff=\Delta^\aff(G)=\{\alpha_0=1-\theta, \alpha_1,...,\alpha_{n-1}\}$ denote the set of standard simple affine roots.

\subsubsection{}
Let $\tW$ denote the Iwahori-Weyl group defined by 
\[
\tW:=N_{G(F_x)} (T(F_x))/T(\cO_x). 
\]
We have an inclusion $X_*(T)\hookrightarrow G(F_x)$,  given by the map $\lambda=(\lambda_1,...,\lambda_n)\mapsto \on{diag}(t_x^{\lambda_1},...,t_x^{\lambda_n})$. This gives rise to an exact sequence 
\[
1 \ra X_*(T)\ra \tW \ra W \ra 1.
\]
Permutation matrices give a section of the above short exact sequence, resulting in an isomorphism $\tW \simeq W\ltimes X_*(T)$, cf. \cite[Proposition 13]{PappasRapoport}.

\subsubsection{}\label{s:Omega} 
Let $\Omega$ denote the stabiliser in $\tW$ of an alcove. 
As explained in Lemma 14 of \emph{op. cit}, we have a short exact sequence
\[
1\ra W^\aff \ra \tW\ra X^*(Z(\hG))\ra 1, 
\]
where $W^\aff$ is the affine Weyl group. The group $\Omega \subset \tW$ maps isomorphically to $X^*(Z(\hG))$, leading to the semi-direct product decomposition 
\[
\tW\simeq W^\aff \rtimes \Omega.
\]
For $\GL_n$, we have a natural inclusion $\Omega \hookrightarrow G(F_x)$; moreover, the composition 
\begin{equation}\label{eq:Omega} 
\Omega \hookrightarrow G(F_x) \rar{\det} F_x^\times \rar{\on{ord}} \bZ
\end{equation} 
is an isomorphism. Let $\tw_1\in \Omega$ denote the preimage of $1$ under this isomorphism.

\section{Hypergeometrics sheaves} \label{s:hypergeometric}  
In this section, we recall the definition and some properties of $\ell$-adic hypergeometric local systems, which for brevity, are called \emph{hypergeometric sheaves}. We do not discuss the monodromy group of hypergeometric sheaves, as they are not used in the text. For further details, see \cite{KatzBook}.

\subsection{The definition of hypergeometric sheaves} 

\subsubsection{Convolution} 
Let $\bGm$ denote the multiplicative group $\Spec(k[t,t^{-1}])$, $\mu: \bGm\times \bGm \ra \bGm$ the multiplication, and $\iota: \bGm\ra \bGm$ the inversion map.  The convolution (with compact support) is the functor 
\[
\star: D^b(\bGm)\times D^b(\bGm)\ra D^b(\bGm), \qquad \cF\star \cG:=\mu_! (\cF\boxtimes \cG). 
\]

\subsubsection{Initial data} \label{sss:initialData} 
To talk about hypergeometrics, we need an initial data consisting of 
\begin{itemize} 
\item a pair of non-negative integers $(n,m)\neq (0,0)$; 
\item  a nontrivial additive character $\psi: k\ra \bQlt$; 
\item  multiplicative characters  $k^\times \ra \bQlt$ denoted by $\chi_1,...\chi_n$ and $\rho_1,...,\rho_m$. 
\end{itemize} 
Let $\overline{\psi}$, $\ochi_i$, $\overline{\rho}_j$ denote the inverse characters. Let $\cL_\psi$, $\cL_{\chi_i}$, and $\cL_{\rho_j}$ denote the $\ell$-adic local systems on $\bGm$ whose Frobenius trace functions are $\psi$, $\chi_i$, and $\rho_j$, respectively

\subsubsection{}
To the above initial data, Katz \cite[\S 8.2]{KatzBook} associated the (generalised) hypergeometric sheaf $\sH=\sH(\psi, \chi_1,...,\chi_n, \rho_1,...,\rho_m)$ as follows:   
\[
\sH:= \sH(\psi, \chi_1, \varnothing)\star \cdots\star \sH(\psi, \chi_n,  \varnothing)\star 
\sH(\psi, \varnothing, \rho_1)\star \cdots  \star \sH(\psi, \varnothing, \rho_m)[n+m-1],
\]
where 
\[
\sH(\psi, \chi_i, \varnothing) := \cL_\psi \otimes \cL_{\chi_i}, \qquad \textrm{and} \qquad \sH(\psi, \varnothing, \rho_j):= \iota^* (\cL_{\overline{\psi}} \otimes \cL_{\overline{\rho}_j}). 
\]
As noted in \cite[\S 8.2]{KatzBook}, 
$\iota^* \sH(\psi, \chi_1,...,\chi_n; \rho_1,...,\rho_m)\simeq \sH(\overline{\psi}, \overline{\rho}_1,...,\overline{\rho}_m, \ochi_1,...,\ochi_n)$. 
Thus,  without the loss of generality, we may (and do) assume that $m\leq n$.

\subsubsection{Kloosterman sheaves} 
For $m=0$, the hypergeometric sheaf $\sH(\psi, \chi_1,...,\chi_n; \varnothing)$ is nothing but the generalised Kloosterman sheaf \cite{KatzKloosterman}. When $\chi_i=1$ for all $i$, we recover Deligne's Kloosterman sheaf \cite{Deligne}.

\subsubsection{Finite hypergeometric functions}\label{sss:hyp sum}
Let $\sH=\sH(\psi, \chi_1,...,\chi_n, \rho_1,...,\rho_m)$ be a hypergeometric sheaf. Let $\tr_\sH: k^\times \ra \bQl$ denote its Frobenius trace function. As noted in \cite[\S 8.2.7]{KatzBook}, for every $a\in k^\times$, we have  
\[
\mathrm{tr}_\sH(a)=(-1)^{n+m-1}\sum_{x_1\cdots x_n=ay_1\cdots y_m}\psi(\sum_{i=1}^nx_i-\sum_{i=1}^my_i)\prod_{i=1}^n\chi_i(x_i)\prod_{i=1}^m\rho_i(y_i^{-1}).
\]
The function $\on{tr}_\sH: k\ra \bQl$ is known as a \emph{finite hypergeometric function}.

\subsubsection{Alternative realisation} 
The above explicit expression for the trace function allows one to give an alternative realisation of hypergeometric sheaves. Namely, consider the correspondence 
$$
\begin{tikzcd}
&  \bGm^n\times \bGm^{m}  \arrow{dl}[swap]{p} \arrow{dr}{q}  & \\
\bGa\times \bGm^n\times \bGm^m &       &  \bGm, 
\end{tikzcd}
$$
where 
\[
q(x_1,...,x_n, y_1,...,y_m):=(x_1...x_n)(y_1...y_m)^{-1}, 
\]
and 
\[
p(x_1,...,x_n, y_1,...,y_m):= (\sum_{i=1}^n x_i - \sum_{j=1}^m y_j, x_1,...,x_n, y_1^{-1},...,y_m^{-1}). 
\]
Let 
\[
\sF:=q_! p^* (\cL_\psi\boxtimes\cL_{\chi_1}\boxtimes \cdots\boxtimes \cL_{\chi_n} \boxtimes \cL_{\rho_1} \boxtimes\cdots\boxtimes \cL_{\rho_m})[n+m-1]. 
\]
Using induction on $m$ and $n$, one can show that $\sH\simeq \sF$. Alternatively, we can establish this isomorphism as follows. 
The $\ell$-adic complexes $\sH$ and $\sF$ have the same Frobenius trace function; thus, their classes in the Grothendieck group coincide. We will see below that if $\chi_i\neq \rho_j$, then $\sH$ is actually a simple local system. This gives another proof that $\sH\simeq \sF$, under the assumption that $\chi_i\neq \rho_j$.

\subsubsection{Algebraically closed base field} 
Instead of working over a finite field, we can work over an algebraically closed field $K$ of positive characteristic \cite[\S 8.3]{KatzBook}. 
Namely, for a finite subfield $k\subset K$, we can speak of the Artin--Schrier sheaf $\cL_{\psi}$ on $\bA^1 \otimes_k K$. Moreover, for any tame character $\chi$ of $\pi_1(\bGm\otimes_k K)$, we can speak of the Kummer sheaf $\cL_\chi$ on $\bGm\otimes_k K$. We then define the hypergeometric sheaf $\sH(\psi, \chi_1,...,\chi_n, \rho_1,...,\rho_m)$ exactly as above. 
If $\chi$'s and $\rho$'s are all of finite order, say defined over $k$, then these objects are just the base change to $K$ of the earlier defined objects on $k$. If $m=n$, then one can show that the resulting sheaf on $K$ is independent of $\psi$. This leads to a motivic description of tame hypergeometric sheaves \cite[\S 8.17]{KatzBook}.

\subsection{Basic properties} 
To avoid repeated mention of local systems coming from the base field, in this subsection we assume that we are working over an algebraically closed field. We say that $\chi_i$'s and $\rho_j$'s are \emph{disjoint} if $\chi_i\neq \rho_j$ for all $i$ and $j$. It follows from  \cite[Theorem 8.4.5 and Corollary 8.4.10.1]{KatzBook} that the complex $\sH[1]$ is perverse. 
Moreover, it is irreducible if and only if $\chi_i$'s and $\rho_j$'s are disjoint.

\subsubsection{Tame case} 
Suppose $m=n$ and $\chi_i\neq \rho_j$ for all $i, j$. According to \cite[Theorem 8.4.2]{KatzBook}, the sheaf $\sH$ is lisse on $\bGm-\{1\}$ with tame ramification at $0$, $1$, and $\infty$. As a representation of the inertia group $I(0)$, it is isomorphic to 
\[
\bigoplus_{\textrm{distinct $\chi$'s}} \cL_\chi \otimes \textrm{Unip}(\on{mult}(\chi)),
\]
where $\textrm{Unip}(\on{mult}(\chi))$ denotes the Jordan block of size multiplicity of $\chi$. As an $I(\infty)$-representation, it is isomorphic to 
\[
\bigoplus_{\textrm{distinct $\rho$'s}} \cL_\rho \otimes \textrm{Unip}(\on{mult}(\rho)).
\]
Finally, $I(1)$ acts as a pseudo-reflection with determinant $[x\mapsto x-1]^*\cL_{\Lambda}$, where $\Lambda:=\prod_j \rho_j / \prod_i \chi_i$.

\subsubsection{Wild case} \label{sss:wildHypergeometric}
Suppose $m<n$ and $\chi_i\neq \rho_j$ for all $i, j$. According to \cite[Theorem 8.4.2]{KatzBook}, $\sH$ is lisse on $\bGm$ with tame ramification at $0$ and wild ramification at $\infty$. As an $I(0)$-representation its description is exactly as in the tame case. As an $I(\infty)$-representation it has Swan conductor $1$ and is isomorphic to the direct sum 
\[
\sW \oplus \bigoplus_{\textrm{distinct $\rho$'s}} \cL_\rho \otimes \textrm{Unip}(\on{mult}(\rho)), 
\]
where $\sW$ is an $(n-m)$-dimensional wild local system with a single break $1/(n-m)$. The local rigidity theorem \cite[Theorems 8.6.3 and 8.6.4]{KatzBook}, implies that $\sW$ is isomorphic to a generalised Kloosterman sheaf.  (For the de Rham version of the local rigidity theorem, cf. \cite[Theorem 5]{KamgarpourSage0}.)

\subsubsection{Rigidity} 
Let $S$ be a finite subset of $\bP^1$ over an algebraically closed field. 
A local system $E$ on $\bP^1-S$ is said to be \emph{rigid} if $E$ is completely determined by the collection of representations of the inertia groups $I(x)$, $x\in S$. According to \cite[\S 8.5]{KatzBook}, if $\chi_i\neq \rho_j$ for all $i$ and $j$, then the hypergeometric sheaf $\sH(\psi; \chi_1,...,\chi_n; \rho_1,...,\rho_m)$ is rigid. This result is the key conceptual motivation for our approach to constructing hypergeometric Hecke eigensheaves.

\section{Integral Models} \label{s:integralModels}
The notion of integral model is a convenient tool for studying ramifications in the geometric Langlands program. In this section, we review some basic facts about them. For notational convenience, we continue to assume that $G=\GL_n$ and $X=\bP^1$. For the case of arbitrary reductive groups over general smooth projective curves, see, e.g., \cite{YunCDM, Yu}.

\subsection{Definition of integral models} 
An integral model for $G$ over $X$ is a smooth affine group scheme $\cG$ over $X$ together with an isomorphism $\cG|_{\Spec(F)}\simeq G$. A point $x\in X$ is called \emph{unramified} if $\cG(\cO_x)\simeq G(\cO_x)$; otherwise, $x$ is \emph{ramified}. Ramified points form a finite set $S\subset |X|$. We assume throughout that 
for each $x\in S$, $\cG(\cO_x)$ is a pro-algebraic subgroup of finite codimension in a parahoric $P_x\subset G(F_x)$.

\subsubsection{Construction from local data} 
If one is given a finite set $S\subset |X|$ and for each $x\in S$, a pro-algebraic group $K_x$ of finite codimension in some parahoric, then one can construct an integral model $\cG$ satisfying 
\[
\cG(\cO_x) \simeq 
\begin{cases} 
G(\cO_x) & x\in X-S;\\
K_x & x\in S.
\end{cases} 
\]
We refer the reader to \cite{Yu} for a construction of such integral models in the framework of the Bruhat--Tits theory. Note that Yu constructs the integral models in the local setting. The fact that they can be glued to define a group scheme over $X$ follows from the lemma of Beauville and Laszlo \cite{BL}; see also \cite[Lemma 3.18]{CGP}.

\subsection{Moduli of $\cG$-bundles}\label{s:stacks}
Let $\cG$ be an integral model as above. Let $\Bun_{\cG}$ be the moduli stack of $\cG$-bundles. If $\cG(\cO_x)\subseteq G(\cO_x)$ for all $x\in S$, then we can think of a $\cG$-bundle as a rank $n$-vector bundle equipped with extra data (such as flags, framings, etc.) at $S$. Note that as we have assumed $G=\GL_n$, all parahorics are conjugate to a subgroup of $G(\cO_x)$. Thus, we may (and do)  assume that $\cG(\cO_x)\subseteq G(\cO_x)$ for all $x$. (We shall see in the example of wild hypergeometrics, however, that it is convenient to allow parahorics not a priori in $G(\cO_x)$.)

\subsubsection{Dimension}
Let $\Bun_n:=\Bun_{\GL_n}$ denote the moduli stack of rank $n$ vector bundles on $X$. 
We have a forgetful map $\Bun_\cG \ra \Bun_n$ which is a locally trivial $\prod_{x\in S} G(\cO_x)/\cG(\cO_x)$-fibration. Thus, $\Bun_{\cG}$ is smooth and  
\begin{equation} \label{eq:dimBunG}
\dim(\Bun_{\cG})= \dim(\Bun_n) + \sum_{x\in S} \dim(G(\cO_x)/\cG(\cO_x))=
-\dim(G) + \sum_{x\in S}\dim(G(\cO_x)/\cG(\cO_x)). 
\end{equation}

\subsubsection{Generic locus}
For every $\cE\in \Bun_\cG$, $\Aut_\cG(\cE)$ is an algebraic group. The map $\Bun_{\cG}\ra \bZ_{\geq 0}$ defined by $\cE\mapsto \dim(\Aut_{\cG}(\cE))$ is upper semi-continuous;  therefore, the substack of $\Bun_\cG$ consisting of bundles of minimal dimension is open. We call this the \emph{generic locus} of $\Bun_\cG$.

\subsubsection{Uniformisation} 
Weil's adelic uniformisation (cf. \cite[\S 2.4]{YunCDM}) states that we have a canonical bijection 
\[
\Bun_{\cG}(k)\simeq G(F) \backslash G(\bA_F)/ \cG(\cO_\bA).
\]
In favourable situations, we also have ``one-point uniformisation''. For instance, it is proved in \cite{HeinlothUniformisation} that we have one-point uniformisation whenever $\cG$ is an integral model for a semisimple group over an arbitrary curve. 
A more relevant case for us is when $G=\GL_n$, $X=\bP^1$, $S=\{0,\infty\}$, $\cG(\cO_0)=P^\opp$ and $\cG(\cO_\infty)=P$, where $P$ is a parahoric and $P^\opp$ is its opposite. Then, one can show that every $\cG$-bundle on $\bP^1-\infty$ is trivial; thus, 
\[
\Bun_{\cG}(k) = P^- \backslash G(F_\infty) /P,
\] 
where $P^-=P^\opp \cap G(k[s,s^{-1}])$, cf. \cite[Proposition 1.1]{HNY} or \cite[\S 2.12]{YunEpipelagic}.

\subsection{Connected components of $\Bun_{\cG}$} 

\subsubsection{Kottwitz homomorphism} \label{s:Kottwitz} 
Let $\kappa: \Bun_{\cG}\ra X^*(Z(\hG))\simeq \bZ$ denote the Kottwitz homomorphism, cf. \cite[\S 2]{YunCDM}. Using adelic uniformisation, 
we can identify $\kappa$ (on the level of $k$-points) with the composition 
\begin{equation}\label{eq:Kottwitz}
G(\bA_F)\rar{\det} \bA_F^\times \rar{\deg} \bZ.
\end{equation} 
If we view $\cG$-bundles as vector bundles equipped with additional data, then  $\kappa$ coincides with taking the degree of the underlying vector bundles.

\subsubsection{Connected components}\label{ss:connected components}
The map $\kappa$ factors through $\Bun_\cG\ra\Bun_n$, which, as noted above, is a fibration with connected fibres. Thus, $\kappa$ induces an isomorphism 
$\pi_0(\Bun_{\cG})\simeq\pi_0(\Bun_n)\simeq \bZ$. 
For each $\alpha \in \bZ$, we let $\Bun_{\cG}^\alpha$ denote the corresponding connected component and $\mathring{\Bun_{\cG}^\alpha}$ its generic locus. 

\begin{rem} In \S \ref{s:Omega}, we fixed an isomorphism $\Omega\simeq \bZ$ and a generator $\tw_1\in \Omega$ which maps to $1$ under this isomorphism. Thus, we can also use  $\Omega$ to parameterise connected components of $\Bun_\cG$. In this language, $\Bun_{\cG}^\alpha$ is the component corresponding to $\tw_1^\alpha$. 
\end{rem}

\subsubsection{Identification of components}\label{sss:Hk_alpha}
The integral models of interest to us satisfy the property that for some $x\in X$, we have $\cG(\cO_x)=I(j)$, where the latter is the $j$th Moy--Prasad subgroup of the standard Iwahori. This allows us to identify the components of $\Bun_\cG$ (cf. \cite[Corollary 1.2]{HNY}). Indeed, $\Omega=N_{\tW}(I)$ acts on $I(j)$; thus, it also acts on $\Bun_{\cG}$ by changing the level structure at $x$. With respect to the Kottwitz homomorphism, this action satisfies 
\[
\kappa(\tw_1\cdot\cE)=\kappa(\cE)+1,
\]
It follows that for each $\alpha\in\bZ$, the map $\cE\mapsto \tw_1^\alpha\cdot \cE$ defines an isomorphism 
\begin{equation}\label{eq:Hk}
\Hk_\alpha: \Bun_{\cG}^0\simeq \Bun_\cG^\alpha.
\end{equation}  
These isomorphisms will play an important role in the construction of hypergeometric Hecke eigensheaves.

\subsubsection{Integral model for the centre} \label{s:centre} 
The integral model $\cG$ defines an integral model $\cZ$ for the centre $Z=Z(G)\simeq \bGm$; namely,  
\[
\cZ(\cO_x) := \cG(\cO_x)\cap Z(F_x).
\]
In the examples of interest to us, $\cZ$ is unramified everywhere except possibly at one point $x\in X$; moreover, at this point, $\cZ(\cO_x)$ is either $\cO_x^\times$ or $1+\cP_x$. 
In both cases, the identification 
\[
\Bun_\cZ^0(\overline{k})=\overline{k}^\times\backslash (\cO_x\otimes_k \overline{k})^\times/\cZ(\cO_x\otimes_k \overline{k})
\]
implies that (the set of isomorphism classes of) $\Bun_\cZ^0(\overline{k})$ is a point. Thus, the coarse moduli space of $\Bun_\cZ^0$ is also a point. When constructing hypergeometric Hecke eigensheaf, this fact allows us to bypass some of the technical aspects of \cite{YunCDM}.

\section{Ramified Hecke eigensheaves}\label{s:eigensheaves} 
In this section, we explain what one means by a Hecke eigensheaf on $\Bun_\cG$, where $\cG$ is an integral model for $G$ over $X$, cf.  \cite{HNY, YunCDM}. We continue to assume that $G=\GL_n$ and $X=\bP^1$, though except for \S \ref{ss:Eigenvalue}, the discussions apply verbatim to split reductive groups over smooth projective curves.

\subsection{Geometric Hecke operators}

\subsubsection{Hecke stack} 
The stack of Hecke modifications is defined as 
\[
\Hecke=\mathrm{Hecke}_\cG:=\{(\cE_1,\cE_2,x,\phi)\mid \cE_1,\, \cE_2\in\Bun_{\cG}, x\in (X-S), \,  \phi:\cE_1|_{X-x}\simeq \cE_2|_{X-x} \}.
\]
We have forgetful functors $\pr_1$ and $\pr_2$ mapping $(\cE_1,\cE_2,x,\phi)$ to $\cE_1$ and $(\cE_2,x)$, respectively. Thus, we obtain the \emph{Hecke correspondence} 
\begin{equation}\label{eq:Hecke correspondence}
\begin{tikzcd}
&  \mathrm{Hecke} \arrow{dl}[swap]{\pr_1} \arrow{dr}{\pr_2}  & \\
\Bun_{\cG} &       &  \Bun_{\cG}\times(X-S)
\end{tikzcd}
\end{equation}

The morphism $\pr_2$ is a locally trivial fibration whose fibres are isomorphic to the affine Grassmannian $\Gr=\on{Gr}_G$. The morphism $\pr_1$ is a locally trivial fibration whose fibres are isomorphic to the Beilinson--Drinfeld Grassmanian $\on{GR}=\GR_G$.

\subsubsection{} 
Let $\hG$ denote the Langlands dual group. The geometric Satake isomorphism \cite{Ginzburg,  BD, MirkovicVilonen}
associates to every representation $V$ of $\hG$ a perverse sheaf $\IC_V$ on $\Gr$, and therefore also on $\GR$ and $\Hecke$. For each $V\in \Rep(\hG)$, we let $\Gr_V\subseteq \Gr$, $\GR_V\subseteq \GR$, and $\Hecke_V\subseteq \Hecke$ denote the support of these perverse sheaves.

\subsubsection{}
The geometric Hecke operators are defined by 
\begin{align*} 
\Hk:  \on{Rep}(\hG) \times D^b(\Bun_\cG) &  \ra  D^b(\Bun_\cG\times (X-S))\\
(V, \cE) & \mapsto   \Hk_V(\cE):=\pr_{2,!} (\pr_1^* \cE \otimes \on{IC}_V). 
\end{align*} 
Alternatively, we can first restrict the above correspondence to $\Hecke_V$ 
and then define $\Hk_V$ by the same formula (where now $\pr_i$ are morphisms from $\Hecke_V$).

\subsection{Eigensheaves}
 
\subsubsection{Definition of Hecke eigensheaves}
A non-zero perverse sheaf $\sA$ on $\Bun_\cG$ is called an \emph{eigensheaf} if there exists a $\hG$-local system $E$ on $X-S$, viewed as a tensor functor
\[
\Rep(\hG) \ra \on{LocSys}(X-S), \qquad V\mapsto E_V,
\]
and coherent isomorphisms  $\Hk_V(\sA) \simeq \sA \boxtimes E_V$. 
Here, ``coherent'' means compatible with the tensor structure of $\Rep(\hG)$ and the composition of Hecke operators, cf. \cite[\S 2]{GaitsgorydeJong} for details. The local system $E$ is known as the \emph{Hecke eigenvalue} of $\sA$.

\subsubsection{Reformulation of the core conjecture} 
In the present framework, Conjecture \ref{c:main} can be phrased as follows: for every irreducible $\hG$-local system $E$ on $X-S$, there exists an integral model $\cG$ for $G$ on $X$, with ramification points $S$, and a Hecke eigensheaf $\sA$ on $\Bun_\cG$ whose eigenvalue is $E$. In the cases of interest to us (i.e. the rigid situation), the perverse sheaf $\sA$ satisfies a cleanness property. We now recall what one means by a clean perverse sheaf.

\subsubsection{Cleanness}  
Let $Y$ be an Artin stack of dimension $d$ and $j: U\hookrightarrow Y$ the inclusion of a non-empty open smooth substack. 
Let $\cL$ be a local system on $U$ and $\cP:=\cL[d]\in \on{Perv}(U)$. Then, we have a canonical morphism 
\begin{equation} \label{eq:clean1} 
j_! \cP \ra j_* \cP
\end{equation} 
in $D^b(Y)$, which induces a morphism ${}^p H^0(j_! \cP) \ra {}^p H^0 (j_* \cP)$ in $\mathrm{Perv}(Y)$ whose image is denoted by $j_{!*} \cP$. One says that $j_{!*}\cP$ is \emph{clean} if the morphism \eqref{eq:clean1} is an isomorphism. In this case, we obtain isomorphisms $j_! \cP\simeq j_{!*}\cP \simeq j_*\cP$. In particular, the stalks and costalks of $j_{!*} \cP$ on $Y-U$ vanish.

\subsection{Hecke eigenvalue}\label{ss:Eigenvalue} Let $\sA$ be a Hecke eigensheaf on $\Bun_\cG$ with Hecke eigenvalue $E$. 
For every $\alpha\in \bZ=\pi_0(\Bun_\cG)$, let $j_\alpha:\mathring{\Bun_\cG^\alpha}\hookrightarrow\Bun_\cG^\alpha$ denote the open inclusion of the generic locus. Let $\sA_\alpha$ denote the restriction of $\sA$ to the component $\Bun_{\cG}^\alpha$. In this subsection, we give a convenient expression for $E$ under the assumption that $\sA_\alpha$ is a clean extension of a rank one local system $\cL_\alpha$ on the generic locus, i.e. we assume $\sA_\alpha\simeq j_{\alpha,!}\cL_\alpha[\dim(\Bun_\cG^\alpha)]\simeq j_{\alpha,*}\cL_\alpha[\dim(\Bun_\cG^\alpha)]$. For convenience, we also assume all the components $\Bun_\cG^\alpha$ have the same dimension. As we shall see, these assumptions are satisfied in the hypergeometric (and, more generally, rigid) case.

\subsubsection{Restricting to the fundamental coweight}
As we are dealing with $\hG=\GL_n$, to specify $E$,  it is sufficient to consider the standard representation and describe $E_{\on{Std}}$. Let $\omega_1=(1,0,...,0)$ be the first fundamental coweight. 
Restricting the Hecke correspondence to the substack $\Hecke_{\omega_1}$, we obtain the diagram 
\[
\begin{tikzcd}
&  \text{Hecke}_{\omega_1}\arrow{dl}[swap]{\text{pr}_1}  \arrow{dr}{\pr_2}  & \\
\Bun_\cG^{0} &       &  \Bun_\cG^{1}\times(X-S).
\end{tikzcd}
\]
The morphism $\mathrm{pr}_2$ is a locally trivial fibration whose fibres are isomorphic to 
\[
 \mathrm{Gr}_{\omega_1} = G/P_{\omega_1} \simeq \bP^{n-1}, 
\] 
where $P_{\omega_1}$ is the maximal parabolic associated to $\omega_1$. 
Thus, up to a Tate twist,  $\IC_{\omega_1}=\bQl[n-1]$ and the Hecke eigen property gives us an isomorphism
\[
\sA_1\boxtimes E_\Std\simeq \pr_{2,!}(\pr_1^* \sA_0\otimes \IC_{\omega_1})=\pr_{2,!}\pr_1^* \sA_0[n-1].
\]

\subsubsection{Restricting to the generic locus of $\Bun_\cG^1$} 
Let $\star$ be a point in the generic locus $\mathring{\Bun_\cG^1}$. Restricting the above isomorphism to $\star\times (X-S)$, we obtain an isomorphism 
\[ 
E_\Std[\dim(\Bun_\cG)] \simeq (\pr_{2,!} \pr_1^* \sA_0[n-1])|_{\star\times (X-S)}.
\] 
Alternatively, we can express $E_\Std$ as follows. 
If we restrict the  above correspondence to $\star \times (X-S)$, we obtain the diagram 
\[
\begin{tikzcd}
&  \text{GR}_{\omega_1} \arrow{dl}[swap]{p_1}  \arrow{dr}{p_2}  & \\
\Bun_\cG^{0} &       &  X-S.
\end{tikzcd} 
\]
where $p_1,p_2$ are the restrictions of $\pr_1,\pr_2$. Applying proper base change, we get 
\[
E_{\Std} \simeq p_{2,!} p_1^* \sA_0[n-1-\dim(\Bun_\cG)].
\]

\subsubsection{Restricting to the generic locus of $\Bun_\cG^0$} 
Recall that the cleanness assumption means that $\sA_0$ is the extension by zero from a local system $\cL_0$ on the generic locus $\mathring{\Bun_{\cG}^0}$, shifted by degree $\dim(\Bun_\cG)$. 
Now restricting the above correspondence to $\mathring{\Bun_{\cG}^0}$, we obtain 
\begin{equation}\label{eq:relevantHecke}
\begin{tikzcd}
&  \mathring{\mathrm{GR}}_{\omega_1} \arrow{dl}[swap]{\pi_1}  \arrow{dr}{\pi_2}  & \\
\mathring{\Bun_{\cG}^0} &       &  X-S.
\end{tikzcd} 
\end{equation}
where $\pi_1,\pi_2$ are the restrictions of $p_1,p_2$. Applying proper base change again, we conclude
\begin{equation}\label{eq:eigenvalue}
 E_{\on{Std}}\simeq \pi_{2!} \pi_1^* \cL_0[n-1].
\end{equation}

\subsubsection{} 
When $\sA$ is a Hecke eigensheaf arising from a rigid hypergeometric automorphic data, 
we give a realisation of the above correspondence using an explicit description of $\Bun_\cG$ in terms of lattices. This leads to an explicit formula for $E_\Std$, which in turn,  allows us to prove that it is isomorphic to a hypergeometric local system.

\section{Rigid automorphic data} \label{s:rigid} 
\subsection{Overview} 
\subsubsection{} 
The notion of rigid automorphic data is due to Zhiwei Yun  \cite{YunCDM}. In what follows, we discuss this notion in the presence of three simplifying assumptions; namely, we take $G=\GL_n$, $X=\bP^1$, and ignore central characters. The restriction on $G$ is not very serious; indeed, with a bit more notation, one can write the story for arbitrary split reductive $G$. The restriction on $X$ is also not serious for one knows that (under mild assumptions) rigid automorphic data exists only for curves of genus $0$ and $1$, cf. \cite[Lemma 2.7.12.(2)]{YunCDM}. This is the automorphic analogue of the fact that there are no interesting rigid local systems on curves of genus greater than $1$ \cite[\S 1]{KatzRigid}.

\subsubsection{} 
 The omission of central characters is, however, serious and allows us to bypass a large amount of technical details of \cite{YunCDM}. The version of automorphic data without central characters is sufficient for our purposes because we are primarily  interested in Hecke eigensheaves. To understand rigid automorphic representations, however, one needs the full theory of \emph{op. cit}. For instance, our ``stripped down'' notion of automorphic data prevents us from making a precise statement regarding hypergeometric automorphic representations, cf. \S \ref{sss:representations}. We believe, however, that this omission is worth the price, for otherwise, we would need to introduce a large amount of distracting notation and technical details. 

\subsubsection{} After giving our version of rigid automorphic data, we explain how they give rise to Hecke eigensheaves. 
The construction of eigensheaves given below is essentially that of \cite{YunCDM}, however, at one crucial step (roughly, going from weak to strong Hecke eigensheaves), we need to use an argument from \cite{HNY} in order to avoid one of the assumptions of \cite{YunCDM} (namely, Assumption 4.7.1) which is not satisfied for general hypergeometric automorphic data.

\subsection{Automorphic data} 

\begin{defe} \label{d:automorphicData}
An \emph{automorphic data} for $G$ on $X$ is a finite subset $S\subset X$ together with a pair
$(K_S,\gamma_S)$, where
\begin{itemize}
\item $K_S=\{K_x\}_{x\in S}$ is a collection of pro-algebraic groups $K_x$  with finite codimension in some parahoric subgroup in $G(F_x)$. By an abuse of notation, we also denote $K_S=\prod_{x\in S} K_x$.
\item $\gamma_S$ is a collection $\{\gamma_x\}_{x\in S}$ of rank one character sheaves
$\gamma_x$ on $K_x$, which is the pullback of a rank one character sheaf from some finite dimensional quotient $K_x\twoheadrightarrow K_x/K_x^+$. Here, $K_x^+\subseteq K_x$ is a pro-algebraic normal subgroup. We let 
\[
L_x:=K_x/K_x^+, \qquad L_S:=\prod_{x\in S} L_x,\qquad \textrm{and} \qquad   \gamma_S:=\boxtimes_{x\in S} \gamma_x.
\]  
\end{itemize}
\end{defe} 
For a review of rank one character sheaves, cf.  \cite{Masoud} or the appendix of \cite{YunCDM}.

\subsubsection{Integral models associated to automorphic data}\label{sss:group schemes}
To every automorphic data $(K_S, \gamma_S)$, one associates integral models  $\cG$ and $\cG'$ satisfying
\[
\begin{aligned}
\cG|_{X-S}&=G\times(X-S);\qquad \cG|_{\cO_x}=K_x^+,\ \forall x\in S;\\
\cG'|_{X-S}&=G\times(X-S);\qquad \cG'|_{\cO_x}=K_x,\ \forall x\in S. 
\end{aligned}
\]
In what follows, we use the properties of $\Bun_\cG$ and $\Bun_{\cG'}$ discussed in the previous section. 

\subsubsection{Stabilisers of bundles}
Let $\cF\in \Bun_{\cG}$ and let $\Stab_{L_S}(\cF)$ denote its stabiliser in the stacky sense; i.e., 
\[
\Stab_{L_S}(\cF):=\{(l, \eta) \, | \, l\in L_S, \, \, \eta\in \on{Isom}(\cF, l\cdot\cF)\}. 
\]
We have a canonical morphism
\[
\Stab_{L_S}(\cF)\ra L_S,\qquad (l,\eta)\mapsto l,
\]
which allows us to define the pullback of $\gamma_S$ to $\Stab_{L_S}(\cF)$.

\subsubsection{Relating stabilisers to $L_S$} 
There is a canonical forgetful map 
\[
p: \Bun_{\cG}\ra \Bun_{\cG'},
\] 
which is an $L_S$-torsor. 
If $\cE:=p(\cF)\in \Bun_{\cG'}$, then $ \Aut_{\cG'}(\cE)\simeq \Stab_{L_S}(\cF)$, and the resulting map $\Aut_{\cG'}(\cE)\ra L_S$ coincides with the composition 
\begin{equation} \label{eq:Res}
\on{Res}:\ \Aut_{\cG'}(\cE) \ra \prod_{x\in S} \Aut(\cE|_{\cO_x})\simeq K_S \ra L_S.
\end{equation}

\subsection{Rigidity} 
Suppose we have an automorphic data $(K_S,\gamma_S)$ with associated integral models $\cG'$ and $\cG$.

\subsubsection{Relevant $\cG$-bundles}  
A bundle $\cF\in \Bun_\cG(\overline{k} )$ is called \emph{relevant} if the pullback of $\gamma_S$ to $(\on{Stab}_{L_S}(\cF))^\circ$  is the constant sheaf; otherwise $\cF$ is irrelevant. Note that $\cF$  is relevant if and only if all the elements in its $L_S$-orbit are relevant; thus, we can talk about relevant \emph{orbits} on $\Bun_\cG$.

\subsubsection{Relevant $\cG'$-bundles} 
We also have a notion of relevant $\cG'$-bundles. Namely, we say that $\cE\in \Bun_{\cG'}(\overline{k} )$ is relevant if one (and therefore all) $\cG$-bundles in the fibre $p^{-1}(\cE)$ is relevant.  Thus, $\cE\in \Bun_{\cG'}$ is relevant if and only if the pullback of $\gamma_S$ to $(\Aut_{\cG'}(\cE))^\circ$ is constant. In particular, if $\Aut_{\cG'}(\cE)$ is trivial, then $\cE$ is automatically relevant.

\subsubsection{Definition of (strict) rigidity} 
Let $\cZ\subseteq \cG'$ be the integral model for the centre $Z\subseteq G$ associated to $\cG'$  (\S \ref{s:centre}); i.e., 
$\cZ(\cO_x) := \cG'(\cO_x)\cap Z(F_x)$. 

\begin{defe}\label{d:rigidity} \mbox{}
\begin{enumerate} 
\item An automorphic data is called \emph{rigid} if for all $\alpha \in \bZ=\pi_0(\Bun_{\cG'})$, there exist a unique relevant element $\cE_\alpha$  on the component $\Bun_{\cG'}^\alpha$ (equivalently, there is a unique relevant orbit $\cO_\alpha$ on $\Bun_\cG^\alpha$).  
\item 
An automorphic data  is called \emph{strictly rigid} if it is rigid and the following properties hold: 
\begin{enumerate} 
\item the stabilisers of relevant elements are trivial, i.e., $\Aut_{\cG'}(\cE_\alpha)=\{1\}$ for all $\alpha \in \bZ$;  
\item  the coarse moduli space of $\Bun_{\cZ}^0$ is a point.
\end{enumerate} 
\end{enumerate} 
\end{defe}

\subsubsection{How to establish rigidity?} 
In the cases of interest to us, we prove rigidity by first showing that the generic locus of $\Bun_{\cG'}^\alpha$ consists of a unique bundle and this bundle has trivial automorphisms. This bundle is, therefore, automatically relevant. We then show that all other bundles in $\Bun_{\cG'}^\alpha$ are irrelevant, by exhibiting a one-dimensional subscheme in their stabiliser such that the restriction of $\gamma_S$ to this subscheme is non-constant.

\subsubsection{Numerical requirement for rigidity}\label{s:numerical}
Suppose we have a strictly rigid automorphic data. Then $\{\cE_\alpha\}$ is an open substack of  $\Bun_{\cG'}^\alpha$ with trivial stabiliser. It follows that $\dim(\Bun_{\cG'})=0$. In view of \eqref{eq:dimBunG}, we obtain: 
\begin{equation} \label{eq:numerical}
\sum_{x\in S} \dim(G(\cO_x)/\cG'(\cO_x))=\dim(G).\footnote{For groups other than $\GL_n$, one encounters parahorics that are not in $G(\cO_x)$. In this case, one should replace the dimension by relative dimension.}
\end{equation}  
This numerical requirement should be compared with the numerical criteria for (cohomological) rigidity of local systems, cf. \cite[Proposition 3.2.7]{YunCDM} in positive characteristic and \cite[Proposition 11]{FG} or \cite[\S 4.2]{KamgarpourSage} in characteristic zero.

\subsection{Eigensheaves from rigid automorphic data} \label{ss:eigensheaves} 
Let $(K_S, \gamma_S)$ be a strictly rigid automorphic data with integral models $\cG'$ and $\cG$. Let us assume further that the level structure for some $s\in S$ is $I(j)$, where the latter is the $j$th Moy--Prasad subgroup of the Iwahori $I$. We now explain how to construct a Hecke eigensheaf on $\Bun_\cG$ from this data.

\subsubsection{Transporting $\gamma_S$ to $\cO_0$}
Let $\cO_0$ be the unique relevant orbit on $\Bun_\cG^0$ and let 
\[
j_0: \cO_0 \hookrightarrow \Bun_\cG^0 
\]
denote the corresponding open inclusion. Note that $\cO_0$ is a torsor for $L_S$. We choose, once and for all, a trivialisation of this torsor, i.e., an isomorphism $\iota: L_S\simeq \cO_0$.  The character sheaf $\gamma_S$ then defines a local system $\gamma_0$ on $\cO_0$.

\subsubsection{Extending $\gamma_0$ to a perverse sheaf $\sA_0$} 
It follows readily from the uniqueness of relevant orbits that the local system $\gamma_0$ extends to a clean perverse sheaf on $\Bun_\cG^0$. Moreover,  uniqueness of relevant orbits also implies that, up to local systems coming from the base field $k$, $\sA_0$ is the unique  $(L_S, \gamma_S)$-equivariant irreducible perverse sheaf on $\Bun_\cG^0$, cf. \cite{YunCDM}, Lemma 4.4.4.

\subsubsection{Transporting $\sA_0$ to other components} 
 Recall we assumed that one of the level structure of $\cG$ is $I(j)$. As discussed in \S \ref{sss:Hk_alpha}, this assumption implies that for every $\alpha \in \bZ=\pi_0(\Bun_\cG)$, we have an isomorphism 
\[
\Hk_\alpha: \Bun_\cG^0\simeq \Bun_\cG^\alpha,\qquad \cE \mapsto \tw_1^\alpha\cdot \cE.
\]
Under this isomorphism, $\cO_0$ maps to the unique relevant orbit $\cO_\alpha \subset \Bun_\cG^\alpha$. Let $\sA_\alpha:=(\Hk_\alpha^{-1})^*\sA_0$. 
Then $\sA_\alpha$ is a perverse sheaf on $\Bun_\cG^\alpha$ which is a clean extension from a local system on $\cO_\alpha$.

\subsubsection{Correct equivariance property} 
There is a subtlety involved in specifying the equivariance property of $\sA_\alpha$ as it is \emph{not}, in general, $(L_S,\gamma_S)$-equivariant, cf. \cite[\S 4.1.11]{XuZhu}. 
Instead, let 
$\gamma_S^\alpha:=\tw_1^\alpha \cdot \gamma_S$
be the twist of $\gamma_S$ by $\tw_1^\alpha$. Then $\gamma_S^\alpha$ is also a character sheaf on $L_S$ and 
$\sA_\alpha$ is $(L_S,\gamma_S^\alpha)$-equivariant. As above, up to local systems coming from $k$, $\sA_\alpha$ is the unique $(L_S, \gamma_S^\alpha)$-equivariant irreducible perverse sheaf on $\Bun_\cG^\alpha$.

\subsubsection{A key theorem of Yun} 
Let $\sA$ denote the perverse sheaf on $\Bun_\cG$ whose restriction to $\Bun_\cG^\alpha$ equals $\sA_\alpha$. 

\begin{thm}[Yun]
The perverse sheaf $\sA$ is a Hecke eigensheaf. \label{t:Yun}
\end{thm} 

\begin{proof} We explain how this theorem follows from considerations in \cite{YunCDM} together with an argument of \cite{HNY}.  
The Hecke correspondence \eqref{eq:Hecke correspondence} is equivariant with respect to the action of $L_S$. As $\sA_\alpha$ is $(L_S,\gamma_S^\alpha)$-equivariant, it follows that $\Hk_V(\sA)|_{\Bun_\cG^\alpha\times(X-S)}$ is also $(L_S,\gamma_S^\alpha)$-equivariant. Lemma 4.4.4.(2) of \cite{YunCDM} then implies that  
\[
\Hk_V(\sA)|_{\Bun_\cG^\alpha\times(X-S)}\simeq\sA_\alpha\boxtimes E_V^\alpha,
\]
for some  $\ell$-adic complex  $E_V^\alpha$ on $X-S$. The fact that we have a strictly rigid automorphic data implies that Assumptions 4.4.1 of \cite{YunCDM} are satisfied. Thus, we are in a position to apply Theorem 4.4.2 of \emph{op. cit.} to conclude that $E_V^\alpha$ is a semisimple local system. 

It remains to show that $E_V^\alpha$ is independent of $\alpha$. Here, we use the argument of \cite[\S4.2]{HNY}. Namely, we may assume that $\mathrm{IC}_V$ is supported on $\Bun_\cG^\alpha$ for some $\alpha$. Note that $\Hk_\alpha$ commutes with $\Hk_V$ because they are Hecke operators supported at different points of $X=\bP^1$ (the former is supported at a point in $S$, the latter is supported on $X-S$). We therefore obtain
\[
\sA_0\boxtimes E_V^0=\Hk_V(\sA_{-\alpha})=\Hk_V(\Hk_{-\alpha}(\sA_0))
=\Hk_{-\alpha}(\Hk_V(\sA_0))=\Hk_{-\alpha}(\sA_\alpha\boxtimes E_V^\alpha)
=\sA_0\boxtimes E_V^\alpha.
\]
It follows that  $E_V^0=E_V^\alpha$ for all $\alpha$. Let us denote this local system on $X-S$ by $E_V$. Then the assignment $V\mapsto E_V$ defines a tensor functor $\mathrm{Rep}(\GL_n)\rightarrow\mathrm{LocSys}(X-S)$, i.e. a $\GL_n$-local system. This is the Hecke eigenvalue of $\sA$; thus, $\sA$ is a Hecke eigensheaf, as required. 
\end{proof}

\subsubsection{Comparison to Yun's construction}
The construction of the Hecke eigensheaf in \cite{YunCDM} follows a slightly different path. Namely, one first chooses (non-canonical) isomorphisms $L_S\simeq \cO_\alpha$ to transport the local system $\gamma_S$ to each $\cO_\alpha$. The collection of the corresponding perverse extensions defines a perverse sheaf on $\Bun_\cG$, which Yun proves is a  ``weak Hecke eigensheaf" (this is Theorem 4.4.2 \cite{YunCDM} invoked above). Roughly speaking, this means that the resulting eigenvalue is a local system on each component. Yun then proves that under additional assumptions (namely Assumptions 4.7.1 of \emph{op. cit.}),  the weak Hecke eigensheaf is actually a genuine Hecke eigensheaf. In the hypergeometric setting, however, these assumptions are satisfied only when the characters at $0$ are trivial, i.e. $\chi_i=1$ for all $i$. Thus, to treat general hypergeometric sheaves, we need the above variant of Yun's theorem.

\section{Main definitions and theorems} \label{s:main} 
In this section, we associate to every hypergeometric initial data $(\psi, \chi_1,...,\chi_n, \rho_1,...,\rho_m)$ (\S \ref{sss:initialData}), an automorphic data (Definition \ref{d:automorphicData}), and state our main results. As hypergeometric sheaves have  different ramification profiles in the tame and wild setting, we need to treat these cases separately.

\subsection{Tame hypergeometric automorphic data}\label{ss:tame data} 
Recall that $k$ is a finite field. Let $n$ be a positive integer, and $\chi_1,...,\chi_n$ and $\rho_1,...,\rho_n$ multiplicative characters $k^\times \ra \bQlt$. 

\begin{defe} \label{d:tamedata}
The \emph{tame hypergeometric automorphic data} is defined by $S:=\{0,1,\infty\}\subset \bP^1$ and 
\[
(K_x, \gamma_x):=
\begin{cases} 
(I^\opp, \cL_{\overline{\chi}}) & x=0; \\ 
(\mathtt{Q}, \bQl) & x=1;\\
(I, \cL_\rho) &  x=\infty.
\end{cases} 
\]
\end{defe} 

Here, $I\subset G(\cO_\infty)$ is the standard Iwahori subgroup of $G$, and $I(1)$ is its first Moy--Prasad subgroup. The character  $\rho=(\rho_1,...,\rho_n): I\ra \bQlt$ is defined via the composition
\[
I\ra I/I(1) \simeq T\simeq (k^\times)^n\ra \bQlt. 
\]
Let $\cL_\rho$ denote the character sheaf on $I$ whose Frobenius trace is $\rho$. The group $I^\opp \subset G(\cO_0)$ is the Iwahori opposite to $I$. As above, $\chi=(\chi_1,...,\chi_n)$ defines a character of $I^\opp$. We let $\cL_{\overline{\chi}}$ denote the character sheaf on $I^\opp$ whose trace function is $\ochi=\chi^{-1}$. The group $\mathtt{Q}$ was defined in \S\ref{sss:tameIntro} as the preimage of the mirabolic  of $G$. Note that $Z\mathtt{Q}$ is a parahoric subgroup of $G$.\footnote{The barycentre $x$ of the facet corresponding to $Z\mathtt{Q}$ is determined by $\alpha_0(x)=\alpha_{n-1}(x)=1/2$ and $\alpha_1(x)=\cdots=\alpha_{n-2}(x)=0$.} It is technically convenient for us to remove the centre because we want our automorphic data to be \emph{strictly} rigid.

\subsubsection{Integral models} 
To the above automorphic data, one associates the integral models $\cG$ and $\cG'$. The former group scheme was defined \eqref{eq:tameGroup}. The latter is defined by 
\[
\cG'(\cO_x) =
\begin{cases}
I^\opp & x=0; \\ 
\mathtt{Q} & x=1; \\ 
I & x=\infty; \\
G(\cO_x) & \textrm{otherwise}. 
\end{cases} 
\]
One readily verifies $\sum_{x\in S} \dim(G(\cO_x)/\cG'(\cO_x))=\dim(G)$; thus, the numerical requirement for strict rigidity \eqref{eq:numerical} is satisfied.

\subsubsection{Why the mirabolic shows up?}\label{s:motivationTame} 
The following discussion is informal and not used elsewhere in the text.  Let $\sH$ be a tame hypergeometric sheaf. Let $F_1$ denote the local field at $x=1$ and 
\[
\rho: \Gal(\overline{F_1}/F_1)\ra \GL_n(\bQl),
\]
the Galois representation defined by the restriction of $\sH$ to $x=1$. Recall  that $\rho$ is tamely ramified with pseudo-reflection monodromy. This implies that $\rho$ is a direct sum of $n$-characters, all but one of which are unramified. Under the local Langlands bijection \cite{LRS}, $\rho$ is mapped to the irreducible quotient $\pi_1$ of the principal series representation associated to a character $\alpha=(\alpha_1,...,\alpha_n)$ of $T(F)=(F^\times)^n$, where $\alpha_i$ is unramified for all $i\neq n$. Using properties of smooth representations of $G(F_1)$ discussed in \cite{KS2}, one can show that $\pi_1$ has a vector fixed under $Z\mathtt{Q}$. Moreover, if the monodromy of $\sH$ at $1$ is nontrivial, then $Z\mathtt{Q}$ is the largest compact open subgroup of $G(F_1)$ with a fixed vector in $\pi_1$. Thus, it is natural to take $\mathtt{Q}$ as the level structure at $1$. The fact that the resulting automorphic data satisfies the numerical requirement for rigidity is further evidence that this is the correct choice. (Of course, the ultimate vindication is that this automorphic data is rigid and its Hecke eigenvalue is $\sH$.)

\subsection{Wild hypergeometric automorphic data} \label{ss:wild data}
Let $m$ and $n$ be non-negative integers satisfying  $0\leq m<n$. 
Let $\chi_1,...,\chi_n$ and $\rho_1,...,\rho_m$ be multiplicative characters $k^\times \ra \bQlt$. 
\begin{defe}\label{d:wilddata}
The \emph{wild hypergeometric automorphic data} is $S:=\{0,\infty\}\subset \bP^1$ and 
\[
(K_x,\gamma_x):=
\begin{cases}
(I^\opp,\cL_{\overline{\chi}}) \quad x=0;\\
(J,\cL_\mu) \quad\quad\ x=\infty.
\end{cases}
\] 
\end{defe}
The pair $(I^\opp,\cL_{\overline{\chi}})$ is the same as in the tame case. The definition of $(J, \cL_\mu)$, given below, is more subtle and involves salient features of principal gradings of $\fg$. This definition is motivated by the fact that on the punctured neighbourhood of $\infty$, we have a decomposition 
\[
\sH_\infty \simeq (\on{Kl}_d)_\infty \oplus T_{m},
\]
where $\on{Kl}_d$ denotes a wild $d:=(n-m)$-dimensional (generalised) Kloosterman sheaf and $T_m$ is a tame rank $m$ local system with monodromy given by the $\rho_j$'s (see \S \ref{s:hypergeometric} for details). 
To construct $(J,\cL_\mu)$, we first construct the automorphic data corresponding to the wild part and then add the tame data.

\subsubsection{Principal gradings} 
The torsion element $\check{\rho}_G/d\in \cA_\bQ$ defines a  principal  grading
\begin{equation} 
\fg=\bigoplus_{i\in \bZ/d\bZ} \fg_i.
\end{equation}  
Let $G_0$ denote the connected subgroup of $G$ with Lie algebra $\fg_0$. Then $G_0$ acts on $\fg_1$ by conjugation. This is the Vinberg representation. We refer the reader to Appendix \ref{s:grading} for recollections on principal gradings and the corresponding Vinberg representations.

\subsubsection{The functional}
The definition of the automorphic data at $\infty$ depends on the choice of an appropriate functional on $\fg_1$. Characterising all functionals which give rise to the correct rigid automorphic data is subtle (see below). For our purposes, it will be sufficient to work with  
\begin{equation} \label{eq:phisp}
\phi=\phi_\spp:=
\begin{cases} 
E_{11}^* & \textrm{$d=1$}; \\
E_{12}^* + E_{23}^* + ...+ E_{d-1, d}^* + E_{d,1}^* & \textrm{$d\in \{2,...,n\}$}. 
\end{cases} 
\end{equation} 
If $d=n$, we obtain an affine generic functional in the sense of \cite[\S 1.3]{HNY}. In general, the $G_0$-orbit of $\phi$ is closed. In fact, if $d>n/2$, then $\phi$ is the unique, up to $G_0$-conjugacy and scalar multiplication, element of $\fg_1^*$ with a closed orbit (this fails for $d\leq n/2$). 
Note, however, that if $d<n$, then $\phi$ is \emph{not} stable because the stabiliser of $\phi$ is not finite. This is a key difference between our work and much of the related literature \cite{HNY, RY, YunEpipelagic, Chen}. In fact, to define the automorphic data at $\infty$, we need to bring the stabiliser of $\phi$ into play.

\subsubsection{Parahoric and Moy--Prasad subgroups} The functional $\phi$ takes care of the wild part of the data at $\infty$. To add the tame data (i.e., the characters $\rho_j$), we need to extend $\phi$ in an appropriate manner. Let $P\subset G(F_\infty)$ denote the parahoric subgroup associated to $\check{\rho}/d$, and let  $P\supset P(1)\supset P(2)\supset \cdots $ denote its Moy--Prasad filtration. Let $\fp\supset \fp(1)\supset \fp(2)\supset \cdots$ denote the corresponding Lie algebras. Let 
\begin{equation} 
L:=P/P(1)\qquad \textrm{and} \qquad V:=P(1)/P(2).
\end{equation} 
We have canonical isomorphisms 
$L\simeq G_0$ and $V\simeq \fp(1)/\fp(2)\simeq \fg_1$. 
Thus, we can view $\phi$ as a functional on $V$ or as a homomorphism $P(1)/P(2)\ra \bGa$.

\subsubsection{The subgroup $J$} 
To define $J$, we need to study the stabiliser of $\phi\in V^*$ under the action of $L$ on $V^*$. As in the tame case, it will be convenient to remove the centre. For $\phi$ specified in \eqref{eq:phisp}, we consider subgroups $L_\phi$ and $B_\phi$ of the stabiliser of $\phi$ in $L$ whose  image under $L\simeq G_0$ are as follows:
\begin{equation}\label{eq:Borel of Stab}
L_{\phi}\simeq\begin{bmatrix} 
\on{Id}_d & 0 \\
0 & \GL_{m}
\end{bmatrix}\cap G_0,\qquad \textrm{and} \qquad 
B_{\phi}\simeq\begin{bmatrix} 
\on{Id}_d & 0 \\
0 &B_m
\end{bmatrix}\cap G_0.
\end{equation} 
Here $B_m\subset \GL_m$ is the subgroup of upper triangular matrices. Thus, $\Stab_L(\phi)= ZL_\phi$, and $B_\phi$ is a Borel subgroup of $L_\phi$. Now define
\begin{equation} \label{eq:J}
J:=B_\phi P(1) \subseteq LP(1)=P.
\end{equation} 

\begin{exam}
\begin{enumerate} 
\item[(i)]
If $m=0$, then $L_\phi$ and $B_\phi$ are trivial and $P$ is the Iwahori; thus, $J=I(1)$ and we recover the Kloosterman automorphic data defined in \cite{HNY}. 
\item[(ii)] Suppose $m=1$ and let $\cP$ denote the maximal ideal of $\cO=\cO_\infty$. Then $B_\phi=\diag(1,1,...,1,*)$, and $P$ and $J$ are as follows:
\[
P = \begin{pmatrix} 
 \cO^\times & \cO & ... & \cO & \cP^{-1} \\
 \cP & \cO^\times & ... & \cO & \cO \\
 ... & ... & ... & ... & ... \\
 \cP & \cP & ... & \cO^\times & \cO \\
 \cP & \cP & ... & \cP & \cO^\times\\
\end{pmatrix},
\quad \quad 
J= \begin{pmatrix} 
 1+\cP & \cO & ... & \cO & \cO \\
 \cP& 1+\cP & ... & \cO & \cO \\
 ... & ... & ... & ... & ... \\
 \cP & \cP & ... & 1+\cP & \cO \\
 \cP^2 & \cP & ... & \cP & \cO^\times\\
\end{pmatrix}.
\]
\end{enumerate} 
\end{exam}

\subsubsection{The character $\mu$} 
We are finally ready to define the character $\mu:J\ra \bQlt$. First, note that $\rho=(\rho_1,\cdots, \rho_m)$ defines a character of $B_\phi$ via the composition 
\[
B_\phi \ra B_\phi/[B_\phi, B_\phi] \simeq (k^\times)^{m} \rar{\rho} \bQlt.
\] 
On the other hand, we have the character $\psi \phi: P(1)\ra \bQlt$. Now define
\begin{equation}\label{eq:mu} 
\mu=\mu(\psi, \phi, \rho): J=B_\phi P(1)\ra \bQlt, \qquad \mu(bp):=\rho(b)\psi\phi(p),\qquad b\in B_\phi,\,\, p\in P(1). 
\end{equation} 
One readily verifies that $\mu$ is indeed a character of $J$. 
Let $\cL_\mu$ be the character sheaf whose Frobenius trace is $\mu$. This concludes the definition of the wild hypergeometric automorphic data.

\subsubsection{The associated integral models}
To the above automorphic data, one associates the integral models $\cG$ and $\cG'$. The former group scheme was defined in \eqref{eq:wildGroup}. The latter is defined by 
\[
\cG'(\cO_x) =
\begin{cases}
I^\opp & x=0; \\ 
J & x=\infty; \\ 
G(\cO_x) & \textrm{otherwise}. 
\end{cases} 
\]
Let $J'\subset G(\cO_\infty)$ be a conjugate of $J$. As shown in Lemma \ref{l:numericalJ} 
$\dim(G(\cO_\infty)/J')=\dim(B)$. 
Thus, the numerical requirement for strict rigidity \eqref{eq:numerical} is satisfied.

\subsection{Main results} \label{ss:main}
As noted in \S \ref{s:hypergeometric}, to every hypergeometric initial data $(\psi, \chi_1,...,\chi_n, \rho_1,...,\rho_m)$, Katz associated the hypergeometric sheaf $\sH=\sH(\psi, \chi_1,...,\chi_n, \rho_1,...,\rho_m)$. Moreover, he proved that $\sH$ is irreducible  if and only if $\chi_i$'s and $\rho_j$'s are disjoint; i.e., $\chi_i\neq \rho_j$ for all $i,j$. In this case, he also proved that $\sH$ is rigid.  

In the previous two subsections, we have associated to the initial data, the corresponding hypergeometric automorphic data (Definitions \eqref{d:tamedata} and \eqref{d:wilddata}).

\begin{thm}\label{t:rigidity} If $\chi_i$'s and $\rho_j$'s are disjoint, then 
the hypergeometric automorphic data is strictly rigid (Definition \ref{d:rigidity})
\end{thm} 

This theorem is proved in \S \ref{s:tameRigidity} (resp. \S \ref{s:wildRigidity}) for the tame (resp. wild) case.

\subsubsection{} 
It follows from Theorems \ref{t:Yun} and \ref{t:rigidity} that whenever $\chi_i$'s and $\rho_j$'s are disjoint, we have a Hecke eigensheaf $\sA=\sA(\psi, \chi_1,...,\chi_n, \rho_1,...,\rho_m)$ on $\Bun_\cG$. It remains to identify its Hecke eigenvalue:

\begin{thm}\label{t:eigenvalue} 
The Hecke eigenvalue of $\sA$ is, after base change to $\overline{k} $, isomorphic to $\sH\otimes_k \overline{k}$. 
\end{thm}

This theorem is proved in \S \ref{s:tame eigenvalue} (resp. \S \ref{s:wild eigenvalue}) for the tame (resp. wild) case.

\subsection{Aside: Hypergeometric automorphic representations} \label{sss:representations} 
In this subsection, we explain what the above considerations tell us about hypergeometric automorphic representations (\S \ref{sss:hypergeomtricRep}). We do not formulate precise results here since we have chosen (for convenience and brevity) not to deal with the full package of Yun's rigid automorphic data (cf. the beginning of \S \ref{s:rigid}). The discussion in this subsection is informal and not used elsewhere in the text.

\subsubsection{Tame case} 
Our results on tame hypergeometric Hecke eigensheaves indicate that, up to unramified twists and central characters,  there is a unique automorphic representation $\displaystyle \pi=\otimes'_{x\in X} \pi_x$ of $G(\bA)$ satisfying: 
\begin{itemize} 
\item[(i)] $\pi_x$ is unramified for all $x\in \bP^1-\{0,1,\infty\}$; 
\item[(ii)] $\pi_0$ has a $(I^\opp, \chi)$-fixed vector; 
\item[(iii)] $\pi_1$ has a $\mathtt{Q}$-fixed vector;
\item[(iv)] $\pi_\infty$ has a $(I,\rho)$-fixed vector. 
\end{itemize}

\subsubsection{Wild case} 
Similarly, up to unramified twists and central characters,  there should be a unique automorphic representation $\displaystyle \pi=\otimes'_{x\in X} \pi_x$  satisfying: 
\begin{itemize} 
\item[(i)] $\pi_x$ is unramified for all $x\in \bP^1-\{0,\infty\}$; 
\item[(ii)] $\pi_0$ has a $(I^\opp, \chi)$-fixed vector; 
\item[(iii)] $\pi_\infty$ has a $(J,\mu)$-fixed vector. 
\end{itemize} 
Note that the representation $\pi_\infty$ is supercuspidal if and only if $d=n$; i.e., if and only if we are in the Kloosterman setting, for otherwise, the local Langlands parameter is reducible (\S \ref{sss:wildHypergeometric}).\footnote{In particular, aside from the Kloosterman case, hypergeometric automorphic representations are not of the type studied in \cite{TemplierSawin}.}

\section{Rigidity in the tame case}\label{s:tameRigidity}
The goal of this section is to prove Theorem \ref{t:rigidity} in the tame case.  We start by parameterising the objects of $\Bun_{\cG'}$ and describing their automorphisms.

\subsection{Parameterisation of $\cG'$-bundles}\label{s:tame bundles}

\subsubsection{An auxiliary moduli stack} \label{s:auxiliary}
Let $\Bun(I^{\opp},I)$
denote moduli stack of rank $n$ vector bundles on $\bP^1$ with $I^\opp$-level structure at $0$ and $I$-level structure at $\infty$. Let $I^-:=I^\opp \cap G(k[s,s^{-1}])$. According to \cite[Proposition 1.1]{HNY}, we have 
\[
\Bun(I^{\opp},I)(k)= 
I^-\backslash G(k(\!(s)\!)/I = \bigsqcup_{\tw\in\widetilde{W}}I^-\backslash I^-\tw I/I.
\]
Here, the first equality follows from one-point uniformisation, which states that every bundle in $\Bun(I^\opp, I)$ is trivialisable on $\bP^1-\{\infty\}$. The second equality is the Birkhoff decomposition. 

\subsubsection{}
We conclude that (the isomorphism classes of) bundles in $\Bun(I^\opp, I)$ are labelled by elements of $\tW$. 
For  each $\tw \in \tW$, the automorphism group of the corresponding bundle is given by
\[
S(\tw):=\Stab_{I^-}(\tw I) = I^-\cap\tw I\tw^{-1}. 
\]
Note that  $S(\tw)\supseteq T$ with equality if and only if $\tw\in \Omega$.
Thus, the generic locus of $\Bun(I^\opp, I)$ consists of bundles labelled by $\tw\in \Omega$.

\subsubsection{} 
Recall that $\Bun_{\cG'}$ is the moduli stack of rank $n$ vector bundles on $\bP^1$ with $I^\opp$, $\mathtt{Q}$ and $I$ level structure at $0$, $1$, and $\infty$, respectively. 
The canonical map $\pi: \Bun_{\cG'}\ra \Bun(I^\opp, I)$ which forgets the level structure at $1$ is a $G/Q$-fibration.
Thus, to each bundle  $\cE\in \Bun_{\cG'}$, we can associate a pair $(\tw, g)\in \tW\times G/Q$. 
Let $\ev_1:  G(k[s,s^{-1}])\rightarrow G(k)$ denote the evaluation map sending $s$ to $1$. Let 
\begin{equation} \label{eq:Gtw}
G(\tw):=\ev_1(S(\tw)).
\end{equation}
If $\cF\in \Bun(I^\opp, I)$ is a bundle labelled by $\tw$, then  $S(\tw)$ acts on the fibre $\pi^{-1}(\cF)\simeq G/Q$ via $G(\tw)$. Thus, we obtain: 

\begin{lem}\label{l:tame G'-bundles}
The isomorphism classes of $\cG'$-bundles are in bijection with pairs $(\tw,G(\tw)g)$, where $\tw\in\tW$ and  $G(\tw)g\subset G/Q$ is a $G(\tw)$-orbit in $G/Q$.
\end{lem}

\subsection{Mirabolic flag variety} To understand automorphisms of $\cG'$-bundles, we need explicit descriptions of $G/Q$ and $G(\tw)$. 
Note that $ZQ$ is a (maximal) parabolic subgroup of $G$; thus, the Bruhat decomposition implies 
\[
G/Q=
 \bigsqcup_{i=1}^n B w_iQ/Q, 
\]
where $w_i\in W$ is the transposition $(in)\in S_n$. Note that the transpositions $(in)$, $i\in \{1,2,...,n\}$, are representatives for $W(G)/W(ZQ)=S_n/S_{n-1}$.

\subsubsection{Bruhat cells} 
Let $X_i:=Bw_iQ/Q$ denote the Bruhat cell associated to $w_i$. 
The torus $T$ acts on $G/Q$ by left multiplication, preserving each Bruhat cell. 
The action of $T$ on the cell $X_1$ has a unique open dense orbit $\mathring{X}_1$, which can be explicitly described as follows.  
First, observe that
\[
X_1=T (\prod_{j=2}^n U_{\alpha_{1j}}) w_1 Q/Q.
\]
Note that $U_{\alpha_{1j}}$'s are root subgroups in the quotient $U/(U\cap w_1Qw_1^{-1})$. For each $j\in \{2,...,n\}$, let $X_{1j}$ denote the closed subscheme of $X_1$ consisting of those elements whose $U_{\alpha_{1j}}$ component is trivial. 
Let $\mathring{X}_1:= X_1-\bigcup_{j=2}^n X_{1j}$. Then $\mathring{X}$ is an open dense subvariety of $X$ and $\mathring{X}_1=T\mathring{g}$, where 
\begin{equation}\label{eq:generic g}
\mathring{g}: = \exp \big(\sum_{j=2}^n E_{1j}\big) w_1Q/Q. 
\end{equation}

\subsubsection{Generic locus}
We think of $\mathring{X}_1$ as the generic locus of $X_1$. This terminology is further justified by Lemma \ref{l:Stab_T of G/Q}.(iii) below. 
In view of \S \ref{s:auxiliary}, it is natural to expect that the generic locus of $\Bun_{\cG'}$ consists of bundles labelled by pairs $(\tw, \mathring{g})$, where $\tw\in\Omega$. We confirm this expectation in Corollary \ref{c:tame genericLocus}.

\subsubsection{Stabilisers} 
We record some basic facts about the action of $G$ on $G/Q$. The proofs are direct computations and omitted. Recall the subgroup $T_j\subseteq T$ defined in \eqref{eq:Tj}. 

\begin{lem}\label{l:Stab_T of G/Q}
\begin{enumerate} 
\item[(i)] Let $i\in \{2,...,n\}$ and $g\in X_i$, then $T_1\subseteq \Stab_T(g)$. 
\item[(ii)] Let $j\in \{2,...,n\}$ and $g\in X_{1j}$, then $T_j\subseteq \Stab_T(g)$. 
\item[(iii)]  $\Stab_T(\mathring{g})=\{1\}$. 
\end{enumerate} 
\end{lem} 

Next, let $\beta=\alpha_{pq}$ be a root of $G$. Let $Y_\beta$ be the subscheme of $TU_\beta$ consisting of elements of the form $tu(t)$, where $t=\diag(1,...,t_q, ..,1)\in T_q$ and
\begin{equation}\label{eq:tame u(t)}
u(t):=
\begin{cases} 
\exp((1-t_q) E_{pq}) & p\neq 1\neq q;\\
\exp((t_q-1)E_{pq}) & \textrm{otherwise}. 
\end{cases} 
\end{equation} 
Note that $Y_{\beta}$ is a variety isomorphic to $\bA^1-\{0\}$; however, it is not a group. 

\begin{lem}\label{l:1-dim stab of G/Q}
$Y_{\beta}$ is a subscheme of the identity component of   $\Stab_{TU_\beta} (\mathring{g})$. 
\end{lem}

\subsection{The group $G(\tw)$.} 
Recall that $G(\tw)$ is the image of $S(\tw)$ under evaluation map $\ev_1$ \eqref{eq:Gtw}.
It is, therefore, generated by $T$ together with certain root subgroups, where the roots are the image of affine roots in $S(\tw)$ under $\ev_1$. Thus, we obtain:  
\begin{lem}\label{l:image group}
$\Phi(G(\tw))=\varnothing$ if and only if $\tw\in \Omega$, in which case $G(\tw)=T$. 
\end{lem}

\begin{proof}
If $\tw\in\Omega$, then $S(\tw)=I^-\cap I=T=G(\tw)$ and $\Phi(G(\tw))=\varnothing$. For the converse, suppose $\tw\notin\Omega$. Then there exists at least one simple affine root $\alpha_i\in\Delta^\aff(G)=\{\alpha_0,\alpha_1,...,\alpha_{n-1}\}$, such that $\tw\alpha_i<0$. Since $S(\tw)=I^-\cap\tw I\tw^{-1}$, this means $\tw\alpha_i\in\Phi^\aff(S(\tw))$. Let us write $\tw\alpha_i=\alpha+m$, where $\alpha\in\Phi(G)$ and $m\in\bZ$. Then $\ev_1(U_{\tw\alpha_i})=U_\alpha\subseteq G(\tw)$; thus, $\alpha \in \Phi(G(\tw))$; in particular, $\Phi(G(\tw))$ is non-empty. 
\end{proof}

\subsection{Automorphisms of $\cG'$-bundles} 
Let $\cE\in \Bun_{\cG'}$ be a bundle associated to a pair $(\tw,g)\in \tW\times G/Q$.  Then
\begin{equation} 
\Aut_{\cG'}(\cE)\simeq  \ev_1^{-1}(\Stab_{G(\tw)}(g)).
\end{equation} 
 Recall that $G(\tw)=\ev_1(S(\tw))$. Thus, for each $\beta\in \Phi(G(\tw))$, there exists $\tilde{\beta}\in \Phi^\aff(\Stab(\tw))$ such that that $\beta$ is the finite part of $\tilde{\beta}$. We therefore have an isomorphism $\ev_1(TU_{\tilde{\beta}})\simeq TU_\beta$. Let $Y_{\tilde{\beta}} \subset TU_{\tilde{\beta}}$ be the inverse of $Y_\beta$ under this isomorphism. 

\begin{prop} \label{p:tameAut} \mbox{}
\begin{enumerate} 
\item[(i)] If $\tw\in\tW$ and $g\in X_i$, $2\leq i\leq n$, then $T_1\subseteq (\Aut_{\cG'}(\cE))^\circ$. 
\item[(ii)] If $\tw\in\tW$ and $g\in X_{1j}\subseteq X_1-\mathring{X}_1$, $2\leq j\leq n$, then $T_j\subseteq (\Aut_{\cG'}(\cE))^\circ$.
\item[(iii)] If $\tw \notin \Omega$, $g=\mathring{g}\in \mathring{X}_1$, and $\beta\in\Phi(G(\tw))$, then $ Y_{\tilde{\beta}}\subseteq (\Aut_{\cG'}(\cE))^\circ$. 
\item[(iv)] If $\tw\in \Omega$ and $g\in \mathring{X}_1$, then $\Aut_{\cG'}(\cE)$ is trivial.
\end{enumerate} 
\end{prop} 

\begin{proof}
Recall that for all $\tw\in \tW$, we have  $T\subseteq G(\tw)$. Thus, (i) and (ii) follows from Part (i) and (ii) of Lemma \ref{l:Stab_T of G/Q}, respectively. When $\tw\not\in\Omega$, Lemma \ref{l:image group} implies that there exists $\beta\in\Phi(G(\tw))$; thus, $TU_\beta\subseteq G(\tw)$. Therefore, (iii) follows from Lemma \ref{l:1-dim stab of G/Q}. When $\tw\in\Omega$, we have $\Stab(\tw)=I^-\cap I=T=G(\tw)$. Note that any $g\in\mathring{X}$ is in the same $T$-orbit as $\mathring{g}$. Thus, (iv) follows from Lemma \ref{l:Stab_T of G/Q}.(iii).  
\end{proof}

\subsubsection{Generic locus} 
As discussed in \S \ref{s:Kottwitz}, the Kottwitz homomorphism (which is the degree map here) defines an isomorphism $\pi_0(\Bun_{\cG'})\simeq \bZ$. Let $\tw_1$ denote the generator of $\Omega$ specified in \S \ref{s:Omega}. The above proposition immediately implies the following: 
\begin{cor}\label{c:tame genericLocus} For each $\alpha \in \bZ$, the generic locus $\mathring{\Bun_{\cG'}^\alpha}$ consists of one element, namely, the bundle labelled by $(\tw_1^\alpha, \mathring{g})$
\end{cor}

\subsection{Rigidity} 
The following is a more precise version of Theorem \ref{t:rigidity} in the tame setting. 
\begin{thm}\label{t:rigidTame} For each $\alpha\in \bZ$, the only relevant element on $\Bun_{\cG'}^\alpha$ is the generic element, i.e., the bundle labelled by $(\tw_1^\alpha, \mathring{g})$. 
\end{thm}

\subsubsection{} 
As a first step in proving this theorem, let us explain what being relevant in this context means. Let $\cE$ be a bundle associated to the pair $(\tw, g)\in \tW\times G/Q$. 
If we make the identification 
\[
\Aut_{\cG'}(\cE)=\ev_1^{-1}(G(\tw)\cap gQg^{-1})\subseteq S(\tw)=I^-\cap\tw I\tw^{-1},
\]
then the map $\mathrm{Res}$, defined in \eqref{eq:Res}, is given by 
\begin{equation}\label{eq:autoMap} 
\on{Res}:\ \Aut_{\cG'}(\cE) \ra  \Aut(\cE|_{\cO_0}) \times \Aut(\cE_{\cO|_\infty})  \simeq I^\opp \times I,\qquad h\mapsto (h,\tw^{-1}h\tw). 
\end{equation}  
By definition, $\cE$ is relevant if the pullback of the local system $\gamma_S:=\cL_{\overline{\chi}}\boxtimes\cL_\rho$ to $(\Aut_{\cG'}(\cE))^0$ is constant.

\subsubsection{} 
If $\tw\in \Omega$ and $g\in \mathring{X_1}$ (equivalently, $g$ is in the $T$-orbit of $\mathring{g}$), then $\cE$ has a trivial automorphism group and is therefore relevant. We now show that if $\tw\notin \Omega$ or $g\notin \mathring{X}_1$, then the pullback of $\gamma_S$ to the one-dimensional subscheme of $(\Aut_{\cG'}(\cE))^\circ$ given in the Proposition \ref{p:tameAut} is non-constant. We prove this by showing that the Frobenius trace function of  $\Res^*(\cL_{\overline{\chi}}\boxtimes\cL_\rho)$ over this subscheme is non-constant, i.e. not identically equal to $1$.

\subsubsection{} 
Let $\tw\in\tW$, $g\in X_i$, $2\leq i\leq n$, and $t=(t_1,1,...,1)\in T_1$. By Proposition \ref{p:tameAut}.(i), $T_1\subseteq \Aut_{\cG'}(\cE)$. The Frobenius trace of $\Res^*(\cL_{\overline{\chi}}\boxtimes\cL_\rho)$ over $T_1$ is
\[
\overline{\chi}(t)\rho(\tw^{-1}t\tw)=\chi_1^{-1}(t_1)\rho_l(t_1),
\]
where $l$ is defined by the equality $\tw^{-1}T_1\tw=T_l$. If above expression equals $1$ for all $t_1\in k^\times$, then $\chi_1=\rho_l$, which contradicts with the assumption that $\chi_i$'s and $\rho_j$'s are disjoint. 
Thus, the Frobenius trace function is non-constant.

\subsubsection{} 
Let $\tw\in\tW$, $g\in X_{1j}\subseteq X_1-\mathring{X}_1$, $2\leq j\leq n$, and $t=(1,...,1,t_j,1,...,1)\in T_j$. Analogous to the above, the Frobenius trace of the restriction of $\gamma_S$ to $T_j\subseteq \Aut_{\cG'}(\cE)$ is given by the function $\chi_j^{-1}(t_j)\rho_l(t_j)$. The assumption $\chi_j\neq \rho_l$  implies that the trace function is non-constant.

\subsubsection{}  
Let $\tw \in \tW-\Omega$ and $g\in \mathring{X}_1$. Without the loss of generality, we can take $g=\mathring{g}$. Let $\beta=\alpha_{ij}\in\Phi(G(\tw))$, $t=(1,...,1,t_j,1,...,1)\in T_j$, and $u(t)\in U_{\tilde{\beta}}$ be the preimage of the element given in \eqref{eq:tame u(t)} under isomorphism $\ev_1:TU_{\tilde{\beta}}\simeq TU_\beta$. Then $tu(t)\in Y_{\tilde{\beta}}$. The Frobenius trace of the restriction of $\gamma_S$ to $Y_{\tilde{\beta}}\subseteq \Aut_{\cG'}(\cE)$ is given by
\[
\overline{\chi}(tu(t))\rho(\tw^{-1}tu(t)\tw)=\overline{\chi}(t)\rho(\tw^{-1}t\tw)=\chi_j^{-1}(t_j)\rho_l(t_j)
\]
Again, the assumption $\chi_j\neq \rho_l$ implies that the trace function is non-constant. This concludes the proof of Theorem \ref{t:rigidTame} \qedhere.

\section{Hecke eigenvalue in the tame case}\label{s:tame eigenvalue} 
The goal of this section is to prove Theorem \ref{t:eigenvalue} in the tame case. We start by giving a description of $\Bun_{\cG}$ in terms of lattices (cf. \cite[\S 3]{HNY}).

\subsection{Alternative description of $\Bun_{\cG}$} 
Let $\Bun^+$ be the classifying stack of 
$(\cE,F^*\cE,\{v^i\}, F_*\cE,\{v_i\}, V_{n-1}, v)$, where:
\begin{itemize}
\item $\cE$ is a vector bundle of rank n on $\bP^1$;
\item $\cE=F^0\cE\supset F^1\cE\supset\cdots\supset F^n\cE=\cE(-\{0\})$ is a decreasing filtration $F^*\cE$ giving a complete flag of the fibre of $\cE$ at $0$;
\item $v^i\in F^{i-1}\cE/F^i\cE$ is a nonzero vector, $1\leq i\leq n$;
\item $\cE(-\{\infty\})=F_0\cE\subset F_1\cE\subset\cdots\subset F_n\cE=\cE$ is an increasing filtration $F_*\cE$ giving a complete flag of the fibre of $\cE$ at $\infty$;
\item $v_i\in F_i\cE/F_{i-1}\cE$ is a nonzero vector, $1\leq i\leq n$;
\item $\cE(-\{1\})\subset V_{n-1}\subset \cE$ is an increasing filtration giving a partial flag of mirabolic type at $1$, i.e., $V_{n-1}/\cE(-\{1\})$ is an $(n-1)$-dimensional subspace of $\cE/\cE(-\{1\})$;
\item $v\in\cE/V_{n-1}$ is a nonzero vector.
\end{itemize}
Above data corresponds to level structures $I^\opp(0)$, $I(\infty)$, and $\mathtt{Q}$. Thus, we obtain an isomorphism of stacks
$\Bun_\cG\simeq \Bun^+$.

\subsubsection{} 
Choosing a trivialisation of $\cE$ over $\bP^1-\{0,\infty\}$, we can rewrite the above moduli problem in terms of lattices. 
Let $\Lambda$ be the free $k[t,t^{-1}]$-module with basis ${e_1,e_2,...,e_n}$. Let $e_{i+jn}:=t^j e_i$ for $j\in\bZ$, $1\leq i\leq n$. Then $\{e_i\}_{i\in\bZ}$ is a $k$-basis of $\Lambda$. Let $R$ be a $k$-algebra. An $R[t]$-lattice in $R\otimes_k\Lambda$ is a $R[t]$-submodule $\Lambda'\subset R\otimes_k\Lambda$ such that there exists a positive integer $M$ satisfying 
\[
\text{Span}_R\{e_i|i>M\}\subset\Lambda'\subset\text{Span}_R\{e_i|i\geq-M\},
\]
and both $\Lambda'/\text{Span}_R\{e_i|i>M\}$ and $\text{Span}_R\{e_i|i\geq-M\}/\Lambda'$ are projective $R$-modules.

\subsubsection{}
Let $\widetilde{\Bun}^+$ be the stack whose $R$-points classify the data $(\Lambda^*,\{v^i\},\Lambda_*,\{v_i\},V_{n-1},v)$, where:
\begin{itemize}
\item $R\otimes_k\Lambda\supset \Lambda^0\supset \Lambda^1\supset\cdots\supset\Lambda^n=t\Lambda^0$ is a chain $\Lambda^*$ of $R[t]$-lattices such that $\Lambda^i/\Lambda^{i+1}$ is a rank one projective rank $R$-module;
\item $v^i\in\Lambda^{i-1}/\Lambda^i$, $1\leq i\leq n$ is an $R$-basis;
\item $\Lambda_0=t^{-1}\Lambda_n\subset \Lambda_1\subset\cdots\subset\Lambda_n\subset R\otimes\Lambda$ is a chain $\Lambda_*$ of $R[t^{-1}]$-lattices such that $\Lambda_i/\Lambda_{i-1}$ is rank one projective $R$-module, $1\leq i\leq n$; 
\item $v_i\in\Lambda_i/\Lambda_{i-1}$, $1\leq i\leq n$, is a an $R$-basis;
\item $(t-1)R\otimes \Lambda \subset V_{n-1}\subset R\otimes \Lambda$ a $R\otimes \Lambda$-submodule such that 
\[
V_{n-1}/(t-1)R \otimes \Lambda\subset R\otimes \Lambda/(t-1)R\otimes \Lambda
\]
is a projective rank $n-1$ $R\otimes \Lambda$-submodule;
\item $v\in R\otimes \Lambda/V_{n-1}$ a $R$-basis. 
\end{itemize}

\subsubsection{} 
The group $G(k[t,t^{-1}])$ acts on $\Lambda$, and therefore, also on $\widetilde{\Bun}^+$, giving an isomorphism
\[
\Bun^+\simeq \widetilde{\Bun}^+/G(k[t,t^{-1}]).
\]
Henceforth, we  regard $\Bun_\cG$ as the moduli of $G(k[t,t^{-1}])$-orbits of chains of lattices and vectors.

\subsubsection{}
The degree of a vector bundle can be calculated in terms of lattices as follows:
\[
\deg(\Lambda^*,\Lambda_*):=\chi_R(\iota:\Lambda^0\oplus\Lambda_0\rightarrow R\otimes\Lambda)=\text{rk}_R\ker(\iota)-\text{rk}_R\text{coker}(\iota).
\]
For each $\alpha \in \bZ$, let $\Bun^{+,\alpha}$ be the substack classifying degree $\alpha$ lattices. Then $\Bun^{+,\alpha}$'s are the components of $\Bun^+$. We now explicitly describe the open embedding of the generic locus 
\[
j_\alpha: \cO_\alpha = \mathring{\Bun^{+,\alpha}}\simeq T_0\times T_\infty \hookrightarrow \Bun^{+, \alpha}.
\]

\subsubsection{The generic locus} 
Recall that $\Lambda$ is a free $k[t,t^{-1}]$ module with basis $e_1,...,e_n$. 
Let $\star\in\Bun^{+,0}$ be the  $G(R[t,t^{-1}])$-orbit of the data $(\Lambda^*(\star),\{v^i(\star)\},\Lambda_*(\star),\{v_i(\star)\},V_{n-1}(\star), v(\star))$ where
\begin{itemize}
\item $\Lambda^i(\star)=\langle  e_{i+1}, e_{i+2}, ... \rangle \subset R\otimes \Lambda$;
\item $v^i(\star)=e_i$;
\item $\Lambda_i(\star)=\langle  ..., e_{i-1}, e_i \rangle \subset R\otimes \Lambda$; 
\item $v_i(\star)=e_i$;
\item $V_{n-1}(\star)=\langle e_1,...,e_{n-1} \rangle+(t-1)R\otimes \Lambda \subset R\otimes \Lambda$;
\item $v(\star)=e_n$.
\end{itemize}
The map $j_\alpha$ is given by 
\begin{equation} 
j_\alpha (a,b):= (\Lambda^*(\star), \{a v^i(\star)\}, \tw_1^\alpha\cdot\Lambda_*(\star), \{ \tw_1^\alpha b \cdot v_i(\star)\}, uw\cdot V_{n-1}(\star), \mathring{g}\cdot v(\star)), 
\end{equation} 
where $(a,b)\in T_0\times T_\infty$. For future use, let $\star_1:=j_1(1,1)\in \Bun^{+,1}$.

\subsection{Relevant part of the Hecke stack}
Our goal here is to explicitly describe the correspondence in \eqref{eq:relevantHecke}. 

\subsubsection{}
We want to compute $\mathring{\mathrm{GR}_{\omega_1}}$ defined in the diagram \eqref{eq:relevantHecke}. Its $R$-points are $R[t,t^{-1}]$-morphisms $M: R\otimes \Lambda \rightarrow R\otimes\Lambda$ such that $M$ is an isomorphism at all but one point $x\in\bP^1-\{0,1,\infty\}$. At $x$, $M$ gives upper modification associated to $\omega_1$. In addition, $M$ needs to satisfy $M(j_0(a,b))=\star_1$ for some $a\in T_0$ and $b\in T_\infty$. We now write down the matrix for such $M$ explicitly.

\subsubsection{}
First, we have $M(\Lambda^i(\star),\Lambda_i(\star))=(\Lambda^i(\star_1),\Lambda_i(\star_1))$. Thus, with respect to the basis ${e_1,...,e_n}$, any such $M$ takes the form
\[
M=\begin{pmatrix}
x_1 &   &      &       &ty_n\\
y_1 &x_2&      &       &    \\
&y_2&      &       &    \\
&   &\ddots&\ddots &    \\
&   &      &y_{n-1}&x_n  
\end{pmatrix}. 
\]

\subsubsection{}
Second, $M$ maps $v^i$, $v_i$ of $j_0(a,b)$ to the corresponding data for $\star_1$. If we let $a=\diag(a_1,...,a_n)\in T_0$ and $b=\diag(b_1,...,b_n)\in T_\infty$, then we obtain 
\[
M(ae_i)=M(a_ie_i)=e_i, \qquad M(be_i)=M(b_ie_i)=\tw_1 e_i=e_{i+1}.
\]
In view of the matrix form of $M$, these equations amount to 
\[
a_i=x_{i}^{-1},\qquad b_i=y_i^{-1}.
\]

\subsubsection{}
Third, $M$ maps the data $V_{n-1}$ and $v$ of $j_0(a,b)$ to that of $\star_1$, resulting in the equations 
\[
M\mathring{g}\langle e_1,...,e_{n-1}\rangle=\mathring{g}\langle e_1,...,e_{n-1}\rangle, \qquad M\mathring{g}e_n=\mathring{g}e_n. 
\]
Using the explicit matrices of $M$ and $\mathring{g}$ , the above equalities amount to
\begin{equation}
x_1-y_1=x_n-y_n=x_i+y_i=1, \quad 2\leq i\leq n-1.
\end{equation}

\subsubsection{} 
Finally, $\det M=x_1\cdots x_n-(-1)^nty_1\cdots y_n$ vanishes for exactly one $t\in\bP^1-\{0,1,\infty\}$. If
we replace $x_i$ with $-x_i$ for $1\leq i\leq n$, and replace $y_1,y_n$ with $-y_1,-y_n$, then $\det M$ vanishes at $\pi_2(M)=\prod_{i=1}^n\frac{x_i}{y_i}$. Under this substitution, $x_i=y_i-1$, $1\leq i\leq n$. Putting all this together, we obtain the description of the Hecke stack given in \S \ref{sss:HeckeTame}.

\subsection{Proof of Theorem \ref{t:eigenvalue} in the tame case} 
According to \eqref{eq:eigenvalue}, the Hecke eigenvalue is given by 
\[
E_{\on{Std}}=\pi_{2!}\pi_1^*(\cL_{\overline{\chi}}\boxtimes\cL_\rho)[n-1].
\]
Our goal now is to show that $E_{\on{Std}}$ is geometrically isomorphic to the tame hypergeometric sheaf $\sH=\sH(\psi, \chi_1,...,\chi_n, \rho_1,...,\rho_n)$. 

\subsubsection{} Since both $\sH$ and $E_{\on{Std}}$ are local systems on $\bGm$ and $\sH$ is irreducible, it is sufficient to show that the Frobenius trace functions are equal up to a nonzero scalar. Now for $a\in\bP^1-\{0,1,\infty\}$, we have 
\begin{align*}
\mathrm{tr}_{E_{\on{Std}}}(a)&= \mathrm{tr}( \pi_{2!}\pi_1^*(\cL_{\ochi}\boxtimes\cL_{\rho})[n-1])(a)\\
&=(-1)^{n-1} \sum_{\prod_{i=1}^n(y_i-1)/y_i=a}\, \, \prod_{i=1}^n\chi_i(1-y_i)\rho_i(y_i^{-1})\rho_1(-1)\rho_n(-1)\\
&=(-1)^{n-1} \sum_{\prod_{i=1}^n(1-y_i^{-1})=a}\, \, \prod_{i=1}^n\chi_i(1-y_i^{-1})\chi_i^{-1}(-y_i^{-1})\rho_i(-y_i^{-1})\prod_{i=2}^{n-1}\rho_i(-1)\\
&=\sum_{\prod_{i=1}^n(1-y_i)=a}\varepsilon\prod_{i=1}^n\chi_i(1-y_i)(\chi_i^{-1}\rho_i)(-y_i)\quad\quad (y_i\mapsto y_i^{-1}, \varepsilon:=(-1)^{n-1} \prod_{i=2}^{n-1}\rho_i(-1))\\
&=\sum_{\prod_{i=1}^n z_i=a,z_i\neq 0,1}\varepsilon\prod_{i=1}^n\chi_i(z_i)(\chi_i^{-1}\rho_i)(z_i-1) \quad\quad (z_i=1-y_i)
\end{align*}

\subsubsection{} 
On other hand, recall from \S \ref{sss:hyp sum} that the Frobenius trace of $\sH$ is given by 
\begin{align*}
\mathrm{tr}_\sH(a)&=(-1)^{n+n-1}\sum_{\prod_{i=1}^n x_i=a\prod_{i=1}^n y_i}\psi(\sum_{i=1}^n(x_i-y_i))\prod_{i=1}^n\chi_i(x_i)\rho_i(y_i^{-1})\\
&=-\sum_{\prod_{i=1}^n z_i=a,y_i\neq0}\psi(\sum_{i=1}^n y_i(z_i-1))\prod_{i=1}^n\chi_i(z_i)\chi_i(y_i)\rho_i(y_i^{-1})\quad\quad (z_i:=x_i/y_i)\\
&=-\sum_{\prod_{i=1}^n z_i=a,z_i\neq 0,1, y_i\neq0}\psi(\sum_{i=1}^n y_i(z_i-1))\prod_{i=1}^n\chi_i(z_i)(\chi_i^{-1} \rho_i)(y_i^{-1}).
\end{align*}

\subsubsection{} 
The last step follows because the assumption $\chi_i\neq \rho_i$ implies that $\sum_{y_i\neq 0} \chi_i^{-1} \rho_i(y_i^{-1})=0$. 
We therefore obtain 
\begin{align*}
\mathrm{tr}_\sH(a)&=-\sum_{\prod_{i=1}^n z_i=a,z_i\neq 0,1, y_i\neq0}\psi(\sum_{i=1}^n y_i(z_i-1))\prod_{i=1}^n\chi_i(z_i)(\chi_i^{-1} \rho_i)(y_i^{-1})\\
&=-\sum_{\prod_{i=1}^n z_i=a,z_i\neq 0,1, w_i\neq 0}\psi(\sum_{i=1}^n w_i)\prod_{i=1}^n(\chi_i\rho_i^{-1})(w_i)\chi_i(z_i)(\chi_i^{-1}\rho_i)(z_i-1)\quad\quad(w_i=y_i(z_i-1))\\
&=-\sum_{\prod_{i=1}^n z_i=a,z_i\neq 0,1}\prod_{i=1}^n(\sum_{w_i\neq 0}\psi( w_i)(\chi_i\rho_i^{-1})(w_i))\prod_{i=1}^n\chi_i(z_i)(\chi_i^{-1} \rho_i)(z_i-1)\\
&=-(\prod_{i=1}^n G(\psi,\chi_i\rho_i^{-1}))\sum_{\prod_{i=1}^n z_i=a,z_i\neq 0,1}\prod_{i=1}^n\chi_i(z_i)(\chi_i^{-1} \rho_i)(z_i-1)\\
&=-(\prod_{i=1}^n G(\psi,\chi_i\rho_i^{-1})) \varepsilon^{-1} \cdot\mathrm{tr}_{E_{\on{Std}}}(a),
\end{align*}
where  $G(\psi,\chi_i\rho_i^{-1})=\sum_{w_i\neq 0}\psi( w_i)(\chi_i\rho_i^{-1})(w_i)\neq 0$ is a Gauss sum. This concludes the proof of Theorem \ref{t:eigenvalue} in the tame case. \qed

\subsubsection{} We note that one can reformulate above proof in a purely sheaf-theoretic language (at the cost of notational inconveniences). Thus, the result also holds in characteristic $0$.

\section{Rigidity in the wild case}\label{s:wildRigidity}
The goal of this section is to prove Theorem \ref{t:rigidity} in the wild case.  We start by parameterising the objects of $\Bun_{\cG'}$ and describing their automorphisms.

\subsection{Parameterisation of $\cG'$-bundles}\label{s:wild bundles} Recall that the level structure at $\infty$ is given in terms of the parahoric $P$ associated to the barycentre $\check{\rho}/d$. Let $I'$ be an Iwahori contained in $P$. Then, 
we have a decomposition 
\[
\Bun_{\cG'}(k) = I^-\backslash G(k(\!(s)\!))/J= 
\bigcup_{\tw \in \widetilde{W}} I^-\backslash I^- \tw I'/J = 
\bigcup_{\tw\in \widetilde{W}} I^-\backslash I^-  \tw  P/J=
\bigcup_{\tw\in \widetilde{W}, \ell\in L} I^-\backslash I^-  \tw\ell J/J, 
\]
The first equality follows because every $\cG'$-bundle on $\bP^1-\{\infty\}$ is trivialisable (cf. \S \ref{s:auxiliary}), 
the second from the Birkhoff decomposition, the third from the fact that $I'\subseteq P$, and the fourth from $P=LP(1)= LJ$. 

\subsubsection{} 
The previous paragraph implies that to every pair $(\tw, \ell) \in \tW\times L$, we can associate an element $\cE\in \Bun_{\cG'}(k)$. Moreover, if we let $x=\tw \ell$, then
\[
\Aut_{\cG'}(\cE) \simeq \Stab_{I^-}(x J) = J\cap x^{-1}I^- x.
\]
Clearly, there can be two different elements of $\tW\times L$ mapping to the same element of $\Bun_{\cG'}(k)$. 
Our next goal is to formulate a refinement of the above parametrisation using a convenient subset of $\tW\times L$. To this end, we introduce an auxiliary subgroup $\Upsilon\subseteq L$.

\subsubsection{The subgroup $\Upsilon$} 
Let $B_L$ be the  Borel subgroup of $L$ and $U_L$ its unipotent radical. Recall that $J=B_\phi P(1)$, where $B_\phi\subseteq B_L$ is a Borel of $L_\phi\subseteq L$. Let $T_\phi$ and $U_\phi$ be, respectively, the maximal torus and unipotent radical of $B_\phi$. We can see from \eqref{eq:Borel of Stab} that $U_\phi$ is generated by root subgroups. We can therefore define a complement of $U_\phi$ in $U_L$ as follows. 
Let $\Upsilon\subseteq U_L$ be the subscheme defined by 
\begin{equation} 
\Upsilon:=\prod_{\alpha\in \Phi(U_L)- \Phi(U_\phi)} U_{\alpha}. 
\end{equation} 
Explicitly, we have: 
\begin{align*}
\Phi(\Upsilon)&=\{\alpha_{ij}|\ 1\leq i\leq d,\ i+d\leq j\leq n,\  j\equiv i\!\!\!\mod d   \};   \\
\Phi(U_\phi)&=\{\alpha_{ij}|\ d+1\leq i\leq n,\ i+d\leq j\leq n,\  j\equiv i\!\!\!\mod d   \}. 
\end{align*}
From the above discussions, one readily verifies the following:

\begin{lem} \label{l:U_phi} 
The scheme $\Upsilon$ is a commutative normal subgroup of $U_L$ and $U_L=U_\phi \ltimes \Upsilon$. Moreover, given $\beta=\alpha_{ij} \in \Phi(U_\phi)$, there exists a unique root $\alpha\in \Phi(\Upsilon)$ such that $\alpha+\beta\in \Phi(G)$. Explicitly,  
$\alpha = \alpha_{pi}$, where $p\in \{1,2,..,d\}$ is determined by the requirement $p\equiv i\!\mod d$. 
In this case, $\alpha+\beta=\alpha_{pj} \in \Phi(G)$. 
\end{lem}

\subsubsection{Refined parameterisation} 
\begin{prop}\label{p:wild G'-bundles}
Every element of $\Bun_{\cG'}(k)$ can be represented by a double coset $I^- xJ$, where $x=\tw\ell$, and $\tw$ and $\ell$ satisfy:
\begin{enumerate}
\item[(i)] $\tw\alpha>0$ for all $\alpha \in \Delta(L)$; 
\item[(ii)] $\ell=wu$, where $w\in W_L$ and $u\in \Upsilon$.
\end{enumerate} 
Moreover, if we write that $\displaystyle u=\prod_{\alpha \in \Phi(\Upsilon)} u_\alpha$, then we can assume $u_\alpha$ is either $1$ or $\exp(E_\alpha)$. 
\end{prop}

\subsubsection{An auxiliary lemma}
To prove this proposition, we need a lemma. Let $B_L^-\subseteq L$ be the Borel opposite to $B_L$.  
\begin{lem}\label{l:intersection=B_L}
If $\tw\alpha>0$ for all $\alpha \in \Delta(L)$, then  $L\cap\tw^{-1}I^-\tw=B_L^-$. 
\end{lem} 
\begin{proof}
The assumption $\tw\alpha>0$ implies $\tw(-\alpha)<0$; thus, the one parameter subgroup $U_{-\alpha}$ is in fact a subgroup of  $\tw^{-1}I^-\tw$ for all $\alpha\in\Delta(L)\subset \Phi^\aff(G)$. These one parameter subgroups generate the unipotent radical $U_L^-$ of $B_L^-$. On the other hand, $T\subset\tw^{-1}I^-\tw$; thus, 
$B_L^-=TU_L^-\subseteq L\cap\tw^{-1}I^-\tw$.  

For the reverse inclusion, consider the Bruhat decomposition $ L=\bigsqcup_{w\in W_L}B_L^-wB_L^-$. 
Since $B_L^-$ is already a subgroup of $\tw^{-1}I^-\tw$, we obtain
\[
B_L^-wB_L^-\cap\tw^{-1}I^-\tw\neq\emptyset\Leftrightarrow B_L^-wB_L^-\subseteq\tw^{-1}I^-\tw\Leftrightarrow w\in\tw^{-1}I^-\tw\Leftrightarrow \tw w\tw^{-1}\in I^-.
\]
Now the Birkhoff decomposition of $G(k(\!(s)\!))$ implies that $\tw w\tw^{-1}\in I^-$ if and only if $w=1$. It follows that 
$L\cap\tw^{-1}I^-\tw\subseteq B_L^-$. 
\end{proof}

\subsubsection{Proof of Proposition \ref{p:wild G'-bundles}} 
Let $\cE\in \Bun_{\cG'}(k)$ and represent $\cE$ by a pair $(\tw, \ell)\in \tW\times L$. For any $w$ in the Weyl group $W_L$ of $L$, we may replace $\tw$ and $\ell$ with, respectively, $\tw w^{-1}$ and $w\ell$, without changing the double coset $I^-\tw\ell J$. Thus, we may assume $\tw\alpha>0$ for all $\alpha \in \Delta(L)$. 

Next, using the Bruhat decomposition of $L$, we can write 
\[
\ell=bwu,\qquad b\in B_L^-,\qquad u\in U_L, \qquad w\in W_L.
\]
The previous lemma implies that $\tw b\tw^{-1}\in I^-$. Thus, we may assume $b=1$ without changing the double coset $I^-\tw\ell J$. Also, since $J=B_\phi P(1)$ and $B_\phi=T_\phi U_\phi$, we may assume that $u\in \Upsilon$. Replacing $u$ with its $T_\phi$-conjugation would also not change the double coset. This concludes the proof of the proposition.  \qed

\subsubsection{A lemma for future use} 
\begin{lem}\label{l:Aut in B_phi}
Assume $x=\tw\ell$ where $\tw$ and $\ell=wu$ are as in Proposition \ref{p:wild G'-bundles}. Then 
\[
B_\phi \cap x^{-1} I^- x = B_{\phi}\cap\ell^{-1}B_L^-\ell=\{tv\in T_\phi U_{\phi}\mid (t^{-1}ut)vu^{-1}\in U_L\cap w^{-1}U_L^-w \}.
\]
\end{lem} 
\begin{proof} 
The first equality is immediate from Lemma \ref{l:intersection=B_L}. For the second equality, we have
\[
\begin{split}
B_{\phi}\cap\ell^{-1}B_L^-\ell&=\{tv\in T_{\phi} U_\phi \mid wutvu^{-1}w^{-1}\in B_L^- \}\\
&=\{tv\in T_{\phi} U_\phi\mid (wtw^{-1})w(t^{-1}utvu^{-1})w^{-1}\in B_L^-=TU_L^- \} \\
&=\{tv\in T_{\phi} U_\phi\mid (t^{-1}ut)vu^{-1}\in U_L\cap w^{-1}U_L^-w \}.
\end{split}
\]
\end{proof}

\subsection{Generic part of $\Bun_{\cG'}$} \label{ss:genericWild}
In this subsection, we find a natural candidate for the generic locus of $\Bun_{\cG'}$.

\subsubsection{The generic part of $\Upsilon$} 
With respect to the adjoint action of $T_\phi$ on $\Upsilon$, there is a unique open dense orbit $\mathring{\Upsilon}$ consisting of the elements 
in $\Upsilon$ whose $U_\alpha$ component is nontrivial for all $\alpha\in \Phi(\Upsilon)$. Thus, $\mathring{\Upsilon}=\Ad_{T_\phi}(\mathring{u})$, where 
\begin{equation}
\mathring{u}:=\exp \left(\sum_{\alpha\in \Phi(\Upsilon)} E_\alpha\right). 
\end{equation} 
We think of $\mathring{\Upsilon}$ as the generic locus of $\Upsilon$. 

\subsubsection{Conjugating $J$ into $G(\cO_\infty)$}  
To find the generic locus of $\Bun_{\cG'}$, it is convenient to conjugate $P$ into $G(\cO_\infty)$. 
In \S \ref{ss:fundamentalAlcove}, we construct an element $\tw_d \in \tW$ which conjugates $P$ to a parahoric subgroup 
 $P'\subset G(\cO_\infty)$.
Let $J'$, $U'$, $\Upsilon'$, etc. denote the conjugates of $J$, $U$, $\Upsilon$, etc..

\subsubsection{Generic part of $\Bun(I^\opp, J')$} Let $\Bun(I^\opp, J')$ denote the moduli stack of $G$-bundles with $I^\opp$-level structure at $0$ and $J'$-level structure at $\infty$. As $J$ and $J'$ are conjugate, $\Bun(I^\opp, J')$ is isomorphic to $\Bun_{\cG'}$. Note that the standard Iwahori has a decomposition $I= T\Upsilon' J'$. Thus, analogous to \S \ref{s:wild bundles}, we have:   
\[
\Bun(I^\opp, J')(k)
=I^-\backslash G(k(\!(s)\!))/J'
=\bigsqcup_{\tw\in \widetilde{W}} I^-\backslash I^-\tw I/J' 
=\bigsqcup_{\tw\in \widetilde{W}} I^-\backslash I^-\tw \Upsilon' J'/J'. 
\]

As $J'$ contains $I$, we have a forgetful map $\Bun(I^\opp, J')\ra \Bun(I^\opp, I)$. As discussed in \S \ref{s:auxiliary}, the generic locus of $\Bun(I^\opp, I)$ consists of bundles labelled by $\tw\in \Omega$. Thus, it is natural to expect that the generic locus of $\Bun(I^\opp, J')$ consists of bundles labelled by pairs $(\tw, g)$, where $\tw\in \Omega$ and $g\in \mathring{\Upsilon}'$.

\subsubsection{} 
Returning to $\Bun_{\cG'}$, in view of above considerations, it is natural to expect that the generic locus of $\Bun_{\cG'}$ consists of bundles labelled by $(\tw, g)$, where $\tw\in \Omega \tw_d$ and $g\in \mathring{\Upsilon}$. We confirm this expectation in the next subsection.

\subsection{Automorphisms of $\cG'$-bundles}\label{ss:automorphism} 
\begin{prop}\label{p:wildAut}
Let $\cE\in \Bun_{\cG'}$ be a bundle associated to $I^-xJ$, where $x=\tw\ell$, and $\tw$ and $\ell=wu$ are as in Proposition \ref{p:wild G'-bundles}. Then 
\begin{enumerate}
\item[(i)] If $u\notin \mathring{\Upsilon}$, then $T_j\subseteq (\Aut_{\cG'}(\cE))^\circ$ for some $j\in \{d+1, d+2,..., n\}$. 
\item[(ii)] If $u\in \mathring{\Upsilon}$ and $w\neq 1$,  let $\beta=\alpha_{ij}$ be a root in $\Phi(U_L\cap w^{-1}U_L^-w)$. Then 
\begin{itemize}
\item[(a)] If $\beta\in \Phi(\Upsilon)$, then $T_j\subseteq (\Aut_{\cG'}(\cE))^\circ$. 
\item[(b)] If $\beta\in \Phi(U_\phi)$, then there exists a one-dimensional subscheme $Y_\beta\subset (\Aut_{\cG'}(\cE))^\circ$ of the form $Y_\beta=\{tv(t)\in B_\phi\,|\, t\in T_j, v(t)\in U_\beta\}$. 
\end{itemize} 
\item[(iii)] If $u\in \mathring{\Upsilon}$, $w=1$, and $\tw\notin \Omega \tw_d$, then for some lowest weight $\delta$ of the $L$-module $V$, we have  $\ell^{-1} U_{\delta}\ell\subseteq (\Aut_{\cG'}(\cE))^\circ$. 
\item[(iv)] If $u\in \mathring{\Upsilon}$, $w=1$, and $\tw\in \Omega \tw_d$, then $\Aut_{\cG'}(\cE)$ is trivial. 
\end{enumerate} 
\end{prop} 

We now discuss the proof of this proposition. 

\subsubsection{} 
For (i), write $u=\exp(\sum_{\alpha\in\Phi(\Upsilon)}\lambda_{\alpha}E_{\alpha})$.  By assumption, $\lambda_\alpha=0$ for some $\alpha=\alpha_{ij}\in \Phi(\Upsilon)$. One easily verifies that $T_j$ commutes with $u$; thus, by Lemma \ref{l:Aut in B_phi},
\[
T_j\subseteq (B_\phi \cap \ell^{-1} B_L^- \ell)^\circ \subseteq (J\cap x^{-1} I^- x)^\circ=(\Aut_{\cG'}(\cE))^\circ.  
\]
\subsubsection{} 
For part (ii).(a), let $t:=\mathrm{diag}(1,...,t_j,...,1)\in T_j$. We claim that $t\in B_\phi \cap \ell^{-1} B_L^- \ell$. To see this, note that $t^{-1}ut u^{-1} \in U_\beta\subset U_L\cap w^{-1} U_L^- w$, where the first inclusion follows from the fact that $\alpha(t)=1$ for all $\alpha\in \Phi(\Upsilon)-\{\beta\}$. As above, the result follows from Lemma \ref{l:Aut in B_phi}. 
\subsubsection{} 	
For part (ii).(b), note that any $u\in\mathring{\Upsilon}$ is $T_\phi$-conjugate to $\mathring{u}$; thus, we may assume $u=\mathring{u}$. Let  $t=\diag(1,...,t_j,..,1)\in T_j$ and $v=v(t):=\exp((1-t_j)E_\beta)$. 

Claim:  $tv\in B_\phi\cap\ell^{-1}B_L^- \ell$. 

In view of Lemma \ref{l:Aut in B_phi}, it is sufficient to show that 
$t^{-1}ut=vuv^{-1}$. To see the latter equality, note that by Lemma \ref{l:U_phi}, there exists a unique root $\alpha=\alpha_{pi}\in\Upsilon$ such that $\alpha+\beta=\alpha_{pj}\in \Phi(G)$. Thus, if we let $\log(u):=\sum_{\alpha\in\Phi(\Upsilon)}E_\alpha$, then we obtain 
\begin{align*} 
vuv^{-1} &= \exp(v \log(u) v^{-1}) = \exp(\log(u) + (1-t_j)[E_\beta, \log(u)]) \\
&= \exp(\log(u)-(1-t_j)E_{pj}) = \exp(\sum_{\alpha\in \Phi(\Upsilon)} E_\alpha - (1-t_j)E_{pj}) = t^{-1}u t, 
\end{align*} 
establishing the claim. 
(Note that higher order terms in the expansion of $v\log(u)v^{-1}$ vanish because $2\beta+\alpha$ is not a root.)

To conclude the proof of (ii).(b), let $Y_\beta:=\{t v(t)\, | \, t\in T_j\}$. Then $Y_\beta$ is a one-dimensional subscheme of $T_j U_\beta$ and the above claim implies 
\[
Y_\beta \subseteq (B_\phi \cap \ell^{-1} B_L^- \ell)^\circ \subseteq (J\cap x^{-1} I^- x)^\circ. 
\]
\subsubsection{} 
For (iii), recall that by assumption, $\tw\alpha>0$ for all $\alpha \in \Delta(L)$. Thus, if $\tw\notin \Omega\tw_d$, then by Proposition \ref{p:simple affine roots}, there exist lowest weight $\delta\in \wt^-(V)$ such that $\tw\delta<0$. Therefore, we obtain 
\[
\tw U_{\delta} \tw^{-1} \subseteq \tw P(1) \tw^{-1} \cap I^- \implies \ell^{-1} U_\delta \ell \subseteq (P(1) \cap x^{-1} I^- x)^\circ \subseteq (J\cap x^{-1} I^- x)^\circ. 
\]
\subsubsection{} 
Finally, for (iv), first observe that $P\cap x^{-1}I^-x=\ell^{-1}(P\cap\tw^{-1} I^-\tw)\ell$. Now $P\cap\tw^{-1}I^-\tw$ is generated by $T$ and affine root subgroups, and these affine root subgroups are contained in either $L$ or $P(1)$. Thus, we have
$P\cap \tw^{-1}I^-\tw=(L\cap \tw^{-1}I^-\tw)(P(1)\cap \tw^{-1}I^-\tw)$. Conjugating by $\ell^{-1}$, we obtain
\[
P\cap x^{-1}I^-x=(L\cap x^{-1}I^-x)(P(1)\cap x^{-1}I^-x).
\]
As $J=B_\phi P(1)\subseteq P$, this implies
\[
\Aut_{\cG'}(\cE)=J\cap x^{-1}I^-x=(B_\phi\cap x^{-1}I^-x)(P(1)\cap x^{-1}I^-x).
\]
Now the fact that $\tw\in\Omega\tw_d$ implies that $\tw P\tw^{-1}$ contains the standard Iwahori $I$. Thus, $\tw P(1)\tw^{-1}\subseteq I(1)$, which in turn implies
\[
P(1)\cap x^{-1} I^- x=x^{-1} (\tw P(1)\tw^{-1} \cap I^-)x=\{1\}. 
\]
On the other hand, as $U_L\cap U_L^-=\{1\}$,   Lemma \ref{l:Aut in B_phi} implies
\[
B_\phi \cap \ell^{-1} B_L \ell=\{ tv\in T_\phi  U_\phi\, | \, t^{-1}u t vu^{-1} =1\}.
\] 
We now show that for every $tv$ in the above set, $t=v=1$. 
Indeed,  since $v\in U_\phi$ and $t^{-1}u^{-1} tu\in \Upsilon$, the requirement $v=t^{-1}u^{-1}tu$ implies that $v\in U_\phi\cap \Upsilon=\{1\}$. Thus, $v=1$ and $tu=ut$. As $u\in \mathring{\Upsilon}$ and $t\in T_\phi$, we get 
$t\in \Stab_{T_\phi}(u)= \{1\}$. Thus, $\Aut_{\cG'}(\cE)=1$. This concludes the proof of the proposition. \qed

\subsubsection{Generic locus} 
As discussed in \S \ref{s:Kottwitz}, the Kottwitz homomorphism defines an isomorphism $\pi_0(\Bun_{\cG'})\simeq \bZ$. Let $\tw_1$ denote the generator of $\Omega$ specified in \S \ref{s:Omega}. The above proposition immediately implies: 
\begin{cor}\label{c:wild genericLocus} For each $\alpha \in \bZ$, the generic locus $\mathring{\Bun_{\cG'}^\alpha}$ consists of one bundle, namely, the one labelled by $(\tw_1^\alpha \tw_d, \mathring{u})$.
\end{cor}

\subsection{Proof of strict rigidity} 
The fact that the coarse moduli space of $\Bun_{\cZ}^0$ is a point was discussed in \S \ref{s:centre}. It remains to determine the relevant elements on $\Bun_{\cG'}$. Let $\cE\in \Bun_{\cG'}$ be a bundle associated to $I^-xJ$, where $x=\tw\ell$, and $\tw$ and $\ell=wu$ are as in Proposition \ref{p:wild G'-bundles}. 
Recall that we have a restriction map \eqref{eq:Res}
\begin{align*}
\Res: \Aut_{\cG'}(\cE)\simeq J\cap x^{-1}I^-x&\ra \Aut(\cE|_{\cO_0})\times\Aut(\cE|_{\cO_\infty})\simeq I^\opp\times J\\
h&\mapsto (xhx^{-1},h).
\end{align*}
The bundle $\cE$ is called relevant if the pullback of $\gamma_S=\cL_{\overline{\chi}}\boxtimes\cL_\mu$ to $(\Aut_{\cG'}(\cE))^\circ$ via $\Res$ is constant.

\subsubsection{}
If $\mathring{u}\in\mathring{\Upsilon}$ and $\tw\in\Omega\tw_d$, then Proposition \ref{p:wildAut} implies that $\cE$ is relevant. This gives one relevant point on each component of $\Bun_{\cG'}$. It remains to show that all other points are irrelevant. 
To this end, it suffices to show that the Frobenius trace of $\Res^*(\gamma_S)$ is non-constant over the subscheme of $(\Aut_{\cG'}(\cE))^\circ$ given in Parts (i), (ii), (iii) of Proposition \ref{p:wildAut}. 

\subsubsection{} Suppose $u\not\in\mathring{\Upsilon}$ and let $t:=(1,...,1,t_j,1,...,1)\in T_j$. The Frobenius trace of $\Res^*(\gamma_S)$ over $T_j$ is
\[
\overline{\chi}(\tw wutu^{-1}w^{-1}\tw^{-1})\mu(t)=\overline{\chi}(\tw wtw^{-1}\tw^{-1})\rho(t)=\chi_l^{-1}(t_j)\rho_{j-d}(t_j), \quad\quad\forall t_j\in k^*.
\]
Here, $l$ is determined by the requirement $(\tw w)T_j(\tw w)^{-1}=T_l$, and the first equality comes from the fact that $\overline{\chi}$ is trivial on $I^\opp(1)$ . If this trace function constantly equals $1$, then $\chi_l=\rho_{j-d}$, which contradicts the assumption that $\chi_i$'s and $\rho_j$'s are disjoint. 

\subsubsection{} Suppose $u\in\mathring{\Upsilon}$ and $w\neq 1$, and let $\beta=\alpha_{ij}$ be a root of $U_L\cap w^{-1}U_L^-w$. As any such $u$ is conjugate to $\mathring{u}$ by $T_\phi$, we may assume $u=\mathring{u}$. Let $t=(1,...,1,t_j,1,...,1)\in T_j$. Now we have two cases: 
\begin{itemize} 
\item[(a)] If $\beta\in\Phi(\Upsilon)$, then (as above) the trace function in question restricted to $T_j$ is $\chi_l^{-1}\rho_{j-d}$  and therefore non-constant. 

\item[(b)]  Suppose $\beta\in\Phi(U_\phi)$. Let $t=(1,...,1,t_j,1,...,1)\in T_j$ and $tv(t)\in Y_\beta$. By a similar argument, the Frobenius trace function over $Y_\beta$ equals
\[
\overline{\chi}((\tw wu)tv(t)(\tw wu)^{-1})\rho(tv(t))=\chi_l^{-1}(t_j)\rho_{j-d}(t_j),
\]
and is therefore non-constant. 
\end{itemize} 

\subsubsection{} Finally, suppose $u\in\mathring{\Upsilon}$, $w=1$, and $\tw\notin\Omega\tw_d$. Then there exists a lowest weight $\delta$ of the $L$-module $V$ such that $\ell^{-1}U_\delta\ell\subseteq(\Aut_{\cG'}(\cE))^\circ$. Suppose the Frobenius trace function is constant over $\ell^{-1}U_\delta\ell$. Let $E_\delta$ be a basis of $\fu_\delta=\mathrm{Lie}(U_\delta)$. Then $U_\delta=\{\exp(\lambda E_\delta)|\lambda\in k\}$. As $V=P(1)/P(2)\simeq\fp(1)/\fp(2)$, we may regard $\phi$ as a linear function on $\fp(1)/\fp(2)$. Thus, the Frobenius trace function over $\ell^{-1}U_\delta\ell$ is
\[
\overline{\chi}((\tw\ell)\ell^{-1}\exp(\lambda E_\delta)\ell(\tw\ell)^{-1})\mu(\ell^{-1}\exp(\lambda E_\delta)\ell)=\psi\phi(\lambda u^{-1}E_{\delta}u)=\psi(\lambda\phi(u^{-1}E_\delta u)).
\]
If this function equals $1$ for all $\lambda$, then we obtain
$\phi(u^{-1}E_\delta u)=0$. Note that replacing $u$ with a $T_\phi$-conjugate does not change the double coset; thus, the above equality holds for all $u\in\Ad_{T_\phi}(\mathring{u})=\mathring{\Upsilon}$. Moreover, $\phi$ is stabilised by $U_\phi$, so $\phi(\Ad_u E_\delta)=0$ for all $u\in \mathring{\Upsilon}U_\phi$. As $\mathring{\Upsilon}U_\phi$ is dense in $\Upsilon U_\phi=U_L$, we obtain $\phi(\Ad_{U_L}(E_\delta))=0$. Taking the differential, we get $\phi(\mathrm{ad}_{\fu_L}(E_\delta))=0$. Since $\delta$ is a lowest weight of $V$, $E_\delta$ and $\mathrm{ad}_{\fu_L}(E_\delta)$ span the irreducible submodule $V_\delta\subseteq V$ with lowest weight $\delta$. Thus, $\phi(V_\delta)=0$, which is a contradiction to Corollary \ref{c:phisp}. 
\qed

\section{Hecke eigenvalue in the wild case}\label{s:wild eigenvalue} 
The goal of this section is to prove Theorem \ref{t:eigenvalue} in the wild case. Since the proof is similar to the tame setting, we will only sketch the main steps. 

\subsection{Alternative description of $\Bun_\cG$.} It will be convenient to replace the parahoric $P$ at $\infty$ with its conjugate $P' \subset G(\cO)$, defined in \S \ref{ss:fundamentalAlcove}. Note that this does not change the isomorphism type of the stack $\Bun_{\cG}$.

\subsubsection{} Let $\tau:=[n/d]$ and $\sigma:=n-\tau d$.
Recall the integers $n_i$ defined in \eqref{ni}.
Let $\Bun_{1,2}$ be the classifying stack of  
$(\cE,F^*\cE,\{v^i\}_{1\leq i \leq n}, F_*\cE,\{v_i\}_{1\leq i \leq d})$, where:
\begin{itemize}
\item $\cE$ is vector bundle of rank n on $\bP^1$;
\item $\cE=F^0\cE\supset F^1\cE\supset\cdots\supset F^n\cE=\cE(-\{0\})$ is a decreasing filtration $F^*\cE$ giving a complete flag of the fibre of $\cE$ at $0$;
\item $v^i\in F^{i-1}\cE/F^i\cE$ is a non-zero vector for each $1\leq i\leq n$;
\item $\cE(-\{\infty\})=F_0\cE\subset F_1\cE\subset\cdots\subset F_{d}\cE=\cE$ is an increasing filtration $F_*\cE$ giving a partial flag (of type $P'$) of the fibre of $\cE$ at $\infty$ such that  $\dim F_i\cE/F_{i-1}\cE=\tau+1$ for $1\leq i\leq \sigma$; $\dim F_i\cE/F_{i-1}\cE=\tau$ for $\sigma+1\leq i\leq d$;
\item  $\{v_{n_i+1}, v_{n_i+2}, ...,v_{n_{i+1}}\}$ is a set of vectors in $F_i\cE/F_{i-2}\cE$ whose image in $F_i\cE/F_{i-1}\cE$ is a basis. Here, $F_{-1}\cE=F_{d-1}\cE(-\{\infty\})$ and $F_{-2}\cE=F_{d-2}\cE(-\{\infty\})$.
\end{itemize}

\subsubsection{} 
As in the tame setting, 
we can reformulate the moduli problem by choosing a trivialisation of $\cE$ over $\bP^1-\{0,\infty\}$. 
Let $\widetilde{\Bun}_{1,2}$ be the stack of bundles classifying the data 
$(\Lambda^*, \{v^i\}, \Lambda_*, \{v_i\})$, where the latter is defined analogously to the tame case. Then, we have  isomorphisms of stacks 
\[
\Bun_\cG\simeq \Bun_{1,2} \simeq \widetilde{\Bun}_{1,2}/G(k[t,t^{-1}]). 
\]
Let $\Bun_{1,2}^\alpha$ be the substack classifying data with degree $\alpha\in\bZ$. These are the components of $\Bun_{1,2}$.

\subsubsection{}  Note that after conjugation by $\tw_d$, the relevant points in $\Bun_{\cG'}$ become $I^-\tw \mathring{u}'J'$, where $\tw\in\Omega$. Moreover, the inclusion of a relevant orbit $j_\alpha: \cO \hookrightarrow \Bun_{\cG}^\alpha$ can be explicitly described as follows: 
\[
L_S=T\times B_{\phi'} V'\simeq\cO_\alpha=I^-(1)\backslash I^-(1)T \tw_1^\alpha \mathring{u}' J'/P'(2)
 \hookrightarrow \Bun_\cG^\alpha\simeq\Bun_{1,2}^\alpha.
\]
\subsubsection{}
We can describe the latter  embedding in terms of lattices. To this end, let $\star_0$ be the $G(k[t,t^{-1}])$-orbit of $(\Lambda^*(\star_0),v^i(\star_0),\Lambda_*(\star_0),v_i(\star_0))$ where
\begin{itemize}
\item $\Lambda^i(\star_0)=\text{Span}\{e_j|j>i\}$, $1\leq i\leq n$;
\item $v^i(\star_0)=e_{i+1}$;
\item $\Lambda_i(\star_0)=\text{Span}\{e_j|j\leq n_{i+1}\}$, $1\leq i\leq d$;
\item $v_i(\star_0)=e_i$, $1\leq i\leq d$.
\end{itemize}
The map $j_\alpha: L_S=T\times B_{\phi'}V'\hookrightarrow \Bun_{1,2}^\alpha$ is then given by 
\[
(t,g)\mapsto (\Lambda^i(\star_0),t\cdot e_{i+1}, \widetilde{w}_1^\alpha \mathring{u}'g\cdot\Lambda_i(\star_0),\widetilde{w}_1^\alpha g\cdot e_i).
\]


\subsection{Relevant part of the Hecke stack}
For $i\in \{1,2,...,d\}$, set $t_{n_i+1}:=x_i^{-1}$.  Let $j(k)$ be the $k$-th smallest number in 
the set $\{j\in \{1,2,...,n\} \, | \, j\neq 1+n_i, \, \forall i\}$ and define 
$t_{j(k)}:=(1-y_k)^{-1}$. Similar, but more involved, computations as the tame setting gives: 

\begin{prop}\label{p:wildHecke} Consider the correspondence  
\[
\begin{tikzcd}
&  \mathring{\mathrm{GR}}_{\omega_1} = \bGm^d\times(\bGm-\{1\})^{n-d} \arrow{dl}[swap]{\pi_1'} \arrow{dr}{\pi_2}  & \\
\bGm^n\times\bGm^{n-d} \times \bGa&       &  \bP^1-\{0,\infty\},
\end{tikzcd}
\]
where
\[
\pi_1'(x_i,y_j)=(t_1,...,t_n,y_1^{-1},...,y_{n-d}^{-1}, \sum_{i=1}^dx_i,), \qquad 
\pi_2(x_i,y_j)=\prod_{i=1}^dx_i\prod_{j=1}^{n-d}(1-y_j^{-1}).
\] 
Then the Hecke eigenvalue is $E_{\on{Std}}\simeq\pi_{2!}(\pi_1')^*(\cL_{\overline{\chi}}\boxtimes\cL_\rho\boxtimes\cL_\psi)[n-1]$. 
\end{prop} 

In the above proposition, $\pi_2$ is as in \eqref{eq:relevantHecke} and $\pi_1'$ is the composition of $\pi_1$ of \eqref{eq:relevantHecke} with $\phi$ and the projection from $B_\phi$ to $T_\phi$. Note that if $d=n$, then $\mathring{\on{GR}}_{\omega_1}\simeq \bGm^n$ and the above correspondence coincides with the Kloosterman diagram \cite[prop. 3.4]{HNY}.

\subsection{Proof of Theorem \ref{t:eigenvalue} in the wild case} 
To show that $E_{\on{Std}}$ is (geometrically) isomorphic to $\sH$, it is sufficient to show that their trace functions are a scalar multiple of each other. Similar computation to the tame case shows that for every $a\in \bP^1-\{0,\infty\}$, we have 
\[
\on{tr}_{E_{\on{Std}}}(a)\cdot(-1)^{n+m-1} \prod_{i=1}^{n-d}\rho_i(-1)\left(\sum_{a_i\neq0}\psi(a_i)\chi_i\rho_i^{-1} (a_i)\right)
=\on{tr}_{\sH(\chi, \rho)}.
\]
\qed


\appendix

\section{Principal gradings}\label{s:grading} 
In this section, we gather some facts about principal gradings of reductive Lie algebras. The results are probably known to the experts but we could not find an appropriate reference.

\subsection{Notation} 
\subsubsection{} Let $G$ be a split reductive group over a field $k$. Let $T$ be a maximal split torus, $X_*(T)=\Hom(\bGm, T)$ the group of cocharacters and $\cA_\bQ=\Hom(\bGm, T)\otimes \bQ$  the rational apartment. Let $\fg:=\on{Lie}(G)$ and  $\Phi$ the root system of $\fg$. For each $\alpha\in \Phi$, let $\fg_\alpha\subset \fg$ be the corresponding root space. Let $h$ denote the Coxeter number of $G$. We fix, once and for all, a Borel subgroup $B\subset G$. Let $\check{\rho}$ denote the half sum of positive coroots. 

\subsubsection{} 
Let $G^\aff=G(\!(s)\!)$ denote the loop group of $G$ with Lie algebra $\fg^\aff$ and affine roots $\Phi^\aff$. For each $\alpha\in \Phi^\aff$, let $\fg_\alpha^\aff \subset \fg^\aff$ be the corresponding affine root space.  Note that affine roots are of the form $\alpha+m$ where $\alpha \in \Phi$ and $m\in \bZ$. We call $\alpha$ the \emph{finite part} of $\alpha+m$. Thus, we have a map 
\begin{equation}\label{eq:Fin}
\mathfrak{Fin}: \Phi^\aff \ra \Phi
\end{equation} 
sending an affine root to its finite part.

\subsection{Inner gradings}\label{ss:inner gradings}

\subsubsection{} 
Let $x\in \cA_\bQ$ be an element of order $d$. By definition, $d$ is the smallest positive integer such that $\alpha(x)\in \frac{1}{d} \bZ$ for all $\alpha \in \Phi^\aff$. For $i\in\bZ/d\bZ$, let  
\[
\Phi_i:=\{\alpha \in \Phi \, | \, \alpha(x)-i/d\in \bZ\}.
\]
Define
\begin{equation} \label{eq:grad}
\fg=\bigoplus_{i\in \bZ/d\bZ} \fg_i, \qquad \qquad   
\fg_i:=
\begin{cases} 
\ft\oplus
\displaystyle \bigoplus_{\alpha \in \Phi_0} \fg_\alpha & i=0;\\
\displaystyle \bigoplus_{\alpha \in \Phi_i} \fg_\alpha & i\neq 0.
\end{cases}
\end{equation} 
Kac proved that, up to conjugation, these are all the inner gradings of $\fg$.

\subsubsection{} A grading is called \emph{principal} (or $N$-regular) if $\fg_1$ contains a principal nilpotent element. Using Kostant's description of principal nilpotent elements, it is easy to see that for all $d\in \{1,2,..,h\}$, the grading defined by $x=\check{\rho}/d$ is principal. According to \cite{RLYG}, up to conjugation, every inner principal grading  is of this form.

\subsubsection{Vinberg pairs} 
Let $G_0$ denote the connected subgroup of $G$ whose Lie algebra is $\fg_0$. Then $G_0$ is a split reductive group. The pair $(G_0, \fg_1)$ was studied extensively by Vinberg under the name of ``$\theta$-groups''. We refer to $(G_0, \fg_1)$ as a Vinberg pair. The adjoint action of $G$ on $\fg$ restricts to an action of $G_0$ on $\fg_1$, called the (finite) \emph{Vinberg representation}.

\subsubsection{Affine gradings} A point $x\in \cA_\bQ$ of order $d$ also defines a grading of $\fg^\aff$. Namely, let \[
\Phi^\aff_i:=\{\alpha \in \Phi^\aff\, |\, \alpha(x)=i/d\}
\]
and define 
\[
\fg^\aff = \bigoplus_{i\in \bZ} \fg^\aff_i,
\]
where 
\[
\fg^\aff_0:=\ft\oplus\bigoplus_{\alpha\in \Phi^\aff_0}\fg_\alpha  \qquad \qquad 
\fg^\aff_i:=\begin{cases} 
\displaystyle \bigoplus_{\alpha\in \Phi^\aff_i} \fg_\alpha & 1\leq i\leq d-1;\\
 \fg^\aff_{i+kd}=s^k\fg^\aff_i, & k\in\bZ.
 \end{cases} 
\]
This is sometimes referred to as the Kac--Moy--Prasad grading of $\fg^\aff$, cf. \cite{Chen}. 

\subsubsection{Affine Vinberg pairs} 
Let $L:=G_0^\aff\subset G^\aff$ be the reductive subgroup of $G^\aff$ with Lie algebra $\fg_0$, and $V:=\fg_1^\aff$. Then $L$ acts on $V$ via the adjoint action. One readily checks that $\mathfrak{Fin}$ \eqref{eq:Fin} defines bijections 
\[
\Phi^\aff(L)=\Phi^\aff_0\leftrightarrow \Phi(G_0), \qquad \qquad \Phi^\aff(V) = \Phi^\aff_1\leftrightarrow \Phi(\fg_1). 
\]
Thus, we have canonical isomorphisms $L\simeq G_0$ and $V\simeq \fg_1$, giving an identification of the finite and affine Vinberg representations.

\subsubsection{Moy--Prasad subalgebras} 
For each non-negative integer $j$, define 
\[
\fp(j):=\bigoplus_{i\geq j} \fg^\aff_i.
\]
Then $\fp=\fp(0)$ is a parahoric subalgebra of $\fg^\aff$ and $\fp(j)$'s are its Moy--Prasad subalgebras. 
Let $P$ be the corresponding (connected) parahoric group with Moy--Prasad subgroups $P(i)$. Then, we have isomorphisms 
\[
 P/P(1)\simeq G_0^\aff \simeq G_0, \qquad\qquad  P(1)/P(2)\simeq \fp(1)/\fp(2)=\fg^\aff_1 \simeq \fg_1.
\]

\subsection{Barycentres} 
\subsubsection{Alcoves} Let $C\subset \cA$ be an alcove. By definition, the boundary $\overline{C}-C$ consists of root hyperplane $\{H_\alpha\}$, where $\alpha$ runs over a basis $\Delta^\aff(C)$ of $\Phi^\aff$. The map $C\mapsto \Delta^\aff(C)$ defines a bijection between alcoves and bases of $\Phi^\aff$.

\subsubsection{Barycentre} An element $x\in \cA_\bQ$ of order $d$ is called a \emph{barycentre} if there exists an alcove $C$ such that for all $\alpha \in \Delta^\aff(C)$,  $\alpha(x)$ is either $0$ or $\frac{1}{d}$. If this condition holds for some alcove $C$, then it holds for all alcoves containing $x$ in their closure. Thus, the definition is independent of the choice of $C$. 

\begin{exam} One can show (e.g. using \cite[Lemma 3.1]{RY}) that $\check{\rho}/d$ is a barycentre if and only if $d\in \{1,...,h\}$. 
\end{exam} 

\subsubsection{Characterisation via the Vinberg representation} Recall that every $x\in \cA_\bQ$ gives rise to an affine Vinberg pair $(L,V)$. The chosen Borel $B\subset G$ determines a basis $\Delta(L)\subset \Phi^\aff(L)$ and lowest weights $\wt^-(V)\subset \Phi^\aff(V)$ of the $L$-module $V$. Let
\begin{equation}\label{eq:Theta}
\Theta(x):=\Delta(L)\sqcup \wt^-(V)\subset \Phi^\aff. 
\end{equation} 

\begin{prop}\label{p:simple affine roots} An element $x\in \cA_\bQ$ is a barycentre if and only if $\Theta(x)$ is a basis of $\Phi^\aff$. 
\end{prop}

\begin{ntn} \label{n:alcove} 
For a barycentre $x$, we let $C_x$ denote the alcove corresponding to the basis $\Theta(x)$. 
\end{ntn} 

\begin{proof}
By the definition of $L$ and $V$, we have 
\[
\alpha(x)= 
\begin{cases} 
0 & \alpha \in \Delta(L);\\
\frac{1}{d} & \alpha \in \wt^-(V). 
\end{cases} 
\]
Thus, if $\Theta(x)$ is a basis of $\Phi^\aff$, then $x$ is a barycentre (of the facet of the alcove corresponding to $\Theta(x)$). 

Conversely, suppose $x$ is a barycentre. Let $C$ be any alcove containing $x$ in its closure.  Let us write $\Delta^\aff(C)=\{\beta_1,...,\beta_{r+1}\}$, where
\[
\beta_i(x)=
\begin{cases} 
0 & i\in \{1,..., l\};\\
\frac{1}{d} &i\in \{l+1,..., r+1\}.
\end{cases} 
\]
A priori, $\{\beta_1,...,\beta_l\} \subseteq\Phi(L)$ and $\{\beta_{l+1},...,\beta_{r+1}\} \subseteq \Phi(V)$.

We claim that $\{\beta_1,...,\beta_l\}$ is, in fact,  a basis for $\Phi(L)$.
Indeed, as $\Delta^\aff(C)$ is a basis for $\Phi^\aff$, every  affine root has a unique expression $\alpha=\sum_{i=1}^{r+1}n_i\beta_i$, where all $n_i$'s are either nonnegative or non-positive. In particular, for $\alpha\in \Phi(L)$, we have 
\[
0=\alpha(x)=\sum_{i=1}^l n_i\beta_i(x)+\sum_{i=l+1}^{r+1}n_i\beta_i(x)=\frac{1}{d}\sum_{i=l+1}^{r+1}n_i.
\]
Therefore, $n_i=0$ for all $l+1\leq i\leq r+1$. Thus, every root of $L$ has a unique expression of the form $\sum_{i=1}^l n_i \beta_i$, establishing the claim.

It follows that there exists an element $w\in W_L$ such that 
$w\{ \beta_1,...,\beta_l \}=\Delta(L)$. Since $w\cdot C$ also contains $x$ in its closure, we can assume without loss of generality that $w=1$.
 
It remains to show that $\{\beta_{l+1},...,\beta_{r+1}\}=\wt^-(V)$. Let $\theta\in \wt^-(V)$ and write 
$\theta=\sum_{i=1}^{r+1}n_i\beta_i$, where all $n_i$ have the same sign. Then
\[
\frac{1}{d}=\theta(x)=\sum_{i=1}^ l n_i\beta_i(x)+\sum_{i=l+1}^{r+1}n_i\beta_i(x)=\sum_{i=l+1}^{r+1}n_i\frac{1}{d}.
\]
Thus, there exists a unique $j\in \{l+1,...,r+1\}$ such that $n_j=1$ and  $n_i=0$ for all other $i$ in this set. Moreover, as all $n_i$ have the same sign, we obtain that $n_i\geq 0$ for all $i\in \{1,...,l\}$. Next, note that 
\[
\theta +\sum_{i=1}^ l n_i(-\beta_i)=\beta_{j}\in\Phi(V).
\]
Since $\{\beta_1,...,\beta_l\}$ are positive roots of $L$ and $\theta$ is a lowest weight, the above relation implies that $\theta=\beta_j$. It follows that 
$\wt^-(V)\subseteq \{\beta_{l+1},..., \beta_{r+1}\}$.

To see the reverse inclusion, note that every weight can be obtained from a lowest weight; thus, in particular, for every $\beta_\kappa$, $\kappa\in \{l+1,..., r+1\}$, there exists a lowest weight $\theta$ and integers $n_i\geq 0$ such
\[
\beta_\kappa = \theta + \sum_{i=1}^l n_i \beta_i.
\]
By the above discussions, $\theta=\beta_j$ for some $j\in \{l+1,..., r+1\}$. By the uniqueness of the expression of roots as combination of simple affine roots, we conclude that $n_i=0$ for $i=1,...,l$ and $\beta_\kappa=\theta$; thus $\{\beta_{l+1},..., \beta_{r+1}\}\subseteq \wt^-(V)$. This completes the proof of the proposition.

\end{proof}

\subsubsection{} The above proof gives us some information about the  Vinberg representation:
\begin{cor}\label{c:irredRep} If $x$ is a barycentre of order bigger than one, then $V$ is a direct sum of $r+1-l$ irreducible $L$-submodules, where $r$ and $l$ are the semisimple ranks of $G$ and $L$, respectively. 
\end{cor} 

 It would be interesting to further understand the pair $(L,V)$ and the structure of the $L$-module $V$. In the next subsection, we prove some results in this direction for $\GL_n$.

\subsubsection{Normalised Kac coordinate} Let us write $\Delta^\aff=\{\alpha_0, \alpha_1,...,\alpha_r\}$. Let $x\in \cA_\bQ$ be an element of order $d$. Let $\tw\in \tW$ be an element of the Iwahori-Weyl group which maps $x$ into the closure of the fundamental alcove. The \emph{normalised Kac coordinates} of $x$ is 
\[
d\cdot(\alpha_0(\tw x),..., \alpha_r(\tw x))\in \bZ_{\geq 0}^{r+1}.
\]
 The above proposition shows that the normalised Kac coordinates of a barycentre consists of $l$ zeros and $r+1-l$ ones. (Of course, the exact location of zeros and ones depends on the chosen ordering on $\Delta^\aff$.)

\subsection{Inner principal gradings for $\GL_n$}\label{ss:principal gradings GL_n} Henceforth, we restrict to the case $G=\GL_n$  and use the notation of \S \ref{s:notation}. So elements of $G$ are denoted by $(a_{ij})$ and roots by $\alpha_{ij}$. Fix an integer  $d\in \{1,2,...,n\}$ and consider the (principal) grading of $\fg$ defined by the barycentre $x=\check{\rho}/d$. Our goal is to understand the pair $(G_0, \fg_1)$ and the corresponding Vinberg representation. Note that if $d=1$, then $(G_0,\fg_1)=(G,\fg)$ and the Vinberg representation is simply the adjoint representation. Thus, we may assume $d>1$ when convenient.

\subsubsection{}  First, note that the definition of the grading \eqref{eq:grad} implies  
\[
G_0=\langle T,U_{\alpha}\mid \mathrm{ht}(\alpha)\equiv0\mod d\rangle = \{(a_{ij})\in G\mid a_{ij}=0\,\,  \textrm{when} \,\, i\not\equiv j\mod d\},
\]
and
\[
\fg_1=\bigoplus_{\mathrm{ht}(\alpha)\equiv1\mod d}\fg_{\alpha} = \bigoplus_{j\equiv i+1\mod d} \on{Span}\{E_{ij}\}.
\]

\subsubsection{}\label{sss:V}
 Next, define 
\[
G_i:=\{(a_{pq})\in G\mid p\neq q: a_{pq}=0 \ \text{unless}\  p\equiv q\equiv i\mod d;\ a_{pp}=1, p\not\equiv i\mod d \},
\]
and
\[
V_i:=\bigoplus_{p\equiv i, \,\, q\equiv i+1\on{mod}\, d}\on{Span}\{E_{pq}\}.
\]

\begin{prop}\label{p:Vinberg's pair} 
Let $\displaystyle \tau:=[n/d]$ and $\sigma:=n-\tau d$. Let $\Delta(G_0)$ be the set of simple root of $G_0$. Then
\begin{itemize} 
\item[(i)]
$G_0=\prod_{i=1}^d G_i$, where $G_i\simeq GL_{\tau+1}$ for $1\leq i\leq \sigma$ and $G_i\simeq GL_\tau$ for $\sigma+1\leq i\leq d$. Moreover, 
\[
\Delta(G_0)=\{\alpha_{i,i+d}\, | \, 1\leq i\leq n-d \}. 
\]
\item[(ii)]  $\fg_1=\bigoplus_{i=1}^d V_i$. Moreover, when $d>1$, each
$V_i$ is an irreducible $G_0$-representation with lowest weight 
\[
\alpha_{(i)}:=
\begin{cases}
\alpha_{i+\tau d,i+1} \qquad\ \ 1\leq i\leq \sigma;\\
\alpha_{i+(\tau-1)d,i+1} \quad \sigma+1\leq i\leq d-1;\\
\alpha_{\tau d,1}   \qquad\qquad\ i=d.
\end{cases}
\]
\end{itemize}
\end{prop}
\begin{proof} Part (i) is easily verified from the definition. For (ii), note that $G_0$ maps the root subspace $\fg_{\alpha_{ij}}$  into direct sum of several $\fg_{\alpha_{i'j'}}$ where $i\equiv i',j\equiv j'\mod d$. It follows that $V_i$ is an $G_0$-submodule of $\fg_1$. Now the weights in $V_i$ are of the form $\alpha_{pq}$, $p\equiv i$ and $q\equiv i+1$ modulo $d$. Thus, the lowest weight is $\alpha_{i+t'd,i+1}$ or $\alpha_{t'd,1}$ for maximal possible $t'$, giving the above formula. 
\end{proof} 

\subsubsection{The functional $\phi_\spp$} Recall the functional $\phi_\spp\in \fg_1^*$ defined in \eqref{eq:phisp}. The above proposition immediately implies a fundamental property of this functional: 
\begin{cor} \label{c:phisp}
The restriction of $\phi_\spp$ to every irreducible constituent $V_i\subseteq V$ is non-zero. 
\end{cor}

\subsubsection{Explicit description of $\Theta(x)$} The above proposition also gives a description of the affine roots $\Theta(x)$ associated to $x$ \eqref{eq:Theta}. 
Let $\wt^-(\fg_1)=\{\alpha_{(1)},\alpha_{(2)},...,\alpha_{(d)}\}$. Let $(k_1, ..., k_n)$ be an ordering of $\{1,2,...,n\}$ defined as follows: 
\[
(k_1,k_2,...,k_n) = (1,1+d, 1+2d, ..., 1+\tau d, 2, 2+d, ..., (\tau-1)d, \tau d).
\]
By convention, $k_{n+1}=1$. Recall that $n=\tau d + \sigma$, where $0\leq \sigma\leq d-1$. For $i\in \{1, 2,...,d+1\}$, let 
\begin{equation} \label{ni}
n_i:=\begin{cases}
(i-1)(\tau+1) & 1\leq i\leq \sigma+1;\\
\sigma (\tau+1)+(i-\sigma -1)\tau & \sigma+2\leq i\leq d+1. 
\end{cases}
\end{equation} 
Proposition \ref{p:Vinberg's pair} then implies: 

\begin{lem}\label{l:cycle of simple affine roots} The map $\mathfrak{Fin}$ \eqref{eq:Fin} defines a bijection between 
$\Theta(x)$  and $\{\alpha_{k_i,k_{i+1}}\mid 1\leq i\leq n\}$.
Under this bijection, $\wt^-(V)$ is mapped to $\{\alpha_{k_{n_i}, k_{n_i+1}}\mid 2\leq i\leq d+1\}$. 
\end{lem}

\subsection{Moving to the fundamental alcove} \label{ss:fundamentalAlcove}  Recall that in Conventions \ref{n:alcove}, we defined a canonical alcove $C_x$ containing $x=\check{\rho}/d$ in the closure. Note that $C_x$ is not, in general, the fundamental alcove. It will be convenient to use the  Iwahori--Weyl group to move $x$ into the closure of the fundamental alcove. 

\subsubsection{}
Let $\tw_d$ be an element of the Iwahori--Weyl group $\tW$ such that $\tw_d C_x$ is the fundamental alcove. Note that $\tw_d$ is unique, up to the stabiliser $\Omega\subset \tW$ of fundamental alcove. In particular, the finite part of $\tw_d$, denoted by $w_x\in W=S_n$, is unique, up to translations by an $n$-cycle. We assume, without the loss of generality, that $w_x$ fixes $1$.  Lemma \ref{l:cycle of simple affine roots} then implies: 

\begin{cor}\label{c:permutation} 
For all $i\in \{1,2,...,n\}$,  $w_x(k_i)=i$. 
\end{cor}

\subsubsection{Normalised Kac Coordinates}
Let $x':=\tw_d \cdot x$. By definition, $x'$ is in the closure of the fundamental alcove. 
The following lemma gives the normalised Kac coordinates of $x$:

\begin{lem}\label{Kac coordinates} We have 
$
\alpha_j(x')=
\begin{cases} 
1/d & j=n_i\ \textrm{for some}\ i\in \{2,...,d+1\};\\
0 & \textrm{otherwise}.
\end{cases} 
$
\end{lem}

\begin{proof} 
By Corollary \ref{c:permutation} and Lemma \ref{l:cycle of simple affine roots}, we have  $\tw_d. \wt^-(V) = \{\alpha_{n_2},..., \alpha_{n_{d+1}}\}\subset \Delta^\aff(G)$ where, by convention, $\alpha_{n_{d+1}}=\alpha_n=\alpha_0=1+\alpha_{n,1}$. 
\end{proof} 

\subsubsection{The associate Vinberg pair} Let $(G_0', \fg_1')$ denote the Vinberg pair associated to $x'$.  Proposition \ref{p:Vinberg's pair} implies that $G_0'$ is the Levi subgroup of $G=\GL_n$ given by 
\begin{equation}\label{eq:G_0'}
G_0'=\GL_{\tau+1} \times ... \times \GL_{\tau+1} \times \GL_\tau \times ...\times \GL_\tau.
\end{equation} 
Here, there are $\sigma$ many $GL_{\tau+1}$'s and $d-\sigma$ many $GL_\tau$'s. Proposition \ref{p:Vinberg's pair} also implies that $\fg_1'\subset \fg$ is given by 
\[
\fg_1'=
\begin{pmatrix}
0&\on{Mat}_{\tau+1,\tau+1}&      &              &            &          &      &\\
&\ddots        &\ddots&              &            &          &      &\\
&              &0     &\on{Mat}_{\tau+1,\tau+1}&            &          &      &\\
&              &      &0             &\on{Mat}_{\tau+1,\tau}&          &      &\\
&              &      &              &0           &\on{Mat}_{\tau,\tau}&      &\\
&              &      &              &            &\ddots    &\ddots&\\
&              &      &              &            &          &  0&\on{Mat}_{\tau,\tau}\\
 \on{Mat}_{\tau,\tau+1}&   &              &            &          &      &      &0
\end{pmatrix}
\]

\subsubsection{The character $\phi_\spp'$} Recall the functional $\phi_\spp \in \fg_1^*$ defined by \eqref{eq:phisp}. Conjugating by $\tw_d$, we obtain the functional on $\fg_1'$ given by 
\begin{equation}\label{eq:phisp'}
\phi_\spp'=E_{1+n_1,1+n_2}^*+E_{1+n_2,1+n_3}^*+\cdots+E_{1+n_{d-1},1+n_d}^*+E_{1+n_d,1+n_1}^*.
\end{equation}
Note that the stabiliser of $\phi_\spp'$ (under the action of $G_0'$ on $(\fg_1')^*$), is $ZL'_{\phi'}$, where $Z$ is the centre and $L'_{\phi'}\subset G_0'$ is the block-diagonal subgroup defined by 
\begin{equation} \label{eq:Lphi'}
L'_{\phi'} :=\mathrm{diag}
\left(
\begin{pmatrix}
1 &       \\
&GL_\tau(k)
\end{pmatrix},
\cdots,
\begin{pmatrix}
1 &       \\
&GL_\tau(k)
\end{pmatrix},
\begin{pmatrix}
1 &       \\
&GL_{\tau-1}(k)
\end{pmatrix},
\cdots,
\begin{pmatrix}
1 &       \\
&GL_{\tau-1}(k)
\end{pmatrix}
\right).
\end{equation} 
Here, there are $\sigma$ many $\GL_\tau$ and $d-\sigma$ many $\GL_{\tau-1}$. 

\subsubsection{Subgroups of the positive loop group} Associated to $x'$, we have the parahoric $P'\subset G(\cO)$ and the corresponding affine Vinberg pair  
\[
(L',V'):=(P'/P'(1), P'(1)/P'(2))\simeq (G_0',\fg_1').
\]
The functional $\phi_\spp'$ defines a character $P'(1)$. 
Let $B_{\phi'}$ denote the subgroup of $L'_{\phi'}$ consisting of upper triangular matrices.
Let $J'$ be the subgroup $P'$ defined by 
$J':=B_{\phi'} \ltimes P'(1)$. 
\begin{lem} \label{l:numericalJ}
We have  $\dim(G(\cO)/J')= \dim(B)$. 
\end{lem} 
 
\begin{proof} Observe that 
\[
\dim(G(O)/J')=\dim(G(O)/P'(1))-\dim (B_{\phi'})=\dim (B)+\dim (U_{L'})-\dim (B_{\phi'}), 
\]
where $U_{L'}$ is the subgroup of $L'$ consisting of unipotent upper triangular matrices.
Thus, the lemma amounts to the statement  $\dim(B_{\phi'})=\dim(U_{L'})$. This follows immediately from the explicit description of $L'\simeq G_0'$ \eqref{eq:G_0'} and $L'_{\phi'}$ \eqref{eq:Lphi'}. 
\end{proof}



\begin{bibdiv}
\begin{biblist}


\bib{Abe}{article}
{
	AUTHOR = {Abe, Tomoyuki},
	TITLE = {Langlands correspondence for isocrystals and the existence of
		crystalline companions for curves},
	JOURNAL = {J. Amer. Math. Soc.},
	FJOURNAL = {Journal of the American Mathematical Society},
	VOLUME = {31},
	YEAR = {2018},
	NUMBER = {4},
	PAGES = {921--1057},
	ISSN = {0894-0347},
}


\bib{AndrewsAskeyRoy}{book}
{
	AUTHOR = {Andrews, G. E.},
	Author = {Askey, Richard},
	Author = {Roy, Ranjan},
	TITLE = {Special functions},
	SERIES = {Encyclopaedia of Mathematics and its Applications},
	VOLUME = {71},
	PUBLISHER = {Cambridge University Press, Cambridge},
	YEAR = {1999},
	PAGES = {xvi+664},
}



\bib{Arinkin}{article}
{
	AUTHOR = {Arinkin, D.},
	TITLE = {Orthogonality of natural sheaves on moduli stacks of {$\rm
			SL(2)$}-bundles with connections on {$\Bbb P^1$} minus 4
		points},
	JOURNAL = {Selecta Math. (N.S.)},
	FJOURNAL = {Selecta Mathematica. New Series},
	VOLUME = {7},
	YEAR = {2001},
	NUMBER = {2},
	PAGES = {213--239},
	ISSN = {1022-1824},
}


\bib{ArinkinOper}{article}
{
	Author={Arinkin, D.},
	title={Irreducible connections admit generic oper structures},
	year={2016},
	eprint={https://arxiv.org/abs/1602.08989},
	archivePrefix={arXiv},
	primaryClass={math.AG}
}


\bib{ArinkinFedorov}{article}
{
	AUTHOR = {Arinkin, D.},
	Author = {Fedorov, R.},
	TITLE = {An example of the {L}anglands correspondence for irregular
		rank two connections on {$\Bbb P^1$}},
	JOURNAL = {Adv. Math.},
	FJOURNAL = {Advances in Mathematics},
	VOLUME = {230},
	YEAR = {2012},
	NUMBER = {3},
	PAGES = {1078--1123},
	ISSN = {0001-8708},
}


\bib{ArinkinGaitsgory}{article}
{
	AUTHOR = {Arinkin, D.}
	Author = {Gaitsgory, D.},
	TITLE = {Singular support of coherent sheaves and the geometric
		{L}anglands conjecture},
	JOURNAL = {Selecta Math. (N.S.)},
	FJOURNAL = {Selecta Mathematica. New Series},
	VOLUME = {21},
	YEAR = {2015},
	NUMBER = {1},
	PAGES = {1--199},
}


\bib{AB}{article}
{
	AUTHOR = {Arkhipov, S.},
	Author = {Bezrukavnikov, R.},
	TITLE = {Perverse sheaves on affine flags and {L}anglands dual group},
	NOTE = {With an appendix by Bezrukavnikov and Ivan Mirkovi\'{c}},
	JOURNAL = {Israel J. Math.},
	FJOURNAL = {Israel Journal of Mathematics},
	VOLUME = {170},
	YEAR = {2009},
	PAGES = {135--183},
	ISSN = {0021-2172},
}


\bib{BKV}{article}
{
	AUTHOR = {Baraglia, D.},
	Author = {Kamgarpour, M.},
	Author = {Varma, R.},
	TITLE = {Complete integrability of the parahoric {H}itchin system},
	JOURNAL = {Int. Math. Res. Not. IMRN},
	FJOURNAL = {International Mathematics Research Notices. IMRN},
	YEAR = {2019},
	NUMBER = {21},
	PAGES = {6499--6528},
}


\bib{BL}{article}
{
	AUTHOR = {Beauville, A.},
	Author = {Laszlo, Y.},
	TITLE = {Un lemme de descente},
	JOURNAL = {C. R. Acad. Sci. Paris S\'{e}r. I Math.},
	FJOURNAL = {Comptes Rendus de l'Acad\'{e}mie des Sciences. S\'{e}rie I. Math\'{e}matique},
	VOLUME = {320},
	YEAR = {1995},
	NUMBER = {3},
	PAGES = {335--340},
}


\bib{BD}{article}
{
	AUTHOR = {Beilinson, A.},
	Author = {Drinfeld, V.},
	TITLE  = {Quantization of Hitchin's integrable system and Hecke eigensheaves},
	Note = {\url{https://www.math.uchicago.edu/\textasciitilde mitya/langlands/hitchin/BD-hitchin.pdf}},
	Year={1997},
}


\bib{BenZviNadler}{article}
{
	AUTHOR = {Ben-Zvi, D.},
	Author = {Nadler, D.},
	TITLE = {Betti geometric {L}anglands},
	BOOKTITLE = {Algebraic geometry: {S}alt {L}ake {C}ity 2015},
	SERIES = {Proc. Sympos. Pure Math.},
	VOLUME = {97},
	PAGES = {3--41},
	PUBLISHER = {Amer. Math. Soc., Providence, RI},
	YEAR = {2018},
}


\bib{BZ}{article}
{
	AUTHOR = {Bernstein, I. N.},
	Author = {Zelevinsky, A. V.},
	TITLE = {Induced representations of reductive {${\germ p}$}-adic
		groups. {I}},
	JOURNAL = {Ann. Sci. \'{E}cole Norm. Sup. (4)},
	FJOURNAL = {Annales Scientifiques de l'\'{E}cole Normale Sup\'{e}rieure. Quatri\`eme
		S\'{e}rie},
	VOLUME = {10},
	YEAR = {1977},
	NUMBER = {4},
	PAGES = {441--472},
}


\bib{BCM}{article} 
{
	AUTHOR = {Beukers, F.},
	Author = {Cohen, H.},
	Author = {Mellit, A.},
	TITLE = {Finite hypergeometric functions},
	JOURNAL = {Pure Appl. Math. Q.},
	FJOURNAL = {Pure and Applied Mathematics Quarterly},
	VOLUME = {11},
	YEAR = {2015},
	NUMBER = {4},
	PAGES = {559--589},
}


\bib{Bezrukavnikov}{article}
{
	AUTHOR = {Bezrukavnikov, R.},
	TITLE = {On two geometric realizations of an affine {H}ecke algebra},
	JOURNAL = {Publ. Math. Inst. Hautes \'{E}tudes Sci.},
	FJOURNAL = {Publications Math\'{e}matiques. Institut de Hautes \'{E}tudes
		Scientifiques},
	VOLUME = {123},
	YEAR = {2016},
	PAGES = {1--67},
	ISSN = {0073-8301},
}


\bib{DeBos}{article}
{	
	Author={Bos, Niels uit de},
	title={An explicit geometric Langlands correspondence for the projective line minus four points},
	year={2019},
	eprint={https://arxiv.org/abs/1906.03240},
	archivePrefix={arXiv},
	primaryClass={math.AG}
}

\bib{Chen}{article} 
{ 
	Author={Chen, T-H.}, 
	title = {Vinberg's $\theta$-groups and rigid connections}, 
	Journal={IMRN}, 
	Issue = {23},
	Year={2017},
	Pages = {7321--7343},
}

\bib{CortiGolyshev}{article} 
{
	AUTHOR = {Corti, A.},
	Author = {Golyshev, V.},
	TITLE = {Hypergeometric equations and weighted projective spaces},
	JOURNAL = {Sci. China Math.},
	FJOURNAL = {Science China. Mathematics},
	VOLUME = {54},
	YEAR = {2011},
	NUMBER = {8},
	PAGES = {1577--1590},
}
 
 
\bib{CGP}{article} 
{
	AUTHOR = {Chernousov, V.},
	Author = {Gille, P.},
	Author = {Pianzola, A.},
	TITLE = {Torsors over the punctured affine line},
	JOURNAL = {Amer. J. Math.},
	FJOURNAL = {American Journal of Mathematics},
	VOLUME = {134},
	YEAR = {2012},
	NUMBER = {6},
	PAGES = {1541--1583},
}


\bib{Deligne}{book}
{
	AUTHOR = {Deligne, P.},
	TITLE = {Cohomologie \'{e}tale},
	SERIES = {Lecture Notes in Mathematics},
	VOLUME = {569},
	NOTE = {S\'{e}minaire de g\'{e}om\'{e}trie alg\'{e}brique du Bois-Marie SGA
		$4\frac{1}{2}$},
	PUBLISHER = {Springer-Verlag, Berlin},
	YEAR = {1977},
	PAGES = {iv+312},
}


\bib{DembeleWadim}{article}
{
	Author={Dembélé, L.},
	Author={Panchishkin, A.},
	Author={Voight, J.},
	Author={Zudilin, W.},
	Title={Special hypergeometric motives and their $L$-functions: Asai recognition},
	year={2019},
	eprint={https://arxiv.org/abs/1906.07384},
	archivePrefix={arXiv},
	primaryClass={math.NT}
}


\bib{DP09}{incollection}
{
	AUTHOR = {Donagi, R.},
	Author = {Pantev, T.},
	TITLE = {Geometric {L}anglands and non-abelian {H}odge theory},
	BOOKTITLE = {Surveys in differential geometry. {V}ol. {XIII}. {G}eometry,
		analysis, and algebraic geometry: forty years of the {J}ournal
		of {D}ifferential {G}eometry},
	SERIES = {Surv. Differ. Geom.},
	VOLUME = {13},
	PAGES = {85--116},
	PUBLISHER = {Int. Press, Somerville, MA},
	YEAR = {2009},
}


\bib{DPHitchin}{article}
{
	AUTHOR = {Donagi, R.},
	Author = {Pantev, T.},
	TITLE = {Langlands duality for {H}itchin systems},
	JOURNAL = {Invent. Math.},
	FJOURNAL = {Inventiones Mathematicae},
	VOLUME = {189},
	YEAR = {2012},
	NUMBER = {3},
	PAGES = {653--735},
}


\bib{DP19}{article}
{   
	Author={Donagi, R.},
	Author={Pantev, T.},
	title={Parabolic Hecke eigensheaves},
	year={2019},
	eprint={https://arxiv.org/abs/1910.02357},
	archivePrefix={arXiv},
	primaryClass={math.AG}
}


\bib{Drinfeld83}{article}
{
	AUTHOR = {Drinfel\cprime d, V. G.},
	TITLE = {Two-dimensional {$l$}-adic     representations of the fundamental
		group of a curve over a finite field and automorphic forms on
		{${\rm GL}(2)$}},
	JOURNAL = {Amer. J. Math.},
	FJOURNAL = {American Journal of Mathematics},
	VOLUME = {105},
	YEAR = {1983},
	NUMBER = {1},
	PAGES = {85--114},
}


\bib{Drinfeld87}{article}
{
	AUTHOR = {Drinfel\cprime d, V. G.},
	TITLE = {Two-dimensional $\ell$-adic representations of the Galois group of a global field of
characteristic $p$ and automorphic forms on $\mathrm{GL}(2)$},
	JOURNAL = {J. Sov. Math.},
	VOLUME = {36},
	YEAR = {1987},
	PAGES = {93–105},
}


\bib{Faltings}{article}
{
	AUTHOR = {Faltings, G.},
	TITLE = {Stable {$G$}-bundles and projective connections},
	JOURNAL = {J. Algebraic Geom.},
	FJOURNAL = {Journal of Algebraic Geometry},
	VOLUME = {2},
	YEAR = {1993},
	NUMBER = {3},
	PAGES = {507--568},
}


\bib{Frenkel}{article}
{
	AUTHOR = {Frenkel, E.},
	TITLE = {Recent advances in the {L}anglands program},
	JOURNAL = {Bull. Amer. Math. Soc. (N.S.)},
	FJOURNAL = {American Mathematical Society. Bulletin. New Series},
	VOLUME = {41},
	YEAR = {2004},
	NUMBER = {2},
	PAGES = {151--184},
}



\bib{FrenkelRamified}{incollection}
{
	AUTHOR = {Frenkel, E.},
	TITLE = {Ramifications of the geometric {L}anglands program},
	BOOKTITLE = {Representation theory and complex analysis},
	SERIES = {Lecture Notes in Math.},
	VOLUME = {1931},
	PAGES = {51--135},
	PUBLISHER = {Springer, Berlin},
	YEAR = {2008},
}


\bib{FrenkelBenZvi}{book}
{
	AUTHOR = {Frenkel, E.},
	Author = {Ben-Zvi, D.},
	TITLE = {Vertex algebras and algebraic curves},
	SERIES = {Mathematical Surveys and Monographs},
	VOLUME = {88},
	EDITION = {Second},
	PUBLISHER = {American Mathematical Society, Providence, RI},
	YEAR = {2004},
	PAGES = {xiv+400},
}


\bib{FFL}{article}
{
	AUTHOR = {Feigin, B.},
	Author = {Frenkel, E.},
	Author = {Toledano Laredo, V.},
	TITLE = {Gaudin models with irregular singularities},
	JOURNAL = {Adv. Math.},
	FJOURNAL = {Advances in Mathematics},
	VOLUME = {223},
	YEAR = {2010},
	NUMBER = {3},
	PAGES = {873--948},
}


\bib{FrenkelGaitsgory}{incollection}
{
	AUTHOR = {Frenkel, E.},
	Author = {Gaitsgory, D.},
	TITLE = {Local geometric {L}anglands correspondence and affine
		{K}ac-{M}oody algebras},
	BOOKTITLE = {Algebraic geometry and number theory},
	SERIES = {Progr. Math.},
	VOLUME = {253},
	PAGES = {69--260},
	PUBLISHER = {Birkh\"{a}user Boston, Boston, MA},
	YEAR = {2006},
}


\bib{FG}{article} 
{
    AUTHOR = {Frenkel, E.}
    Author= {Gross, B.},
    TITLE = {A rigid irregular connection on the projective line},
    JOURNAL = {Ann. of Math. (2)},
    VOLUME = {170},
    YEAR = {2009},
    NUMBER = {3},
    PAGES = {1469--1512},
}


\bib{FGV}{article}
{
	AUTHOR = {Frenkel, E.},
	Author = {Gaitsgory, D.},
	Author = {Vilonen, K.},
	TITLE = {On the geometric {L}anglands conjecture},
	JOURNAL = {J. Amer. Math. Soc.},
	FJOURNAL = {Journal of the American Mathematical Society},
	VOLUME = {15},
	YEAR = {2002},
	NUMBER = {2},
	PAGES = {367--417},
}

\bib{FJ}{article}
{
	AUTHOR = {Fres\'{a}n, X.},
	Author = {Jossen, P.},
	TITLE = {Exponential motives},
	eprint = {http://javier.fresan.perso.math.cnrs.fr/expmot.pdf},
	YEAR = {2020},
}


\bib{Gaitsgory}{article}
{
	AUTHOR = {Gaitsgory, D.},
	TITLE = {On a vanishing conjecture appearing in the geometric {L}anglands correspondence},
	JOURNAL = {Ann. of Math. (2)},
	FJOURNAL = {Annals of Mathematics. Second Series},
	VOLUME = {160},
	YEAR = {2004},
	NUMBER = {2},
	PAGES = {617--682},
}

\bib{GaitsgorydeJong}{article}
{
	AUTHOR = {Gaitsgory, D.},
	TITLE = {On de {J}ong's conjecture},
	JOURNAL = {Israel J. Math.},
	FJOURNAL = {Israel Journal of Mathematics},
	VOLUME = {157},
	YEAR = {2007},
	PAGES = {155--191},
}


\bib{GaitsgoryRecent}{incollection}
{
	AUTHOR = {Gaitsgory, D.},
	TITLE = {Progr\`es r\'{e}cents dans la th\'{e}orie de {L}anglands g\'{e}om\'{e}trique},
	NOTE = {S\'{e}minaire Bourbaki. Vol. 2015/2016. Expos\'{e}s 1104--1119},
	JOURNAL = {Ast\'{e}risque},
	FJOURNAL = {Ast\'{e}risque},
	NUMBER = {390},
	YEAR = {2017},
	PAGES = {Exp. No. 1109, 139--168},
}


\bib{Ginzburg}{article}
{
	author={Ginzburg, V.},
	title={Perverse sheaves on a Loop group and Langlands' duality},
	year={1995},
	eprint={https://arxiv.org/pdf/alg-geom/9511007.pdf},
}

\bib{GorbounovSmirnov}{article}
{
	AUTHOR = {Gorbounov, V.},
	Author = {Smirnov, M.},
	TITLE = {Some remarks on {L}andau-{G}inzburg potentials for
		odd-dimensional quadrics},
	JOURNAL = {Glasg. Math. J.},
	FJOURNAL = {Glasgow Mathematical Journal},
	VOLUME = {57},
	YEAR = {2015},
	NUMBER = {3},
	PAGES = {481--507},
}


\bib{Greene}{article}
{
	AUTHOR = {Greene, J.},
	TITLE = {Hypergeometric functions over finite fields},
	JOURNAL = {Trans. Amer. Math. Soc.},
	FJOURNAL = {Transactions of the American Mathematical Society},
	VOLUME = {301},
	YEAR = {1987},
	NUMBER = {1},
	PAGES = {77--101},
}


\bib{Gross}{article}
{
	AUTHOR = {Gross, B.},
	TITLE = {Irreducible cuspidal representations with prescribed local
		behavior},
	JOURNAL = {Amer. J. Math.},
	FJOURNAL = {American Journal of Mathematics},
	VOLUME = {133},
	YEAR = {2011},
	NUMBER = {5},
	PAGES = {1231--1258},
}

\bib{GrossReeder}{article}
  {
   Author = {Gross, B.},
    AUTHOR = {Reeder, M.},
          TITLE = {Arithmetic invariants of discrete Langlands parameters},
   JOURNAL = {Duke Math. J.},
      YEAR = {2010},
    NUMBER = {154},
     PAGES = {431--508},

}



\bib{WittenGukov}{incollection}
{
	AUTHOR = {Gukov, S.},
	Author = {Witten, E.},
	TITLE = {Gauge theory, ramification, and the geometric {L}anglands
		program},
	BOOKTITLE = {Current Developments in Mathematics, 2006},
	PAGES = {35--180},
	PUBLISHER = {Int. Press, Somerville, MA},
	YEAR = {2008},
}


\bib{Guignard}{article}
{
	AUTHOR = {Guignard, Q.},
	TITLE = {On the ramified class field theory of relative curves},
	JOURNAL = {Algebra Number Theory},
	FJOURNAL = {Algebra \& Number Theory},
	VOLUME = {13},
	YEAR = {2019},
	NUMBER = {6},
	PAGES = {1299--1326},
}


\bib{PappasRapoport}{article}
{
	AUTHOR = {Haines, T.},
	Author = {Rapoport, M.},
	TITLE = {On parahoric subgroups},
	NOTE = {Appendix to Pappas and Rapoport's \emph{Twisted loop groups and their affine flag varieties}},
	JOURNAL = {Adv. Math.},
	FJOURNAL = {Advances in Mathematics},
	VOLUME = {219},
	YEAR = {2008},
	NUMBER = {1},
	PAGES = {188--198},
}


\bib{Heinloth}{article}
{
	AUTHOR = {Heinloth, J.},
	TITLE = {Coherent sheaves with parabolic structure and construction of
		{H}ecke eigensheaves for some ramified local systems},
	JOURNAL = {Ann. Inst. Fourier (Grenoble)},
	FJOURNAL = {Universit\'{e} de Grenoble. Annales de l'Institut Fourier},
	VOLUME = {54},
	YEAR = {2004},
	NUMBER = {7},
	PAGES = {2235--2325 (2005)},
}

\bib{HeinlothUniformisation}{article} 
{
AUTHOR = {Heinloth, J.},
	TITLE = {Uniformization of $\cG$-bundles},
	JOURNAL = {Math. Ann.},
	VOLUME = {347},
	YEAR = {2010},
	PAGES = {499--528},
}

\bib{HNY}{article} 
{
    Author={Heinloth, J.},
    Author={Ng\^{o}, B. C.},
    Author={Yun, Z.},
    Title={Kloosterman sheaves for reductive groups}, 
    Year={2013}, 
    Journal={Ann. of Math. (2)},
    Volume={177},
    Pages={241--310},
}


\bib{Horja}{article}
{
	AUTHOR = {Horja, R. P.},
	TITLE  = {Hypergeometric functions and mirror    
		symmetry in toric varieties},
	NOTE = {Thesis (Ph.D.)--Duke University},
	eprint={https://arxiv.org/abs/math/9912109},
	YEAR = {1999},
	PAGES = {120},
}


\bib{JY}{article}
{
Author={Jakob, K.}, 
Author={Yun, Z.}, 
Title={Quasi-epipelagic representations and rigid automorphic data}, 
Year={2020},
Journal={In preparation},
}

\bib{Masoud}{article}
{
	AUTHOR = {Kamgarpour, M.},
	TITLE = {Stacky abelianization of algebraic groups},
	JOURNAL = {Transform. Groups},
	FJOURNAL = {Transformation Groups},
	VOLUME = {14},
	YEAR = {2009},
	NUMBER = {4},
	PAGES = {825--846},
}





\bib{KS2}{article}
{
	AUTHOR = {Kamgarpour, M.}
	Author = {Schedler, T.},
	TITLE = {Ramified {S}atake isomorphisms for strongly parabolic
		characters},
	JOURNAL = {Doc. Math.},
	FJOURNAL = {Documenta Mathematica},
	VOLUME = {18},
	YEAR = {2013},
	PAGES = {1275--1300},
}

\bib{KS}{article}
{
	AUTHOR = {Kamgarpour, M.},
	Author = {Schedler, T.},
	TITLE = {Geometrization of principal series representations of
		reductive groups},
	JOURNAL = {Ann. Inst. Fourier (Grenoble)},
	FJOURNAL = {Universit\'{e} de Grenoble. Annales de l'Institut Fourier},
	VOLUME = {65},
	YEAR = {2015},
	NUMBER = {5},
	PAGES = {2273--2330},
}


\bib{KamgarpourSage0}{article} 
{
	Author={Kamgarpour, M.},
	Author = {Sage, D. S.},
	TITLE = {A geometric analogue of a conjecture of {G}ross and {R}eeder},
	JOURNAL = {Amer. J. Math.},
	FJOURNAL = {American Journal of Mathematics},
	VOLUME = {141},
	YEAR = {2019},
	NUMBER = {5},
	PAGES = {1457--1476},
}



\bib{KamgarpourSage}{article} 
{
	Author={Kamgarpour, M.},
	Author = {Sage, D. S.},
	title={Rigid connections on $\mathbb{P}^1$ via the Bruhat-Tits building},
	year={2019},
	Journal={to appear in Proceedings of the London Maths. Society},
	eprint={https://arxiv.org/abs/1910.00181},
} 




\bib{WittenKapustin}{article}
{
	AUTHOR = {Kapustin, A.},
	Author = {Witten, E.},
	TITLE = {Electric-magnetic duality and the geometric {L}anglands
		program},
	JOURNAL = {Commun. Number Theory Phys.},
	FJOURNAL = {Communications in Number Theory and Physics},
	VOLUME = {1},
	YEAR = {2007},
	NUMBER = {1},
	PAGES = {1--236},
}


\bib{KatzKloosterman}{book}
{
	AUTHOR = {Katz, N. M.},
	TITLE = {Gauss sums, {K}loosterman sums, and monodromy groups},
	SERIES = {Annals of Mathematics Studies},
	VOLUME = {116},
	PUBLISHER = {Princeton University Press, Princeton, NJ},
	YEAR = {1988},
	PAGES = {x+246},
}


\bib{KatzBook}{book}
{
	AUTHOR = {Katz, N. M.},
	TITLE = {Exponential sums and differential equations},
	SERIES = {Annals of Mathematics Studies},
	VOLUME = {124},
	PUBLISHER = {Princeton University Press, Princeton, NJ},
	YEAR = {1990},
	PAGES = {xii+430},
}


\bib{KatzRigid}{book}
{
	AUTHOR = {Katz, N. M.},
	TITLE = {Rigid local systems},
	SERIES = {Annals of Mathematics Studies},
	VOLUME = {139},
	PUBLISHER = {Princeton University Press, Princeton, NJ},
	YEAR = {1996},
	PAGES = {viii+223},
}


\bib{Kedlaya}{article}
{
	Author={Kedlaya, K. S. },
	Title={Etale and crystalline companions, I},
	year={2018},
	eprint={https://arxiv.org/abs/1811.00204},
	archivePrefix={arXiv},
	primaryClass={math.NT}
}


\bib{Koornwinder}{incollection}
{
	AUTHOR = {Koornwinder, Tom H.},
	TITLE = {Jacobi functions and analysis on noncompact semisimple {L}ie
		groups},
	BOOKTITLE = {Special functions: group theoretical aspects and applications},
	SERIES = {Math. Appl.},
	PAGES = {1--85},
	PUBLISHER = {Reidel, Dordrecht},
	YEAR = {1984},
}


\bib{Lafforgue}{article}
{
	AUTHOR = {Lafforgue, L.},
	TITLE = {Chtoucas de {D}rinfeld et correspondance de {L}anglands},
	JOURNAL = {Invent. Math.},
	FJOURNAL = {Inventiones Mathematicae},
	VOLUME = {147},
	YEAR = {2002},
	NUMBER = {1},
	PAGES = {1--241},
}


\bib{LaumonLanglands}{article}
{
	Author = {Laumon, G.},
	fjournal = {Duke Mathematical Journal},
	journal = {Duke Math. J.},
	number = {2},
	pages = {309--359},
	publisher = {Duke University Press},
	title = {Correspondance de Langlands géométrique pour les corps de fonctions},
	volume = {54},
	year = {1987},
}


\bib{LRS}{article}
{
	Author = {Laumon, G.},
	Author={Rapoport, M.}, 
	Author={Stuhler, U.},
	journal = {Invention. Math.},
	pages = {217--338},
	title = {D-elliptic sheaves and the Langlands correspondence},
	volume = {113},
	year = {1993},
}


\bib{LT}{article}
{
	Author = {Lam, T.},
	Author = {Templier, N.},
	title = {The mirror conjecture for minuscule flag varieties}	
	month = {05},
	eprint={https://arxiv.org/abs/1705.00758},
	year = {2017},
}


\bib{MirkovicVilonen}{article}
{
	AUTHOR = {Mirkovi\'{c}, I.},
	Author = {Vilonen, K.},
	TITLE = {Geometric {L}anglands duality and representations of algebraic
		groups over commutative rings},
	JOURNAL = {Ann. of Math. (2)},
	FJOURNAL = {Annals of Mathematics. Second Series},
	VOLUME = {166},
	YEAR = {2007},
	NUMBER = {1},
	PAGES = {95--143},
	ISSN = {0003-486X},
}


\bib{Miyatani}{article}
{
	Author={Miyatani, K.},
	Title={$p$-adic generalized hypergeometric equations from the viewpoint of arithmetic D-modules},
	year={2016},
	eprint={https://arxiv.org/abs/1607.04852},
	archivePrefix={arXiv},
	primaryClass={math.AG}
}


\bib{NadlerYun}{article}
{
	AUTHOR = {Nadler, D.},
	Author = {Yun, Z.},
	TITLE = {Geometric {L}anglands correspondence for {$\rm SL(2)$}, {$\rm
			PGL(2)$} over the pair of pants},
	JOURNAL = {Compos. Math.},
	FJOURNAL = {Compositio Mathematica},
	VOLUME = {155},
	YEAR = {2019},
	NUMBER = {2},
	PAGES = {324--371},
	ISSN = {0010-437X},
}


\bib{Ono}{article}
{
	AUTHOR = {Ono, K.},
	TITLE = {Values of {G}aussian hypergeometric series},
	JOURNAL = {Trans. Amer. Math. Soc.},
	FJOURNAL = {Transactions of the American Mathematical Society},
	VOLUME = {350},
	YEAR = {1998},
	NUMBER = {3},
	PAGES = {1205--1223},
}


\bib{Opdam}{book}
{
	AUTHOR = {Opdam, Eric M.},
	TITLE = {Lecture notes on {D}unkl operators for real and complex
		reflection groups},
	SERIES = {MSJ Memoirs},
	VOLUME = {8},
	NOTE = {With a preface by Toshio Oshima},
	PUBLISHER = {Mathematical Society of Japan, Tokyo},
	YEAR = {2000},
	PAGES = {viii+90},
	ISBN = {4-931469-08-6},
}


\bib{PRW}{article}
{
	AUTHOR={Pech, C.},
	Author={Rietsch, K.},
	Author={Williams, L.},
	TITLE = {On {L}andau-{G}inzburg models for quadrics and flat sections
		of {D}ubrovin connections},
	JOURNAL = {Adv. Math.},
	FJOURNAL = {Advances in Mathematics},
	VOLUME = {300},
	YEAR = {2016},
	PAGES = {275--319},
}


\bib{RLYG}{article}
{
	AUTHOR= {Reeder, M.},
	Author={Levy, P.},
	Author={Yu, J.},
	Author={Gross, B.},
	TITLE = {Gradings of positive rank on simple {L}ie algebras},
	JOURNAL = {Transform. Groups},
	FJOURNAL = {Transformation Groups},
	VOLUME = {17},
	YEAR = {2012},
	NUMBER = {4},
	PAGES = {1123--1190},	
}


\bib{RY}{article}
{
	Author = {Reeder, M.},
	Author = {Yu, J.},
	TITLE = {Epipelagic representations and invariant theory},
	JOURNAL = {J. Amer. Math. Soc.},
	FJOURNAL = {Journal of the American Mathematical Society},
	VOLUME = {27},
	YEAR = {2014},
	NUMBER = {2},
	PAGES = {437--477},
}


\bib{TemplierSawin}{article}
{	
	Author={Sawin, W.},
	Author={Templier, N.},
	title={On the Ramanujan conjecture for automorphic forms over function fields I. Geometry},
	year={2018},
	eprint={https://arxiv.org/abs/1805.12231},
}


\bib{Serre}{book}
{
	AUTHOR = {Serre, J.-P.},
	TITLE = {Algebraic groups and class fields},
	SERIES = {Graduate Texts in Mathematics},
	VOLUME = {117},
	NOTE = {Translated from the French},
	PUBLISHER = {Springer-Verlag, New York},
	YEAR = {1988},
	PAGES = {x+207},
}


\bib{Slater}{book}
{
	AUTHOR = {Slater, L.},
	TITLE = {Generalized hypergeometric functions},
	PUBLISHER = {Cambridge University Press, Cambridge},
	YEAR = {1966},
	PAGES = {xiii+273},
}


\bib{Varchenko}{book}
{
	AUTHOR = {Varchenko, A.},
	TITLE = {Special functions, {KZ} type equations, and representation
		theory},
	SERIES = {CBMS Regional Conference Series in Mathematics},
	VOLUME = {98},
	PUBLISHER = {Published for the Conference Board of the Mathematical
		Sciences, Washington, DC; by the American Mathematical
		Society, Providence, RI},
	YEAR = {2003},
	PAGES = {viii+118},
}


\bib{Vinberg}{article} 
{ 
	AUTHOR = {Vinberg, E. B.},
	TITLE = {The {W}eyl group of a graded {L}ie algebra},
	JOURNAL = {Izv. Akad. Nauk SSSR Ser. Mat.},
	FJOURNAL = {Izvestiya Akademii Nauk SSSR. Seriya Matematicheskaya},
	VOLUME = {40},
	YEAR = {1976},
	NUMBER = {3},
	PAGES = {488--526, 709},
}


\bib{Witten}{article}
{
	AUTHOR = {Witten, E.},
	TITLE = {Gauge theory and wild ramification},
	JOURNAL = {Anal. Appl. (Singap.)},
	FJOURNAL = {Analysis and Applications},
	VOLUME = {6},
	YEAR = {2008},
	NUMBER = {4},
	PAGES = {429--501},
	ISSN = {0219-5305},
}


\bib{XuZhu}{article}
{
	Author = {Xu, D.},
	Author = {Zhu, X.},
	title  = {Bessel $F$-isocrystals for reductive groups},
	year={2019},
	eprint={https://arxiv.org/abs/1910.13391},
}


\bib{Yu}{article}
{
	AUTHOR = {Yu, J.},
	TITLE = {Smooth models associated to concave functions in {B}ruhat-{T}its theory},
	BOOKTITLE = {Autour des sch\'{e}mas en groupes. {V}ol. {III}},
	SERIES = {Panor. Synth\`eses},
	VOLUME = {47},
	PAGES = {227--258},
	PUBLISHER = {Soc. Math. France, Paris},
	YEAR = {2015},
}


\bib{YunGalois}{article}
{
	AUTHOR = {Yun, Z.},
	Title={Motives with exceptional Galois groups and the inverse Galois problem},
	Journal={Invent. Math.}
	Year={2014},
	Volume={196},
	Pages={267--337}
}



\bib{YunEvans}{article}
{
	AUTHOR = {Yun, Z.},
	TITLE = {Galois representations attached to moments of {K}loosterman
		sums and conjectures of {E}vans},
	NOTE = {Appendix B by Christelle Vincent},
	JOURNAL = {Compos. Math.},
	FJOURNAL = {Compositio Mathematica},
	VOLUME = {151},
	YEAR = {2015},
	NUMBER = {1},
	PAGES = {68--120},
}


\bib{YunCDM}{incollection}
{
	AUTHOR = {Yun, Z.},
	TITLE = {Rigidity in automorphic representations and local systems},
	BOOKTITLE = {Current developments in mathematics 2013},
	PAGES = {73--168},
	PUBLISHER = {Int. Press, Somerville, MA},
	YEAR = {2014},
}


\bib{YunEpipelagic}{article}
{
	AUTHOR = {Yun, Z.},
	Title={Epipelagic representations and rigid local systems},
	Journal={Selecta Math. (N.S.)}, 
	Year={2016},
	pages={1195--1243},
	Volume={22}, 
}


\bib{ZagierDifferentialEquation}{incollection}
{
	AUTHOR = {Zagier, D.},
	TITLE = {The arithmetic and topology of differential equations},
	BOOKTITLE = {European {C}ongress of {M}athematics},
	PAGES = {717--776},
	PUBLISHER = {Eur. Math. Soc., Z\"{u}rich},
	YEAR = {2018},
}


\bib{ZagierZinger}{incollection}
{
	AUTHOR = {Zagier, D.},
	Author = {Zinger, A.},
	TITLE = {Some properties of hypergeometric series associated with
		mirror symmetry},
	BOOKTITLE = {Modular forms and string duality},
	SERIES = {Fields Inst. Commun.},
	VOLUME = {54},
	PAGES = {163--177},
	PUBLISHER = {Amer. Math. Soc., Providence, RI},
	YEAR = {2008},
}


\bib{Zhu}{article}
{
	AUTHOR = {Zhu, X.},
	TITLE = {Frenkel-{G}ross' irregular connection and
		{H}einloth-{N}g\^{o}-{Y}un's are the same},
	JOURNAL = {Selecta Math. (N.S.)},
	FJOURNAL = {Selecta Mathematica. New Series},
	VOLUME = {23},
	YEAR = {2017},
	NUMBER = {1},
	PAGES = {245--274},
}


\bib{Wadim}{article}
{
	AUTHOR = {Zudilin, W.},
	TITLE = {Hypergeometric heritage of {W}. {N}. {B}ailey},
	NOTE = {With an appendix containing letters from Bailey to Freeman
		Dyson},
	JOURNAL = {Notices of the International Congress of Chinese
		Mathematicians},
	VOLUME = {7},
	YEAR = {2019},
	NUMBER = {2},
	PAGES = {32--46},
}

\end{biblist}
\end{bibdiv}
\end{document}